\documentclass[reqno,twoside,11pt,english]{amsart}
\usepackage{amsmath,amsfonts,amssymb,amsthm,epsfig}
\usepackage{graphicx}
\usepackage{esint}

\usepackage[colorlinks,citecolor=blue,linkcolor=blue]{hyperref}
\usepackage{color}
\usepackage[colorlinks,citecolor=blue,linkcolor=blue]{hyperref}

\usepackage{graphicx}
\usepackage{titletoc,epsf}
\usepackage{amsmath,amsfonts,latexsym,amsthm,amsxtra,amssymb,bbm}
\allowdisplaybreaks
\usepackage{graphicx,float}

\allowdisplaybreaks

\let\pa=\partial

\def\ov{\overline}
\def\Re{\mathrm{Re} }
\def\Im{\mathrm{Im}}
\newcommand{\dd}{\mathrm{d}}

\def\Xint#1{\mathchoice
	{\XXint\displaystyle\textstyle{#1}}	{\XXint\textstyle\scriptstyle{#1}}	{\XXint\scriptstyle\scriptscriptstyle{#1}}	{\XXint\scriptscriptstyle\scriptscriptstyle{#1}}	\!\int}
\def\XXint#1#2#3{{\setbox0=\hbox{$#1{#2#3}{\int}$ }
		\vcenter{\hbox{$#2#3$ }}\kern-.6\wd0}}
\def\dashint{\Xint-}

\newtheorem{theorem}{Theorem}[section]
\newtheorem{lemma}[theorem]{Lemma}
\newtheorem{proposition}[theorem]{Proposition}
\newtheorem{corollary}[theorem]{Corollary}
\newtheorem{definition}[theorem]{Definition}

\newtheorem{thmA}{Theorem}[section]

\newtheorem{assumption}[thmA]{Assumption}

\newtheorem{remark}{Remark}

\numberwithin{equation}{section}

\numberwithin{equation}{section}

\topmargin - .23 in

\begin{document}

\title[Long time dynamics for CGL]{Long time dynamics for complex Ginzburg-Landau equations on three-dimensional Riemannian manifolds}

\author{Xing Cheng,~~~~~~Zeyu Lyu,~~~~~~Chenmin Sun,~~~~~~And~~~~~~Yunrui Zheng}

\address[X. Cheng]{ School of mathematics, Hohai University, Nanjing 210098, Jiangsu,  China and Laboratory of Mathematical Modeling and Intelligent Computing for Water Systems, Hohai University, Nanjing, 211100, China. }
\email{{\tt chengx@hhu.edu.cn}}

\address[Z. Lyu]{ School of Mathematics and Key Laboratory of Mathematical MIIT, Nanjing University of Aeronautics and Astronautics, 210016, Nanjing, China.}
\email{{\tt zeyulyu@nuaa.edu.cn}}

\address[C. Sun]{  CNRS, Universit\'e Paris-Est Cr\'eteil, Laboratoire d'Analyse et de Math\'ematiques Appliqu\'ees, UMR CNRS 8050, Cr\'eteil, France}
\email{{\tt chenmin.sun@cnrs.fr}}

\address[Y. Zheng]{ School of mathematics, Shandong University, Shandong 250100, Jinan,   China.}
\email{{\tt yunrui\_zheng@sdu.edu.cn}}

\thanks{X. Cheng has been supported by the NSF of Jiangsu Province (No.~BK20230036), the NSF of China (Grant U2340221) and National Key R\&D Program of China (2022YFC3202602).
Z. Lyu has been supported by the NSFC (No. 125001296), the NSF of Jiangsu Province (No. BK20251386) and the Postdoctoral Fellowship Program of CPSF (No. GZB20250975).
C. Sun is partially supported by the ANR project Smooth ANR-22-CE40-0017.
  Y. Zheng is partially supported by the NSF of Shandong Province (ZR2026MS0070) and the NSFC (No. 11901350 and No. 12371172).  The authors acknowledge support of the Institut Henri Poincar\'e (UAR 839 CNRS-Sorbonne Universit\'e) and LabEx CARMIN (ANR-10-LABX-59-01).}

\begin{abstract}

We study the energy-subcritical and energy-critical complex Ginzburg-Landau equations on three-dimensional Riemannian manifolds.
We establish the asymptotic behavior of solutions via a purely energy-based approach and subsequently prove a fast/slow dichotomy of the solutions. Moreover, for the slow solution, we show that the solution converges to an explicit asymptotic profile.
In particular, we obtain a precise description of both the amplitude and the phase of the leading-order dynamics.
As a by-product, the complex Ginzburg-Landau equation approximation yields global weak solutions to the energy-subcritical and energy-critical defocusing nonlinear Schr\"odinger equations on compact three-dimensional Riemannian manifolds.
We also investigate the long time behavior of the solutions to the energy-subcritical and energy-critical complex Ginzburg-Landau equations on a class of complete noncompact three-dimensional Riemannian manifolds and prove decay of global solutions under suitable geometric assumptions.

\bigskip


\bigskip

\end{abstract}

\maketitle

\section{Introduction}

The complex Ginzburg-Landau equation (CGL) plays a fundamental role in the mathematical description of dissipative nonlinear phenomena and arises in a broad range of physical contexts, including superconductivity, fluid dynamics, nonlinear optics, and pattern formation; see, for instance, \cite{GL,MYLYYZ}.
Owing to its mixed dispersive-dissipative nature, the CGL equation serves as a bridge between nonlinear Schr\"odinger equations and semilinear heat equations and has therefore attracted considerable attention in both mathematics and physics.

Throughout this paper, a Riemannian manifold is always understood to be smooth, connected, and complete, with completeness taken with respect to the Riemannian distance. Unless explicitly stated otherwise, manifolds are assumed to have no boundary. We write $\left(\mathcal{M}^3,g\right)$ for a three-dimensional Riemannian manifold and $\left(\mathcal{M}^d,g\right)$ for a $d$-dimensional Riemannian manifold. Compactness or noncompactness, as well as the presence of boundary, will be specified when relevant. These standing assumptions will not be repeated in subsequent statements. In the three-dimensional setting, we consider the complex Ginzburg-Landau equation on $\left(\mathcal{M}^3,g\right)$, namely,
\begin{align}\label{eq1.CGL}
\left\{
\begin{aligned}
 &\partial_t u - z \Delta_g u = - z  f(u ),  \\
 	& u(0, x) =  u_0 (x),
\end{aligned}
\right.
\end{align}
where $z\in \mathbb{C}$ satisfies $\Re{z}>0$,
$u: \mathbb{R}_+ \times \mathcal{M}^3 \to \mathbb{C}$ is the unknown function, and $\Delta_g$ denotes the Laplace-Beltrami operator on $\mathcal{M}^3$.
The power-type nonlinearity is given by
$f(u) = |u|^{p-1} u$, with $1 < p \le 5$.
The associated energy functional is
\begin{align}\label{u_energy_func}
	E(u(t)) = \frac12 \int_{\mathcal{M}^3} |\nabla_g u(t)|^2  \,\mathrm{d}\mu + \frac{1}{p+1} \int_{\mathcal{M}^3} |u(t)|^{p+1} \, \mathrm{d}\mu,
\end{align}
where $\mathrm{d}\mu$ stands for the Riemannian volume measure. The functional $E(u(t))$ constitutes a Lyapunov functional associated with equation \eqref{eq1.CGL}.

The global well-posedness and blowup phenomena for the CGL on Euclidean spaces have been extensively investigated over the past two decades. In \cite{HW}, Huang-Wang established the global well-posedness of the energy-critical CGL equations via the induction-on-energy method and further explored the inviscid limit behavior of this model. We also refer to \cite{WGZ} for  global well-posedness of both energy-subcritical and energy-critical derivative-type CGL equations. Notably, the CGL equations studied in \cite{HW,WGZ} are essentially of the defocusing type. As for the focusing counterpart, Cheng-Guo-Zheng \cite{CGZ} proved the global well-posedness of the energy-critical CGL equations by employing the concentration-compactness/rigidity theorem method originally developed in \cite{KK11,KM}, which resolves an open problem proposed by Cazenave \cite{Ca12}.
Furthermore, via an inviscid limit analysis, the authors in \cite{CGZ} also verified the existence of global weak solutions to the energy-critical focusing nonlinear Schr\"odinger equation, under the condition that both the total energy and kinetic energy of the initial data are strictly smaller than those of the corresponding ground state. This achievement settles an open problem proposed in \cite{Ca03}. Subsequently, Cheng-Guo-Zheng-Zheng \cite{CGZZ} extended the global well-posedness theory of energy-critical CGL equations to exterior domains. In \cite{CDW},  Cazenave-Dickstein-Weissler established the finite-time blowup of negative-energy CGL solutions based on a virial-type argument. Meanwhile, the rigorous construction of finite-time blowup solutions for the CGL was systematically investigated by Duong-Nouaili-Zaag in \cite{DNZ1,DNZ0,DNZ2}; see also \cite{MZ0,NZ}.

In contrast to the substantial literature on well-posedness and blowup,
much less is known about the long time dynamics of CGL equations on curved geometries.  Although several recent works have addressed related problems on Riemannian manifolds \cite{CMZZ,DMNZ,MNZ*}, a quantitative description of the asymptotic behavior of global solutions on compact three-dimensional Riemannian manifolds is still unavailable in the existing literature.
The primary objective of the present paper is to fill this gap. We establish precise asymptotic profiles for global solutions on compact three-dimensional Riemannian manifolds and identify the leading-order dynamics governing both the amplitude and the phase.

Our first result provides a dichotomy of the asymptotic dynamics of global solutions on compact three-dimensional Riemannian manifolds.
The main mechanism underlying Theorem \ref{th2v79} is the separation between the averaged mode and the fluctuation component.
The dissipative structure of the equation suppresses spatial oscillations, while the averaged mode survives on much longer time scales.
Therefore, the infinite-dimensional dynamics generated by the complex Ginzburg-Landau equation reduce asymptotically to an effective one-dimensional evolution equation governing the spatial average.

\begin{theorem}[Long time behavior and asymptotics of the CGL equation on $\mathcal{M}^3$ without boundary]\label{th2v79}

Let $ \left(\mathcal{M}^3, g \right)$ be compact,
and let $u_0 \in H_{x}^1 \left(\mathcal{M}^3 \right)$. For every $1 < p \le 5$, equation \eqref{eq1.CGL} admits a unique global solution
\[
u \in C_t^0 H_x^1 \left(\mathbb{R}_+ \times \mathcal{M}^3 \right),
\]
which satisfies the following decay estimates: for all sufficiently large $t$,
\begin{align*}
&  \left\|\nabla_{g}^{k} u(t) \right\|_{L_x^2 \left(\mathcal{M}^3 \right)}
\lesssim_{\|u_0\|_{H^k}} t^{- \frac{k}2-\frac{1}{p-1}}, \qquad k \in \{0,1\},
\intertext{ and }
& \left\| u(t) \right\|_{L_x^{\infty} \left(\mathcal{M}^3 \right)} \lesssim_{\|u_0\|_{H^1} } t^{-\frac{1}{p-1}}.
\end{align*}
Furthermore, decomposing $u$ into its spatial mean and fluctuation as
\begin{align}\label{eq4.1v127}
u(t,x) = m(t) + v(t,x), \qquad
m(t) := \dashint_{\mathcal{M}^3} u(t,x) \, \mathrm{d}\mu
= \frac{1}{\mathrm{Vol}_g \left(\mathcal{M}^3 \right)} \int_{\mathcal{M}^3} u(t,x) \, \mathrm{d}\mu,
\end{align}
we have the following asymptotic characterization.
\begin{enumerate}
  \item[(i)] \textbf{Mean component.}
  The spatial average $m(t)$ satisfies
  \[
  |m(t)| \le C t^{-\frac{1}{p-1}},
  \]
  where $C > 0$ is a constant depending only on $p$, $ \Re z$, and the geometry of $\mathcal{M}^3$.

  \item[(ii)] \textbf{Fluctuation component.}
  The remainder $v(t)$ decays exponentially in $H_x^1$,   namely,
  \begin{align}\label{eq1.2v6}
  \|v(t)\|_{H_x^1} \le C e^{-c t},
  \end{align}
  where $C$ and $ c $ are positive constants which depend only on $p$, $z$, $\|u_0 \|_{H_{x}^1}$, and the geometry of $\mathcal{M}^3$.

  \item[(iii)] \textbf{Dichotomy between fast and slow solutions.}
  Exactly one of the following alternatives holds for the solution $u$:

  $\mathrm{(a)}$ \textbf{Fast solution.}
  There exist constants $C,c>0$ such that
  \[
  \|u(t)\|_{H_x^1 \left(\mathcal M^3 \right)}\le Ce^{-ct},  \ \forall\,  t\ge0.
  \]
  $\mathrm{(b)}$ \textbf{Slow solution.}
  There exist a constant $\theta_\infty = \theta_\infty(u_0) \in \mathbb{R}$
  and real-valued functions $\rho(t)>0$ and $\theta(t)$,
  defined for all sufficiently large $t$, such that
  \[
  u(t,x) = \rho(t) e^{i\theta(t)} + v(t,x),
  \]
  where $v(t)$ satisfies the exponential decay estimate \eqref{eq1.2v6}, and
  \[
  \rho(t) = \rho_c(t) \left( 1 + O_{u_0} \left(t^{-1} \right) \right),  \quad
  \theta(t) = \theta_c(t) + \theta_\infty + o_{u_0 } (1),  \text{ as }  t \to +\infty,
  \]
  with
  \[
  \rho_c(t) = \left( \frac{1}{ \Re z (p-1)t} \right)^{\frac{1}{p-1}}, \quad
  \theta_c(t) = -\frac{\Im z}{(p-1)\Re z} \ln t.
  \]
  Here $\rho(t)$ and $\theta(t)$ are the actual time-dependent amplitude and phase of the dominant background component, characterizing the oscillatory factor $e^{i\theta(t)}$ in the decomposition of $u(t,x)$, while their leading-order counterparts $\rho_c(t)$ and $\theta_c(t)$ capture the dominant decay rate and phase modulation of the solution at infinity.
  In the case $\Im z = 0$, the leading logarithmic phase is absent, and the phase converges to a constant.

\end{enumerate}

\end{theorem}
In the following, we will discuss some further extensions.
We will give four remarks concerning extensions of the above result to the CGL on compact Riemannian manifolds with boundary, the CGL on some more general three-dimensional Riemannian manifolds, the CGL on higher-dimensional Riemannian manifolds,
and the semilinear heat equation (NLH) on the corresponding Riemannian manifolds.

\begin{remark}[Asymptotic behavior of the CGL equation on $\mathcal{M}^3 $ with boundary $\partial \mathcal{M}^3 $]\label{re2v118}

Let $\left(\mathcal M^3,g\right)$ be compact with smooth nonempty boundary. For the CGL equation on $\mathcal M^3$,
two important types of boundary conditions are considered: Dirichlet and Neumann conditions. As a result, we have
 the Dirichlet initial-boundary value problem
\begin{align*}
\left\{
\begin{aligned}
 &\partial_t u - z \Delta_g u = - z  f(u) \ \text{in}\ \mathcal{M}^3,\\
 	& u|_{t=0} = u_0, \\
 &  u = 0 \ \text{on}\ \partial \mathcal{M}^3,
\end{aligned}
\right.
\end{align*}
and the Neumann initial-boundary problem
\begin{align*}
\left\{
\begin{aligned}
 &\partial_t u - z \Delta_g u = - z  f(u)\ \text{in}\ \mathcal{M}^3,\\
 	& u|_{t=0} = u_0, \\
 & \frac{\partial u}{\partial \nu } =0\ \text{on}\ \partial \mathcal{M}^3,
\end{aligned}
\right.
\end{align*}
where $\nu$ is the outer normal direction along $\partial \mathcal{M}^3 $.
The initial space is $H_0^1 \left(\mathcal M^3 \right)$ for Dirichlet conditions and
$H^1 \left(\mathcal M^3 \right)$ for Neumann conditions.

Global well-posedness for the CGL equations on compact three-dimensional Riemannian manifolds with boundary under both Dirichlet and Neumann boundary conditions can be established.
Concerning the Neumann case, we can derive asymptotic behavior for solutions consistent with the conclusions of Theorem \ref{th2v79}.
For the Dirichlet case,
the mass dissipation identity \eqref{eq_ener} and Proposition~\ref{pr2.3v122} give $ \| u(t) \|_{L^2} \le e^{- \Re z \lambda_1t} \| u_0 \|_{L^2}$,
 where $\lambda_1>0$ is the first Dirichlet eigenvalue of $-\Delta_g$.
Since $E$ is nonincreasing by the energy dissipation identity \eqref{eq_mass} and $\int_{t-1}^t E(u(s))\,\mathrm ds\le C  \| u(t-1) \|_{L^2}^2 $ for $t\ge1$,
the gradient has the same exponential bound after adjusting the constants.
Together with the bound on $[0,1]$, this implies
\begin{align*}
\|u (t) \|_{H_{x}^1}\le  C e^{- c t}, \ \forall\, t > 0 .
\end{align*}
This means that the solution to CGL on compact three-dimensional Riemannian manifolds with Dirichlet boundary condition is a fast solution.

\end{remark}

\begin{remark}\label{re4v177}

The global well-posedness of the CGL on compact three-dimensional Riemannian manifolds can be extended to any three-dimensional Riemannian manifold that satisfies the following assertion:
\begin{align}\label{eq1.5v177}
 \text{ the Ricci curvature of $ \left( \mathcal{M}^3,g \right)$ is bounded from below and
$\inf\limits_{p \in \mathcal{M}^3} \mathrm{Vol}_g B(p, 1) > 0$. }
\end{align}

\end{remark}

\begin{remark}[CGL equation on $d$-dimensional Riemannian manifolds]\label{re3v121}

Compared with the three-dimensional case, for any dimension $d \ge 1$,
the natural energy-subcritical and energy-critical exponent ranges are
\[
1<p\le1+\frac4{d-2}\quad(3\le d < \infty ),\qquad  1<p<\infty\quad(d=1,2).
\]
Extending the global well-posedness and asymptotic conclusions in Theorem \ref{th2v79} to other dimensions,
including initial-boundary problems, requires the relevant local well-posedness theory and higher-order energy estimates.
The proofs in this paper establish the stated conclusions in dimension three;
the other dimensions are left as possible extensions.
\end{remark}

\begin{remark}[Long time behavior of the solution to the NLH]

The conclusions of Theorem \ref{th2v79} and Remark \ref{re2v118}
extend directly to the NLH posed on the corresponding Riemannian manifolds
\begin{align*}
\left\{
\begin{aligned}
 &\partial_t u -  \Delta_g u = -  f(u ),  \\
 	& u(0, x) =  u_0 (x),
\end{aligned}
\right.
\end{align*}
 since all of our arguments for \eqref{eq1.CGL} remain valid
in the special case $\Im z=0$.
We point out that all related classical results given in \cite{GGH,Ha01,Ha06,HJ,HJK} were established for bounded domains with Dirichlet or Neumann boundary conditions, and the extension to compact Riemannian manifolds currently does not exist in the literature.
This generalization to the geometric setting therefore also constitutes one of the novel contributions of our work.

\end{remark}

Relying on the global well-posedness established in Theorem \ref{th2v79} and Remark \ref{re4v177} for the CGL equations,
by resorting to an inviscid limit argument,
we deduce the existence of global weak solutions and a restricted weak-strong uniqueness result for the energy-subcritical and energy-critical defocusing nonlinear Schr\"odinger equations on three-dimensional Riemannian manifolds satisfying \eqref{eq1.5v177}. The corresponding equation takes the form
	\begin{align}\label{eq1.1v18}
	\begin{cases}
		i \partial_t v + \Delta_g v -  f(v) = 0 , \\
		v(0) = v_0 \in H^1
,
	\end{cases}
\end{align}
where $v: \mathbb{R}_t \times \mathcal{M}^3_x \to \mathbb{C}$. The mass and energy functionals, conserved for sufficiently regular solutions, are
\begin{align*}
	 & M(v(t))=\int_{\mathcal{M}^{3}}|v(t,x)|^2 \,\mathrm{d}\mu,
	\intertext{ and }
& 	E(v(t))=\int_{\mathcal{M}^{3}}\left(\frac{1}{2}|\nabla_g v|^2+\frac{1}{p+1}|v|^{p+1}\,\right) \, \mathrm{d}\mu.
\end{align*}
Our result can be stated as follows.
	\begin{theorem}[Global weak solutions to the three-dimensional energy-subcritical and energy-critical defocusing nonlinear Schr\"odinger  equations]\label{th1.2v21}

	Let $1<p\le5$. Suppose the initial datum $v_0$ of \eqref{eq1.1v18} belongs to $H_{x}^{1} \left(\mathcal{M}^{3} \right)$.
Then equation \eqref{eq1.1v18} admits a global weak solution $v$. A weak-strong uniqueness statement holds under the additional assumptions in 2):

		\begin{itemize}
			\item[1)]
Global existence: There exists a global weak solution $v\in L_t^\infty H_x^1   \bigcap C_t^0 L^2_x \left(\mathbb{R} \times \mathcal{M}^3  \right)$ of \eqref{eq1.1v18} satisfying the energy inequality
			\begin{align}\label{eq1.8v256}
				E(v(t)) \le E(v_0)
			\end{align}
			and mass conservation
			\begin{align}\label{eq1.9v256}
				M(v(t)) = M(v_0).
			\end{align}

    \item[2)]
     Weak-strong uniqueness:
        Let $\tilde v$ be a global strong solution of
\eqref{eq1.1v18}, with $\tilde v\in C_t^0H_x^2$ when $1<p\le2$
and $\tilde v\in C_t^0H_x^3$ when $2<p\le5$.
Let $v$ be a global weak solution from the same datum satisfying \eqref{eq1.8v256} and \eqref{eq1.9v256}
 in 1). Then $v\equiv\tilde v$.
		\end{itemize}

	\end{theorem}

\begin{remark}

Related global strong solution results for the energy-critical defocusing nonlinear Schr\"odinger equation have been established on hyperbolic spaces \cite{IPS}, asymptotically flat spaces \cite{Jao},
cylinders \cite{CGZ18,HP14,IPS1,Zhao}, tori \cite{IP2,Kwak,Yue}, and compact Riemannian manifolds and bounded domains \cite{PTW}.
For the energy-subcritical defocusing nonlinear Schr\"odinger equation on compact Riemannian manifolds, see \cite{BGT,BGT16}.

\end{remark}

As a complementary result, we will  discuss the noncompact setting.
Under the geometric and analytic assumptions stated in
Assumption \ref{assump:A} (see Section \ref{se5}), we have

\begin{theorem}[Decay of the solutions to the CGL equation on noncompact three-dimensional Riemannian manifolds]\label{th3.1v3}
Let $ \left(\mathcal M^3,g \right)$ be noncompact and obey Assumption~\ref{assump:A}, and let $u_0 \in H^1 \left(\mathcal{M}^3 \right) $. For $\frac73\le p\le5$,
equation \eqref{eq1.CGL} with initial datum
$u_0\in H_x^1 \left(\mathcal M^3 \right)$ admits a global solution $u \in C_t^0 H_x^1 \left( \mathbb{R}_+ \times \mathcal{M}^3 \right)$.
Furthermore,
\begin{align}\label{eq6.2v51}
\|u \|_{L_{t,x}^{\frac{5(p-1)}{2}} \left(\mathbb{R}_+ \times \mathcal{M}^3 \right)} < \infty,
\intertext{and}
\lim\limits_{t \to \infty } \|u(t) \|_{H_x^1 \left(\mathcal{M}^3 \right)} = 0.
\label{eq6.3v51}
\end{align}
\end{theorem}

\begin{remark}

If $\mathcal{M}^3$ is $\mathbb{R}^3$, we could obtain \eqref{eq6.3v51} for the full range $1<p\leq 5$ by using a truncated virial type argument.
We believe that \eqref{eq6.3v51} holds for any solution to the equation \eqref{eq1.CGL} on a noncompact Riemannian manifold $ \mathcal M^3$ that obeys Assumption~\ref{assump:A}. As this objective is far from our main theme, we do not present the full proof in the current manuscript.
\end{remark}
Note that on noncompact three-dimensional Riemannian manifolds, \eqref{eq6.3v51} cannot be upgraded with a uniform decay rate in general. The essential reason is that the spectrum of the Laplace-Beltrami operator
$\sigma(- \Delta_g)$\footnote{Here $\sigma(- \Delta_g )$ denotes the spectrum of $- \Delta_g$ on $L^2 \left(\mathcal{M}^3 \right)$.} might contain $0$, which prevents the fast decay of the propagator of CGL.
Denote by $\lambda_{\min}  $ the bottom of $\sigma( - \Delta_g)$, that is,
\begin{align*}
\lambda_{\min}   := \inf \sigma( - \Delta_g).
\end{align*}
We illustrate this point by the following proposition:

\begin{proposition}\label{pr1.4v152}
Let $ \left(\mathcal M^3,g \right)$ be a noncompact Riemannian manifold satisfying Assumption~\ref{assump:A} with $\lambda_{\min}   = 0$.
Then for $1< p \le 5$, there does not exist a function $\Psi: [0,+\infty[ \rightarrow \mathbb{R}_+$ with $\lim\limits_{t\to +\infty}\Psi(t)=0$
such that for any solution $u(t) \in C_t^0 H_x^1 \left(\mathbb{R}_+ \times \mathcal{M}^3 \right)$ of equation \eqref{eq1.CGL}
with initial data $u_0 \in H_x^1 \left(\mathcal{M}^3 \right)$, the following estimate holds:
\[
\|u(t) \|_{H_x^1 \left(\mathcal{M}^3 \right)} \le C\left( \|u_0 \|_{H_x^1 \left(\mathcal{M}^3 \right)} \right) \Psi(t),
\quad \forall \, t \ge 0.
\]
\end{proposition}

In contrast to the case $\lambda_{\min} =0$, for noncompact Riemannian manifolds with $\lambda_{\min } > 0$,
we expect an exponential decay. In fact, we have

\begin{proposition}\label{pr1.5v275}
Let $ \left(\mathcal M^3,g \right)$ be a noncompact Riemannian manifold satisfying Assumption~\ref{assump:A} with $\lambda_{\min} > 0$.
Then for $1< p \le 5$, the solution $u(t) \in C_t^0 H_x^1 \left(\mathbb{R}_+ \times \mathcal{M}^3 \right)$ of equation \eqref{eq1.CGL}
with initial data $u_0 \in H_x^1 \left(\mathcal{M}^3 \right)$ satisfies
\[
\|u(t)\|_{H^1 \left(\mathcal M^3 \right)} \le \left(2 \left(1+\lambda_{\min}^{-1} \right)E(u_0)\right)^{ \frac12}
 e^{- \Re z \lambda_{\min} t}, \ \forall\, t \ge 0.
\]
\end{proposition}

\subsection{Methodology}

Our analysis relies on a rigorous energy method,
which differs from the approaches adopted in existing works on the semilinear heat equation (see \cite{Ben,GGH,Ha01,Ha06,HJ,HJK,QS,RZ}).
Thanks to the maximum principle, the $L_x^\infty$-norm of solutions to the NLH equation can be controlled directly.
This principle effectively endows the NLH equation with a subcritical structure for arbitrary power-type nonlinearities and may itself be interpreted in terms of an associated Lyapunov functional.
By contrast, no counterpart of the maximum principle is available for the CGL equation.
A similar situation arises in several geometric flows, such as the harmonic map heat flow and the Yang--Mills heat flow, where the analysis relies crucially on energy dissipation rather than on a scalar maximum principle; see, for instance, \cite{Rade,Str}.
Accordingly, to establish global well-posedness together with long time asymptotic behavior,
the analysis must be carried out by the energy method within $L_x^2$-based Sobolev spaces. Since the associated energy bounds the $H_x^1$
norm of solutions, we exploit the defocusing character and inherent energy dissipation of the equation to prove global well-posedness across the energy-subcritical and critical regimes $1<p\leq 5$.
Our global well-posedness result for the CGL equation applies to compact three-dimensional Riemannian manifolds.
An inviscid limit argument also yields global weak solutions of the defocusing NLS,
with the restricted weak-strong uniqueness statement of Theorem~\ref{th1.2v21}.

Concerning the $L_x^\infty$ decay of solutions to CGL, we also use the energy method. For the CGL equation, the basic Lyapunov information used here comes from the mass and energy dissipation identities \eqref{eq_ener} and \eqref{eq_mass} in Proposition~\ref{pr2.12v70}, which provide decay estimates in the $L_x^2$ and $H_x^1$ norms. In the three-dimensional case, these estimates are insufficient to control the $L_x^\infty$ norm of solutions. To overcome this difficulty,
we establish $L_x^\infty$ decay estimates by deriving decay bounds for the $H_x^2$ norm. The main obstacle lies in the absence of any $H_x^2$-level Lyapunov functional for the CGL, in sharp contrast to the semilinear heat equation, for which $\|\partial_t u\|_{L_x^2}^2$ serves as a valid $H_x^2$-type Lyapunov functional. Drawing inspiration from the modified energy methods proposed in \cite{BPTV,PTV,PTV3,PTV2,PTV1} for controlling Sobolev norm growth in nonlinear dispersive equations,
we construct an $H_x^2$-level modified energy that satisfies an almost monotonicity property.
Combined with the smoothing effect inherent to the CGL equation (see Theorem \ref{th2.6v107}), which ensures that solutions emanating from $H_x^1$ initial data belong to $H_x^2$ for any positive time, this structure allows us to simultaneously control the $H_x^2$ and $L_x^\infty$ norms via a bootstrap argument. Notably, the proposed modified energy method fully exploits the energy-subcritical and energy-critical nature of the equation, i.e., $1 < p \le 5$ in three dimensions.

To establish the $H_x^2$ decay, we employ the modified energy method (see Subsection~\ref{H2decay}, where the definition of the modified energy functional and the complete derivation of the $H_x^2$ decay estimates are given). The main novelty of this approach lies in the construction of a modified energy functional, which is tailored to cancel the most dangerous nonlinear terms and thereby close the higher-order estimates without loss of derivatives. The argument below yields decay estimates for second-order spatial derivatives by combining the corrected modified energy identity with a quantitative time-derivative estimate. Its extension to further derivatives would also require sufficient differentiability of $f$, together with the corresponding Sobolev and Gagliardo-Nirenberg inequalities.
\vspace{0.3cm}

We further characterize the asymptotic behavior of solutions by a fast/slow dichotomy. Such dynamical behavior is similar to that of the semilinear heat equations \cite{Ben,GGH,Ha01,Ha06,HJ,HJK,QS}. For fast solutions, the dynamics are dominated by the zero-mean component, which decays exponentially fast, whereas the decay rate for the slow solutions is governed by the nonlinear dynamics of zero modes
\[
\frac{\mathrm{d}m(t)}{\mathrm{d}t} = -z f(m).
\]
The crucial feature is that all nonzero modes exhibit exponential decay.
To substantiate this claim, we begin with the energy identity:
\begin{align}\label{eq1.7v132}
\begin{aligned}
& \frac{1}{2}\frac{\mathrm{d}}{\mathrm{d}t} \left\|u(t)-\dashint_{\mathcal{M}^3} u(t)\,\mathrm{d}\mu \right\|_{L_x^2}^2
+ \Re z \|\nabla_g u\|_{L_x^2}^2 \\
= & -\Re z \cdot \Re \int_{\mathcal{M}^3} \left( f(u) - f(m) \right) \overline{(u-m)} \,\mathrm{d}\mu
+ \Im z \cdot \Im \int_{\mathcal{M}^3} \left( f(u) - f(m) \right) \overline{(u-m)} \,\mathrm{d}\mu.
\end{aligned}
\end{align}
For the energy-subcritical and energy-critical semilinear heat equation,
 the second term on the right hand side of \eqref{eq1.7v132} vanishes; the algebraic property
\[
\Re \left[ \left( f(\omega_1) - f(\omega_2) \right) \overline{(\omega_1 - \omega_2)} \right] \geq 0,
\]
and Poincar\'e's inequality yield
\[
\left\| u(t) - \dashint_{\mathcal{M}^3} u(t) \,\mathrm{d}\mu \right\|_{L_x^2(\mathcal{M}^3)}
\leq C e^{-\lambda t}.
\]
For the CGL equation, however, such monotonicity fails, forcing us to treat the nonlinear term perturbatively.
The decay properties of solutions constitute the key ingredient in closing the Gr\"onwall-type inequality and lie at the core of our analysis.

A few remarks on the dimensional scope are in order. In one and two dimensions, the Sobolev inequalities are more favorable,
whereas in dimensions $d\ge4$ both the critical exponent and the interpolation estimates change.
The higher-order estimates are implemented here on a compact three-dimensional Riemannian manifold.
 Extension to other dimensions is straightforward by verifying the associated local well-posedness theory
 and nonlinear estimates developed for the three-dimensional case.

Due to different mechanisms for the long time dynamics on noncompact Riemannian manifolds,
the proof of Theorem \ref{th3.1v3} relies on {different} techniques on noncompact Riemannian manifolds.
Our analysis combines the heat kernel estimate on noncompact Riemannian manifolds and nonlinear energy methods to investigate the long time behavior of the complex Ginzburg-Landau equation on noncompact three-dimensional Riemannian manifolds.
Under suitable geometric assumptions,
we first derive the dispersive estimate for the linear Ginzburg-Landau propagator $e^{t z\Delta_g}$
 by the heat kernel estimates on noncompact Riemannian manifolds.
The complex heat-kernel estimate \eqref{eq1.4v3} is then combined with the abstract Keel-Tao framework to obtain global Strichartz estimates for the CGL on noncompact Riemannian manifolds, which form the basis for the well-posedness theory and the subsequent spacetime estimate.
To study the asymptotic behavior, we exploit the mass and energy dissipation identities \eqref{eq_ener} and \eqref{eq_mass} from Proposition~\ref{pr2.12v70} to prove that the energy is integrable in time and that
the gradient tends to zero. For $p> \frac73$, interpolation then allows the
solution to enter the critical small-data regime at a sufficiently large
time; the endpoint $p= \frac73$ is handled directly by the mass dissipation identity \eqref{eq_ener}.
Together with the small-data spacetime estimate and local regularity, this yields the global spacetime bound.
The convergence of the $L^2$ norm is established through the control of the nonlinear tail by the Strichartz estimate and the exclusion of nontrivial $L^2$ zero modes.
Finally,
we demonstrate the absence of a uniform decay rate for $1 < p \le 5$ when $\lambda_{\min } = 0$ and exponential decay of the $H^1$-norm of the solution to CGL when $\lambda_{\min} > 0$, revealing the essential influence of the noncompact geometry and the bottom of the spectrum of $- \Delta_g$ on the long time dynamics of the solution to the equation.

\textbf{Layout.}
The proof of Theorem \ref{th2v79} comprises three distinct components, corresponding to the three assertions in the theorem.

Section \ref{se2} establishes global well-posedness of the CGL equation \eqref{eq1.CGL} on the compact three-dimensional Riemannian manifold $\mathcal{M}^3$ without boundary for $1<p\leq 5$, which provides the necessary background for the subsequent asymptotic analysis.

Section \ref{se3} is devoted to the derivation of decay estimates for the mass and energy functionals (Subsection \ref{decay_ME}) as well as for the $H_x^2$-norm (Subsection \ref{H2decay}). These bounds serve as the key ingredients for the long time analysis carried out in Section \ref{se4}.

Section \ref{se4} analyzes the long time asymptotic behavior of the solution, establishing:
\begin{itemize}
  \item the polynomial decay of the spatial mean $m(t)$ (part (i));
  \item the exponential decay of the fluctuation $v(t)$ in $H_x^1$ (part (ii));
  \item the dichotomy between fast exponential decay and slow algebraic decay for the mean component, with the refined asymptotic expansion in the slow case (part (iii)).
\end{itemize}

Furthermore, building upon the global well-posedness of the CGL equation and the inviscid limit method,
we derive global weak solutions of the defocusing NLS and establish weak-strong uniqueness in the range $1<p\le5$ in Theorem~\ref{th1.2v21}.
This constitutes the main assertion of Theorem \ref{th1.2v21}.

Finally, Section \ref{se5} is dedicated to decay estimates for the CGL equation on noncompact three-dimensional Riemannian manifolds,
corresponding to the result stated in Theorem \ref{th3.1v3}.
 We also demonstrate the absence of a uniform decay rate for $1 < p \le 5$ when $\lambda_{\min } = 0$ and exponential decay of the $H^1$-norm of the solution to CGL when $\lambda_{\min} > 0$.

\textbf{Notations.}
Let $\mathbb{N}$ denote the set of natural numbers.
We use $X \lesssim Y$ when $X \le CY$ for some constant $C> 0$ and $X \sim Y$ when $X \lesssim Y \lesssim X$.
We use $X \lesssim_\alpha Y$ to denote $X\le C(\alpha) Y$, where $C(\alpha)$ is a positive constant depending on $\alpha$.
We write $C=C( \alpha, \beta,\cdots)$ to mean that $C$ is a constant depending on $\alpha, \beta,\cdots$ and $C$ is allowed to change from line to line.

We define the space $L^2 \left(\mathcal{M}^3  \right)$ of complex-valued functions on $\mathcal{M}^3 $ with the following norm and inner product:
\begin{align*}
L^2 \left( \mathcal{M}^3  \right) & = \left\{ f : \mathcal{M}^3  \to \mathbb{C} \,\big|\,
\int_{\mathcal{M}^3 } |f(x)|^2 \,\mathrm{d}\mu < \infty  \right\},
\intertext{ and }
(f, g)_{L^2}   & = \int_{\mathcal{M}^3 } f(x) \overline{g(x)} \,\mathrm{d}\mu.
\end{align*}
For an integer $k$ and a smooth function $u: \mathcal{M}^3   \to \mathbb{C}$, we denote by $\nabla^k_g u $ the $k$-th covariant derivative of $u$, and $\left|\nabla^k_g u  \right|$ the norm of $\nabla_{g}^k u$ defined in a local chart by
\begin{align*}
\left|\nabla^k_g u \right|^2 = g^{i_1j_1} \cdots g^{i_k j_k}  \left( \nabla^k u \right)_{i_1 \cdots i_k }  \left( \nabla^k \bar{u}  \right)_{j_1 \cdots j_k}.
\end{align*}
For two vector fields $X, Y$, we have the inner product:
\begin{align*}
    \langle X, Y  \rangle_g : = g_{ij}(x) X^i Y^j.
\end{align*}
For any space-time slab $I \times \mathcal{M}^3 $, we use $L_t^q L_x^r \left(I \times \mathcal{M}^3  \right)$ to denote the space of functions $u: I \times \mathcal{M}^3  \to \mathbb{C}$,
whose norm satisfies
	$\|u\|_{L_t^q L_x^r  \left(I \times \mathcal{M}^3   \right)} : = \left( \int_I \|u(t) \|_{L_x^r \left(\mathcal{M}^3  \right) }^q \,\mathrm{d}t \right)^\frac1q < \infty$,
	with the usual modification when $q$ or $r$ is equal to infinity. When $q = r$, we abbreviate $L_t^q L_x^q$ as $L_{t,x}^q$. Similarly, we can define $L_t^q W_x^{s, r} \left( I \times \mathcal{M}^3  \right)$.

\section{Well-posedness theory}\label{se2}

In this section, we develop the well-posedness theory for the complex Ginzburg-Landau equation on three-dimensional Riemannian manifolds.
We first establish local and global well-posedness in the energy space of CGL on three-dimensional Riemannian manifolds. Then, we construct global weak solutions to the corresponding nonlinear defocusing Schr\"odinger equation (NLS) through an inviscid limit approximation.
We also establish weak-strong uniqueness for NLS solutions.

As in \cite{Tre}, we now define the inhomogeneous Sobolev spaces on $\mathcal{M}^3 $.
\begin{definition}[Sobolev space]

 \begin{itemize}
   \item
   For a positive integer $k$ and $p \in ]1, \infty [$, the Sobolev space $W^{k, p} \left(\mathcal{M}^3  \right)$ is defined to be the completion of $\left\{ \phi  \in C^\infty \left(\mathcal{M}^3  \right) \,|\, \| \phi \|_{W^{k, p}} < \infty \right\}$ under the norm
       \[
       \| \phi \|_{W^{k, p}} := \|  \phi \|_{L^p}+ \sum_{j=1}^k  \left\|\nabla^j  \phi  \right\|_{L^p}.
       \]
   \item
For a fractional $s >0$ and $p \in ]1, \infty[$,
we define the Sobolev space $W^{s, p} \left( \mathcal{M}^3 \right)$ to be the collection of functions $ \phi  \in L^p \left(\mathcal{M}^3 \right)$ such that $( 1- \Delta_{g} )^\frac{s}2  \phi  \in L^p \left(\mathcal{M}^3 \right)$ under the norm
\begin{align*}
\| \phi  \|_{W^{s, p}  \left( \mathcal{M}^3 \right)} =  \left\| ( 1- \Delta_{g} )^\frac{s}2  \phi   \right\|_{L^p \left( \mathcal{M}^3 \right)}  .
\end{align*}
For $s >0$, let $H^s = W^{s, 2}$.
 \end{itemize}

\end{definition}

We also recall the Sobolev embedding theorem (see, for instance, \cite{Hebey0,Hebey,Tre}):

\begin{theorem}[Sobolev embedding theorem] \label{th2.6v116}

Let $ \left( \mathcal{M}^3 ,g \right)$ satisfy the assertion \eqref{eq1.5v177}.
 Then for any $1  < p \le q < \infty$ and $s \ge 0$, we have
\begin{align*}
W^{s,p} \left( \mathcal{M}^3 \right) \subseteq L^q \left( \mathcal{M}^3 \right),
 \end{align*}
where $\frac1p  =  \frac1q + \frac{s}3$.

\end{theorem}

In the following, we recall Poincar\'e's inequality in the compact setting (see, for instance, \cite{Hebey,Tre}):
\begin{proposition}[Poincar\'e's  inequality]\label{pr2.3v122}
Let $\left(\mathcal{M}^3,g\right)$ be compact. Define $H$ to be:
\begin{itemize}
\item if $\partial \mathcal{M}^3= \emptyset$ or $\partial \mathcal{M}^3 \ne \emptyset$ with the Neumann boundary condition,
$H =  \left\{ \phi \in H^1 \left( \mathcal{M}^3 \right) \,|\, \int_{\mathcal{M}^3 } \phi (x) \,\mathrm{d} \mu = 0  \right\}$;

\item if $\partial \mathcal{M}^3 \ne \emptyset $ with the Dirichlet boundary condition,
$H= H^1_0 \left( \mathcal{M}^3  \right) $;

\end{itemize}
Then
\begin{align*}
\| \phi  \|_{L^2 \left( \mathcal{M}^3 \right)}^2 \le \lambda_1^{-1} \|\nabla_g  \phi  \|_{L^2 \left(\mathcal{M}^3 \right)}^2,
\ \forall\, \phi \in H,
\end{align*}
where $\lambda_1$ is the first {non-zero} eigenvalue of $- \Delta_g$.
\end{proposition}
There is an alternative formulation of Poincar\'e's inequality in the compact setting, presented in the following proposition.
\begin{proposition}[Poincar\'e's  inequality]\label{th2.10v116}
{ {Assume that $\left(\mathcal{M}^3,g\right)$ is compact}}. For any $ \phi  \in  W^{1, q}  \left( \mathcal{M}^3  \right)$, where $q \in [1, 3[$, there exists a positive constant $C$ such that \begin{align*}
\left\| \phi -
 \dashint_{\mathcal{M}^3} \phi(x)  \,\mathrm{d}\mu
 \right\|_{L^q \left( \mathcal{M}^3  \right)}
\le
C  \|\nabla_g \phi \|_{L^q \left( \mathcal{M}^3  \right) } .
\end{align*}

\end{proposition}

We now derive two fundamental integral identities satisfied by all finite-energy solutions to the CGL equation on $\mathcal{M}^3$:
a mass dissipation identity and an energy dissipation identity.
These relations serve as the core machinery for both the global well-posedness theory established below and the long time decay analysis on compact three-dimensional Riemannian manifolds developed in the subsequent sections.
\begin{proposition}\label{pr2.12v70}
Any solution $u$ to \eqref{eq1.CGL} satisfies the mass dissipation identity
\begin{equation}\label{eq_ener}
	M(u(t))
+2 \Re{z} \int_0^t\|\nabla_g u(\tau) \|_{L_{x}^2 \left(\mathcal{M}^{3 } \right) }^2 \, \dd\tau
+ 2\Re{z}\int_0^t \|u(\tau) \|_{L_{x}^{p+1} \left( \mathcal{M}^3  \right) }^{p+1} \, \dd\tau
=  M(u_0) ,
\end{equation}
where the mass functional is defined as
$ M(u(t)):= \|u(t)\|_{L_x^{2} \left( \mathcal{M}^{3 } \right)}^{2}$. It also satisfies the energy dissipation identity
\begin{equation}\label{eq_mass}
E(u(t)) + \frac{\Re z}{|z|^2} \int_0^t\|\partial_t u(\tau) \|_{L_{x}^2 \left(  \mathcal{M}^{3 } \right)}^2 \, \mathrm{d}\tau = E(u_0),
\end{equation}
with $E(u (t) )$ being given by \eqref{u_energy_func}.
\end{proposition}

\begin{proof}
The result for $\Im{z} =0$ has already been established for the semilinear heat equation, see \cite{QS}.
Therefore, we focus on the case $\Im{z} \neq 0$. We first prove \eqref{eq_ener}.
Differentiating directly, we find
\begin{equation*}
	\begin{split}
\left(1+\frac{(\Re{z})^{2}}{(\Im{z})^{2}}\right)
\frac{ \mathrm{d} }{ \mathrm{d} t} \left( \|u (t) \|_{L_x^{2} \left( \mathcal{M}^{3 } \right)}^{2} \right)
+\frac{2\Re{z}|z|^2}{(\Im{z})^{2}}
\|\nabla_{g} u (t) \|_{L_x^{2} \left( \mathcal{M}^{3 } \right)}^{2}
+\frac{2\Re{z}|z|^2}{(\Im{z})^{2}}\|u (t) \|_{L_x^{p+1} \left( \mathcal{M}^{3 } \right)}^{p+1}=0.
	\end{split}
\end{equation*}
This implies
\begin{align}\label{u_L2}
\frac{ \mathrm{d} }{ \mathrm{d} t} \left( \|u (t) \|_{L_x^{2} \left( \mathcal{M}^{3 } \right)}^{2} \right)
+ 2\Re{z}  \|\nabla_{g} u (t) \|_{L_x^{2} \left( \mathcal{M}^{3 } \right)}^{2}
+ 2\Re{z} \|u (t) \|_{L_x^{p+1} \left( \mathcal{M}^{3 } \right)}^{p+1}=0.
\end{align}
Integrating \eqref{u_L2} in time yields \eqref{eq_ener}.

To prove \eqref{eq_mass}, we use \eqref{eq1.CGL},
 and begin with the following computations:
	\begin{equation}\label{id3.1_1}
	\begin{split}
\frac{ \mathrm{d} }{ \mathrm{d} t} \left( \|\nabla_{g} u (t) \|_{L_x^{2} \left( \mathcal{M}^{3 } \right)}^{2} \right)
&=-\frac{2\Re{z}}{\Im{z}}\Re \left(\partial_{t}u (t) ,i\Delta_{g} u (t)  \right)_{L^2}
-\frac{2|z|^2}{\Im{z}}\Re \left(  f(u (t) ) ,i\Delta_{g} u (t)  \right)_{L^2} .
\end{split}
	\end{equation}
For the first term of the RHS of \eqref{id3.1_1}, there holds
\begin{equation}\label{id3.1_2}
	\begin{split}
-\frac{2\Re{z}}{\Im{z}}\Re \left(\partial_{t}u (t) ,i\Delta_{g} u  (t) \right)_{L^2}
= & -\frac{2\Re z}{|z|^2}\|\partial_{t}u (t) \|_{L_x^{2} \left( \mathcal{M}^{3 } \right)}^{2}-\frac{2\Re{z}}{\Im{z}}\Re \left(\partial_{t}u (t) ,i
f(u (t) )  \right)_{L^2} .
	\end{split}
\end{equation}
For the second term of the RHS of \eqref{id3.1_1}, a direct calculation shows
\begin{equation}\label{id3.1_3}
	\begin{split}
-\frac{2|z|^2}{\Im{z}}\Re \left( f(u (t) ) ,i\Delta_{g} u  (t) \right)_{L^2}
& =\frac{2\Re{z}}{\Im{z}}\Re \left(\partial_{t}u (t) ,i f(u (t) ) \right)_{L^2} -2\Re \left( f(u (t) ) ,\partial_{t}u (t)  \right)_{L^2} .
\end{split}
	\end{equation}
Putting \eqref{id3.1_1}, \eqref{id3.1_2}, and \eqref{id3.1_3} together, we obtain
\begin{equation*}
	\begin{split}
\frac{ \mathrm{d} }{ \mathrm{d} t} \left( \|\nabla_{g} u (t) \|_{L_x^{2} \left( \mathcal{M}^{3 } \right)}^{2} \right)
+{ \frac{2}{p+1} }\frac{ \mathrm{d}} { \mathrm{d} t} \left( \|u (t) \|_{L_x^{p+1} \left(  \mathcal{M}^{3 } \right)}^{p+1} \right)
+\frac{2\Re z}{|z|^2}\|\partial_{t}u (t) \|_{L_x^{2} \left( \mathcal{M}^{3 } \right)}^{2}=0.
	\end{split}
\end{equation*}
Integrating the above identity with respect to time yields \eqref{eq_mass}.

\end{proof}

\subsection{Global well-posedness of the CGL equation on general three-dimensional Riemannian manifolds}

In this subsection, we show the global well-posedness of the equation \eqref{eq1.CGL} on Riemannian manifolds that satisfy \eqref{eq1.5v177}.

\subsubsection{Blowup criterion}

{Let us start with a basic smoothing effect for the linear CGL flow:}
\begin{lemma}\label{le1.1v156}
For any $t > 0$, we have
\begin{align}
 & \bigl\|e^{tz \Delta_g } f  \bigr\|_{H_x^1 \left( \mathcal{M}^{3 }  \right)}
    \le \|f \|_{H_x^1 \left( \mathcal{M}^{3 }  \right)}, \label{eq2.7v239}
\intertext{and}
&  \bigl\|\nabla_g e^{tz \Delta_g } f  \bigr\|_{L_x^2 \left( \mathcal{M}^{3 }  \right)}
    \lesssim t^{-\frac12} \|f \|_{L_x^2 \left( \mathcal{M}^{3 }  \right)}. \label{eq2.8v239}
\end{align}
\end{lemma}

\begin{proof}

Since $\Delta_g$ is self-adjoint and non-positive, the spectral theorem \cite{Lewin} yields a projection-valued measure $\mathrm{d}E_\lambda$ such that
\begin{equation*}
    -\Delta_g = \int_0^\infty \lambda \,\mathrm{d}E_\lambda,
\end{equation*}
and for any bounded Borel function $\varphi$,
\[
    \varphi(-\Delta_g) = \int_0^\infty \varphi(\lambda) \,\mathrm{d}E_\lambda.
\]
In particular,
\begin{equation*}
    e^{tz \Delta_g } = \int_0^\infty e^{-tz \lambda} \,\mathrm{d}E_\lambda.
\end{equation*}
By direct calculation, we have
\begin{equation*}
    \left\|e^{tz \Delta_g} f \right\|_{L_x^2 \left( \mathcal{M}^{3 }  \right)}
    \le \|f\|_{L_x^2 \left( \mathcal{M}^{3 }  \right)} .
\end{equation*}
Similarly, using the spectral representation of the gradient, we see
\begin{align*}
    \left\|\nabla_g e^{tz \Delta_g } f \right\|_{L_x^2 \left( \mathcal{M}^{3 }  \right)}^2
    &= \int_0^\infty \lambda \, e^{-2t (\Re z) \lambda} \,\mathrm{d}\mu_f(\lambda) \le \int_0^\infty \lambda \,\mathrm{d}\mu_f(\lambda)
     = \|\nabla_g f \|_{L_x^2 \left( \mathcal{M}^{3 }  \right)}^2,
\end{align*}
where $\mathrm{d}\mu_f(\lambda) := \mathrm{d}(E_\lambda f,f)_{L^2}$.
Together with the $L^2$ bound, we obtain \eqref{eq2.7v239}.

For \eqref{eq2.8v239}, a simple calculation gives
\begin{align*}
    \left\|\nabla_g e^{tz \Delta_g } f \right\|_{L_x^2 \left( \mathcal{M}^{3 }  \right)}
    &= \left(  \int_0^\infty \lambda \, e^{-2t (\Re z) \lambda} \,\mathrm{d}\mu_f(\lambda) \right)^\frac12
     \le \left( \frac{1}{2e\, t \Re z} \int_0^\infty \mathrm{d}\mu_f(\lambda) \right)^\frac12
    \lesssim (t \Re z)^{-\frac12} \|f\|_{L_x^2 \left( \mathcal{M}^{3 }  \right)}.
\end{align*}
This completes the proof.
\end{proof}
Consequently, we have:
\begin{corollary}\label{lem:smoothing}
For any $T>0$ and $f\in H_x^1 \left(\mathcal{M}^3 \right)$, we have
\begin{align}\label{eq:smoothing}
    \bigl\|e^{tz\Delta_g} f \bigr\|_{L_t^2 H_x^2 \left([0,T] \times \mathcal{M}^3 \right) }
    \le C_{T} \|f\|_{H_x^1 \left(\mathcal{M}^3 \right)},
\end{align}
where $C_{T} > 0$ depends on $T$, $\Re z$, and the geometry of $\mathcal{M}^3$. \end{corollary}

\begin{proof}
Using the spectral resolution of $-\Delta_g$ and the associated measure
$\mathrm{d}\mu_f$ as in the proof of Lemma~\ref{le1.1v156}, we obtain
\begin{align}\label{eq:aux}
    & \ \left\|e^{-tz(-\Delta_g + 1)} f \right\|_{L_t^2 H_x^2 \left([0,T] \times \mathcal{M}^3   \right)}  \notag \\
    &\sim \left\|(-\Delta_g + 1) e^{-tz(-\Delta_g + 1)} f \right\|_{L_{t,x}^2 \left([0,T] \times \mathcal{M}^3   \right)}
    \notag \\
    &= \left( \int_0^T \int_0^\infty (\lambda+1)^2 e^{-2t \Re z (\lambda+1)}
    \,\mathrm{d}\mu_f(\lambda) \,\mathrm{d}t \right)^\frac12 \notag \\
    &= \left( \int_0^\infty (\lambda+1)^2
    \left( \int_0^T e^{-2t \Re z (\lambda+1)} \,\mathrm{d}t \right)
    \mathrm{d}\mu_f(\lambda) \right)^\frac12
    \lesssim (\Re z)^{- \frac12} \|f\|_{H_x^1 \left(\mathcal{M}^3 \right)}.
\end{align}
For any $t \in [0,T]$,  we have
\begin{align}\label{eq2.12v256}
\left\|e^{tz\Delta_g} f \right\|_{H_x^2 \left(\mathcal{M}^3 \right)}
    = e^{t \Re z} \left\|e^{-tz(-\Delta_g + 1)} f \right\|_{H_x^2 \left(\mathcal{M}^3 \right)}.
\end{align}
By \eqref{eq:aux} and \eqref{eq2.12v256}, we get
\begin{align*}
    \left\|e^{tz\Delta_g} f \right\|_{L_t^2H_x^2  \left([0,T] \times \mathcal{M}^3 \right)   }
    &= \left( \int_0^T e^{2t \Re z} \left\|e^{-tz(-\Delta_g + 1)} f \right\|_{H_x^2 \left(\mathcal{M}^3 \right)}^2 \,\mathrm{d}t
    \right)^{ \frac12} \\
    &\le e^{T \Re z} \left\|e^{-tz(-\Delta_g + 1)} f \right\|_{L_t^2 H_x^2 \left([0,T] \times \mathcal{M}^3  \right)}
    \lesssim e^{T \Re z} (\Re z)^{- \frac12} \|f\|_{H_x^1 \left(\mathcal{M}^3 \right)}.
\end{align*}
This completes the proof.
\end{proof}

\begin{lemma}[Blowup criterion]\label{lem:Blowupcriterion}
Assume that $u_0\in H_x^1 \left(\mathcal{M}^3 \right)$.
The equation \eqref{eq1.CGL} is locally well-posed in $C_t^0 H_x^1
\cap L_t^2 H_x^2 \left([0,T_{\max}[ \times \mathcal{M}^3 \right)$ with maximal lifespan $T_{\max}>0$.
Furthermore, the following blowup criterion holds:
\[
T_{\max}<\infty \implies
\begin{cases}
\|u\|_{L_t^2 H_x^{2} \left([0,T_{\max}[ \times \mathcal{M}^3 \right)} = \infty, & p=5,\\
\|u\|_{L_t^\infty H_x^1 \left([0, T_{\max}[ \times \mathcal{M}^3 \right)} = \infty, & 1< p< 5.
\end{cases}
\]
\end{lemma}

\begin{proof}
(1) We first consider the energy-critical case $p=5$.
Let
\[
A:=\|u_0\|_{H_x^1 \left(\mathcal M^3 \right)},
\]
and define
\[
B_R=\left\{ u\in C_t^0 H_x^1 \cap L^2_t  H_x^2\left( [0,T] \times   \mathcal M^3   \right)\,\big|\,
\|u\|_{L_t^\infty H_x^1 \left([0,T]\times \mathcal{M}^3 \right)}\le 2A,
\|u\|_{L_t^2 H_x^2 \left([0,T]\times \mathcal{M}^3 \right)}\le R \right\},
\]
equipped with the metric
\[
d(u,v) := \|u-v\|_{L_t^\infty H_x^1 \left([0,T]\times \mathcal{M}^3 \right)} + \|u-v\|_{L_t^2 H_x^2 \left([0,T]\times \mathcal{M}^3 \right)}.
\]
Define the mapping
\begin{align}\label{eq2.13v245}
\mathcal T(u)(t) = e^{tz\Delta_g}u_0 - z\int_0^t e^{(t- \tau )z\Delta_g} f(u( \tau )) \,\mathrm{d} \tau .
\end{align}
By Minkowski's inequality, \eqref{eq:smoothing}, and the Gagliardo-Nirenberg inequality, we have
\begin{align*}
& \left\| \int_0^t  e^{(t-\tau)z\Delta_g}f(u(\tau))\,\mathrm{d}\tau \right\|_{L_t^2 H_x^2 \left([0,T]\times \mathcal{M}^3 \right)}
\\
 & \leq  \int_0^T \left\| e^{(t-\tau)z\Delta_g}f(u(\tau)) \right\|_{L_t^2 H_x^2 \left([\tau,T]\times \mathcal{M}^3 \right)}
\,\mathrm{d}\tau \notag \\
& \leq C_T \bigl\| \|u\|_{H_x^1 \left(\mathcal{M}^3 \right)}\|u\|_{L_x^{\infty} \left(\mathcal{M}^3 \right)}^{4} \bigr\|_{L_\tau^1([0,T])}  \leq C_T \|u\|_{L_t^{\infty} H_x^1 \left([0,T]\times \mathcal{M}^3 \right)}^3 \|u\|_{L_t^2 H_x^2 \left([0,T]\times \mathcal{M}^3 \right)}^2 .
\end{align*}
Hence, for $u\in B_R$, we have
\[
\left\|\int_0^t e^{(t- \tau )z\Delta_g} f(u( \tau )) \,\mathrm{d} \tau  \right\|_{L_t^2 H_x^2 \left([0,T]\times \mathcal{M}^3 \right)}
\le C_T (2A)^3R^2.
\]
Consequently,
\[
\|\mathcal T(u)\|_{L_t^2 H_x^2 \left([0,T]\times \mathcal{M}^3 \right)} \le
 \left\|e^{tz\Delta_g}u_0 \right\|_{L_t^2 H_x^2 \left([0,T]\times \mathcal{M}^3 \right)}
+ C_T (2A)^3R^2.
\]
On the other hand, the trivial energy estimate yields
\begin{align*}
\|\mathcal T(u)\|_{L_t^{\infty}H_x^1 \left([0,T]\times \mathcal{M}^3 \right)}
 & \leq \|u_0\|_{H_x^1 \left(\mathcal{M}^3 \right)}
+C\|u\|_{L_t^{\infty} H_x^1 \left([0,T]\times \mathcal{M}^3 \right)}^3 \|u\|_{L_t^2 H_x^2 \left([0,T]\times \mathcal{M}^3 \right)}^2 .
\end{align*}
Hence, for $u\in B_R$,
\[
\|\mathcal T(u)\|_{L_t^\infty H_x^1 \left([0,T]\times \mathcal{M}^3 \right)} \le A + C(2A)^3R^2.
\]
Since
$\lim\limits_{T\to0} \left\|e^{tz\Delta_g}u_0 \right\|_{L_t^2 H_x^2 \left([0,T]\times \mathcal{M}^3 \right)} = 0$,
we choose $R = 2  \left\|e^{tz\Delta_g}u_0 \right\|_{L_t^2 H_x^2 \left([0,T]\times \mathcal{M}^3 \right)}$,
and then take $T>0$ sufficiently small so that $R\le A$ and $C_T (2A)^3R \le \frac12$. It follows that
\[
\|\mathcal T(u)\|_{L_t^\infty H_x^1 \left([0,T]\times \mathcal{M}^3 \right)} \le 2A
\quad\text{and}\quad \|\mathcal T(u)\|_{L_t^2 H_x^2 \left([0,T]\times \mathcal{M}^3 \right)} \le R.
\]
Hence $\mathcal{T}: B_R \to B_R$.

Next, we show that $\mathcal T$ is a contraction. For any $u,v\in B_R$,
using the chain rule,
we get
\[
|f(u)-f(v)|\le C \left(|u|^4+|v|^4 \right)| u-v|,
\]
\[
|\nabla_g(f(u)-f(v))| \le C \left( |u|^4+|v|^4 \right)|\nabla_g ( u-v )| +C \left(|u|^3+|v|^3 \right)(|\nabla_g u|+|\nabla_g v|)| u-v |.
\]
Using the Gagliardo-Nirenberg and H\"older inequalities, we obtain
\[
\|f(u)-f(v)\|_{L_t^1H_x^1} \le C\left(  (2A)^2R^2+(2A)^{ \frac52}R^{ \frac32} \right)d(u,v).
\]
Indeed, the first gradient term is bounded by
$C(2A)^2R^2\|w\|_{L_t^\infty H_x^1}$; in the second term one uses the interpolation inequality.
The linear estimate \eqref{eq:smoothing} and Lemma \ref{le1.1v156} therefore imply
\[
d(\mathcal T(u),\mathcal T(v)) \le C\left((2A)^2R^2+(2A)^{ \frac52}R^{ \frac32}\right)d(u,v).
\]
Taking $T$ smaller makes the coefficient less than $ \frac12$.
Thus $\mathcal T$ is a contraction on $B_R$.
The Banach fixed-point theorem therefore yields a unique local solution
\[
u \in C_t^0 H_x^1  \left([0,T] \times \mathcal M^3   \right) \cap L^2_t H_x^2\left([ 0,T] \times   \mathcal M^3  \right).
\]
Applying the same difference estimate on short subintervals, using the
absolute continuity of the $L_t^2H_x^2$ norm, gives uniqueness in the
whole stated solution class. Including the linear initial-data difference
also gives continuous dependence in this class.
We now prove the blowup criterion.

Suppose that $T_{\max}<\infty$ but
$\|u\|_{L_t^2H_x^2 \left([0,T_{\max} [ \times\mathcal M^3 \right)}<\infty$.
The mass and energy dissipation identities \eqref{eq_ener} and \eqref{eq_mass} yield a uniform $H_x^1$ bound.
By the Sobolev and interpolation inequalities, we have
\[
\|u\|_{L_{t,x}^{10} \left([0,T_{\max} [ \times\mathcal M^3 \right)}
\lesssim
\|u\|_{L_t^\infty H_x^1}^{ \frac45}   \|u\|_{L_t^2H_x^2}^{ \frac15}<\infty.
\]
Thus $f(u),u_t=z(\Delta_gu-f(u))\in L_{t,x} ^2 $.
The embedding
\[
W^{1,2}_t L_x^2  \left( [ 0,T_{\max} [  \times \mathcal{M}^3  \right)\cap L^2_t H_x^2 \left( [ 0,T_{\max} [  \times \mathcal{M}^3  \right)
\hookrightarrow C_t^0 H_x^1  \left([0,T_{\max} [    \times \mathcal{M}^3  \right)
\]
supplies an $H_x^1$ limit at $T_{\max}$. Local existence at that limit extends $u$ beyond $T_{\max}$, a contradiction.
Therefore
\[
T_{\max}<\infty\Longrightarrow
\|u\|_{L_t^2H_x^2 \left([0,T_{\max} [ \times\mathcal M^3 \right)}=\infty.
\]
(2) We now consider the energy-subcritical case $1<p<5$.
The argument is analogous to part (1), so we only record the estimate which provides the required smallness in time.
For $0<T\le1$, set
\[
Y_T=C_t^0H_x^1 \left([0,T]\times\mathcal M^3 \right)\cap L_t^2H_x^2 \left([0,T]\times\mathcal M^3 \right),
\qquad
\|u\|_{Y_T}=\|u\|_{L_t^\infty H_x^1}+\|u\|_{L_t^2H_x^2}.
\]
Let
\[
\alpha=\max\left\{0,\frac{p-3}{2p}\right\},
\qquad
\beta=\frac{1-p\alpha}{2}>0.
\]
By the Sobolev, interpolation and H\"older inequalities, we get
\[
\|v\|_{L_x^{2p}}\lesssim \|v\|_{H_x^1}^{1-\alpha}\|v\|_{H_x^2}^{\alpha},
\qquad
\|v\|_{L_t^{2p}L_x^{2p}}\lesssim T^{\frac\beta p}\|v\|_{Y_T}.
\]
Hence,
\[
\|f(u)\|_{L_{t,x}^2}\lesssim T^\beta\|u\|_{Y_T}^p,
\]
and
\[
\|f(u)-f(v)\|_{L_{t,x}^2}
\lesssim T^\beta\left(\|u\|_{Y_T}^{p-1}+\|v\|_{Y_T}^{p-1}\right)\|u-v\|_{Y_T}.
\]
Therefore, applying the linear estimate established in part (1) to the same map $\mathcal{T}$ defined in \eqref{eq2.13v245}, and setting $A=\max\{1,\|u_0\|_{H_x^1}\}$, we obtain a contraction on
\[
X_A=\{u\in Y_T:\ \|u\|_{Y_T}\le2C_0A\}
\]
provided $T=T(A)>0$ is chosen sufficiently small so that
\[
CT^\beta(2C_0A)^{p-1}\le \frac12.
\]
This yields local well-posedness in $Y_T$, with an existence time depending only on the $H_x^1$ norm of the initial data; uniqueness and continuous dependence follow from the same difference estimate as in part (1).

It remains to prove the blowup criterion. If $T_{\max}<\infty$ and
\[
\sup_{0<t<T_{\max}}\|u(t)\|_{H_x^1}\le A<\infty,
\]
then the local existence time furnished above is bounded from below by some $\delta=\delta(A)>0$, independently of the initial time. Choosing $T\in]T_{\max}-\delta,T_{\max}[$ and restarting the equation at time $T$ extends the solution to $[T,T+\delta]$, hence beyond $T_{\max}$, a contradiction. Therefore,
\[
T_{\max}<\infty\Longrightarrow
\sup_{0<t<T_{\max}}\|u(t)\|_{H_x^1 \left(\mathcal M^3 \right)}=\infty.
\]
This completes the proof.
\end{proof}

\subsubsection{Global well-posedness}

For the energy-critical case, by Lemma \ref{lem:Blowupcriterion},
it suffices to rule out concentration of the $L_t^2 H_x^2$-norm of the solution $u$. The core ingredient is the following elliptic estimate:
\begin{lemma}[Elliptic estimate]\label{lem:monotonicity}
Let $v\in H_x^2 \left(\mathcal{M}^3 \right)$ satisfy the elliptic equation
\begin{align}\label{eq2.9v127}
 -\Delta_g v + f(v) = h,
\end{align}
where $h\in L_x^2 \left(\mathcal{M}^3 \right)$.
Then we have the following bound:
\begin{align}
\|\Delta_g v\|_{L_x^2 \left(\mathcal{M}^3 \right)}^2+ \left\|  v \right\|_{L^{2p}_x \left(\mathcal{M}^3 \right)}^{2p}
\le \|h\|_{L_x^2 \left(\mathcal{M}^3 \right)}^2.
\end{align}
\end{lemma}
\begin{proof}
Multiplying \eqref{eq2.9v127} by $-\Delta_g \overline{v}$, integrating over $\mathcal{M}^3$, and taking real parts yields
\begin{align*}
\|\Delta_g v\|_{L_x^2 \left(\mathcal{M}^3 \right)}^2
+ \Re\int_{\mathcal{M}^3} \left\langle \nabla_g f(v), \nabla_g \overline{v} \right\rangle_g \,\mathrm{d}\mu
= -\Re\int_{\mathcal{M}^3} h \Delta_g \overline{v} \,\mathrm{d}\mu.
\end{align*}
A direct computation gives
\begin{align*}
\Re\left\langle \nabla_g f(v), \nabla_g \overline{v} \right\rangle_g
= |v|^{p-1} |\nabla_g v|^2
+ (p-1) |v|^{p-3} \, \Re \left\langle v \nabla_g \overline{v}, \Re \left(v \nabla_g \overline{v} \right) \right\rangle_g .
\end{align*}
Since
\[
\Re \left\langle v \nabla_g \overline{v}, \Re \left(v \nabla_g \overline{v} \right) \right\rangle_g
= \left| \Re \left(v \nabla_g \overline{v} \right) \right|_g^2 \ge 0,
\]
it follows that
\[
\Re\int_{\mathcal{M}^3} \left\langle \nabla_g f(v), \nabla_g \overline{v} \right\rangle_g \,\mathrm{d}\mu \ge 0.
\]
Hence,
\begin{align*}
\|\Delta_g v\|_{L_x^2 \left(\mathcal{M}^3 \right)}^2
\le -\Re\int_{\mathcal{M}^3} h  \Delta_g \overline{v} \,\mathrm{d}\mu
\le \frac{1}{2} \|\Delta_g v\|_{L_x^2 \left(\mathcal{M}^3 \right)}^2 + \frac{1}{2} \|h\|_{L_x^2 \left(\mathcal{M}^3 \right)}^2.
\end{align*}
Absorbing the first term on the right hand side gives the bound for $\Delta_gv$.
The same nonnegative cross term also gives the stronger bound
stated above:
\[
\begin{aligned}
\|h\|_{L^2_x }^2 =\|\Delta_gv\|_{L^2_x }^2+  \left\| f(v) \right\|_{L_x^2}^2
 +2\Re\int_{\mathcal M^3}(-\Delta_g\bar v) f(v)\,\mathrm d \mu \ge\|\Delta_gv\|_{L_x^2}^2+ \left\|   v \right\|_{L_x^{2p}}^{2p} .
\end{aligned}
\]
\end{proof}

We are now in a position to prove
\begin{proposition}[Global well-posedness in $H_{x}^1$]\label{prop:global}
Let $$u\in C_t^0H_x^1\left( \left[0,T_{\mathrm{max}} \right[ \times\mathcal{M}^3 \right)\cap L_t^2H_x^2\left( \left[0,T_{\mathrm{max}} \right[ \times\mathcal{M}^3 \right)$$
be the maximal-lifespan solution of equation \eqref{eq1.CGL}. Then
the maximal lifespan satisfies  $T_{\mathrm{max}}=+\infty$.
\end{proposition}

\begin{proof}

(1) $p = 5$.
In view of Lemma \ref{lem:Blowupcriterion},
it suffices to verify that $\|u\|_{L_t^2H_x^2 \left([0,T [ \times  \mathcal{M}^3  \right)}$ remains bounded (with a bound depending on $T$) for every $T<T_{\max}$. From \eqref{eq_ener}, we obtain
\begin{align}\label{eq2.15v246}
\|u\|_{L_{t,x}^2  \left([0,T]\times  \mathcal{M}^3  \right)}\leq \|u_0\|_{L_x^2}T^{ \frac12}.
\end{align}
Meanwhile, \eqref{eq_mass} yields
\begin{align}\label{eq2.16v246}
\frac{\Re z}{|z|^2}\left\|\partial_tu\right\|_{L_{t, x}^2 \left([0,T]\times  \mathcal{M}^3  \right)}^2\leq E\left(u(0)\right).
\end{align}
By Lemma \ref{lem:monotonicity}, we have
\begin{align}\label{eq2.17v246}
 \|\Delta_gu\|_{L_{t, x }^2  \left([0,T]\times  \mathcal{M}^3  \right)}
\lesssim
\|\partial_tu\|_{L_{t,x}^2 \left([0,T]\times  \mathcal{M}^3  \right)}.
\end{align}
By \eqref{eq2.15v246}, \eqref{eq2.16v246}, and \eqref{eq2.17v246}, we deduce
\begin{align*}
\|u\|_{L_t^2H_x^2
\left([0,T]\times  \mathcal{M}^3  \right)}\leq C \left(\|u_0\|_{H_x^1} \right)\left(1+T^{ \frac12} \right)<\infty.
\end{align*}
As a consequence, we conclude $T_{\mathrm{max}}=+\infty$.

(2) $1 < p < 5$.
In this case, the uniform $H_x^1$ bound following from \eqref{eq_ener}--\eqref{eq_mass}, together with Lemma~\ref{lem:Blowupcriterion}, gives $T_{\mathrm{max}}=+\infty$.
\end{proof}

\begin{theorem}[Smoothing effect]\label{th2.6v107}
For any initial data $u_0 \in H_x^1 \left(\mathcal{M}^3 \right)$ and $t_0 > 0$,
the solution $u$ to the equation \eqref{eq1.CGL} with $p\in ]1,5]$ satisfies
\[
u \in C_t^0 H_x^2 \left([t_0, \infty[ \times \mathcal{M}^3 \right).
\]
Moreover, for any $T>t_0$,
$$ \sup_{t_0\leq t\leq T}\|u(t)\|_{H_x^2}\leq C \left(T, t_0, \|u_0\|_{H_x^1} \right).
$$
\end{theorem}

\begin{proof}
Fix \(t_0>0\) and \(T>t_0\).
 By Proposition~\ref{prop:global}, Sobolev and
interpolation, we deduce that
\[
  \partial_tu=z\bigl(\Delta_gu-f(u)\bigr) \in L^2_{t,x}\bigl([0,T+1]\times\mathcal M^3\bigr).
\]
Choose  \(t_1\in  ] 0,t_0 [ \) such that \(t_1\) is a Lebesgue point of
\(\partial_tu\) in \(L^2\) and \(u(t_1)\in H^2\).

By Duhamel's formula, we have
\begin{align*}
u(t)=e^{(t-t_1)z\Delta_g}u(t_1)-z\int_{t_1}^t e^{(t-\tau)z\Delta_g}f(u(\tau))\, \dd\tau.
\end{align*}
Since $u(t_1)\in H_x^2$, it suffices to show that $\int_{t_1}^t e^{(t-\tau)z\Delta_g}f(u(\tau)) \, \dd\tau  $ is bounded in $H_x^2$.
We distinguish two cases:

\noindent
$\bullet${\bf Case 1: $1<p<3$}

In this case, we apply
 Lemma \ref{le1.1v156} to get
\begin{align*}
\Big\|\int_{t_1}^t e^{(t-\tau)z\Delta_g}f(u(\tau)) \, \dd\tau \Big\|_{H_x^2}
\lesssim &  \int_{t_1}^t  \left(1+(t-\tau)^{-\frac{1}{2}} \right)\|f(u(\tau))\|_{H_x^1} \, \dd\tau.
\end{align*}
Since
$$ \|f(u)\|_{H_x^1}\lesssim \|u\|_{L_x^{\infty}}^{p-1}\|u\|_{H_x^1},
$$
by the H\"older inequality, we obtain that
\begin{align*}
\left\|\int_{t_1}^t e^{(t-\tau)z\Delta_g}f(u(\tau)) \, \dd\tau \right\|_{H_x^2}
\lesssim &  \left\|1+|t-\tau|^{-\frac{1}{2}} \right\|_{L_{\tau}^{\frac{4}{5-p}} \left([0,t] \times \mathcal{M}^3 \right) }
 \|u \|_{L_{\tau}^4L_x^{\infty} \left([0,t] \times \mathcal{M}^3 \right)}^{p-1} \|u \|_{L_{\tau}^{\infty}H_x^1 \left([0,t] \times \mathcal{M}^3 \right) }.
\end{align*}
Thanks to the fact that $p<3$, we have  $|t-\tau|^{- \frac12}\in L^{\frac{4}{5-p}}([0,t])$.

Next, from the interpolation inequality,
we have
\begin{align*}
\|u(\tau)\|_{L_{\tau}^4L_x^{\infty} \left([0,t] \times \mathcal{M}^3 \right)}^4\lesssim \|u \|_{L_{\tau}^{\infty}H_x^1 \left([0,t] \times \mathcal{M}^3 \right)}^2 \|u \|_{L_{\tau}^2H_x^2 \left([0,t] \times \mathcal{M}^3 \right) }^2.
\end{align*}
This implies that, for any $t>t_1$,
\begin{align*}
\Big\|\int_{t_1}^t e^{(t-\tau)z\Delta_g}f(u(\tau)) \, \dd\tau \Big\|_{H_x^2}
\leq C(p,t)\|u \|_{L_{\tau}^{\infty}H_x^1 \left([0,t] \times \mathcal{M}^3 \right) }^{\frac{p+1}{2}}\|u \|_{L_{\tau}^2H_x^2 \left([0,t] \times \mathcal{M}^3 \right)}^{\frac{p-1}{2}}.
\end{align*}

\noindent
$\bullet${\bf Case 2: $3\leq p\leq 5$}

In this case, we estimate straightforwardly
\begin{align*}
\Big\|\int_{t_1}^t e^{(t-\tau)z\Delta_g}f(u(\tau)) \, \dd\tau \Big\|_{H_x^2}\leq \int_{t_1}^t\|f(u(\tau)) \|_{H_x^2} \, \dd\tau.
\end{align*}
By the Gagliardo-Nirenberg inequality, we have
$$ \|f(u)\|_{H_x^2}\lesssim \|u\|_{L_x^{\infty}}^{p-1}\|u\|_{H_x^2}.
$$
Thus, we obtain the inequality
\begin{align*}
\|u(t)\|_{H_x^2}\leq \|u(t_1)\|_{H_x^2}+C\int_{t_1}^t \|u(\tau)\|_{L_x^{\infty}}^{p-1}\|u(\tau)\|_{H_x^2} \, \dd \tau.
\end{align*}
By the integral form of the Gr\"onwall inequality, we deduce that
\begin{align*}
\|u(t)\|_{H_x^2}\leq \|u(t_1)\|_{H_x^2}+C\|u(t_1)\|_{H_x^2}\int_{t_1}^t \|u(\tau)\|_{L_{x}^{\infty}}^{p-1}\exp\left(\int_{t_1}^t \|u(t')\|_{L_x^{\infty}}^{p-1} \, \dd t' \right)\dd \tau.
\end{align*}
Using the fact that
\begin{align*}
 \int_{t_1}^t\|u(\tau)\|_{L_x^{\infty}}^{p-1} \, \dd\tau
 \leq &C(p,t) \|u \|_{L_{\tau}^4L_{x}^{\infty}\left([0,t] \times \mathcal{M}^3 \right)}^{p-1}
 \leq  C(p,t)\|u \|_{L_{\tau}^{\infty}H_x^1  \left([0,t] \times \mathcal{M}^3 \right)}^{\frac{p - 1}{2}}
 \|u \|_{L_{\tau}^2H_x^2 \left([0,t] \times \mathcal{M}^3 \right)}^{\frac{p-1}{2}}
 \end{align*}
and Proposition \ref{pr2.12v70},
we complete the proof of Theorem \ref{th2.6v107}.
\end{proof}

\subsection{Existence and uniqueness of global weak solutions to defocusing NLS on general three-dimensional Riemannian manifolds}

In this subsection, we construct global weak solutions for $1<p\le5$ and prove weak-strong uniqueness in Theorem~\ref{th1.2v21} on three-dimensional Riemannian manifolds satisfying \eqref{eq1.5v177}.
Readers who are only interested in the CGL can skip this subsection and go directly to Section \ref{se3}.

\subsubsection{Existence of global weak solutions}

In this subsection, we show that the energy-subcritical and energy-critical defocusing nonlinear Schr\"odinger equations on arbitrary three-dimensional Riemannian manifolds satisfying \eqref{eq1.5v177} admit global weak solutions. To achieve this, we investigate the inviscid limit of solutions to the associated complex Ginzburg-Landau equation, an argument that relies heavily on uniform boundedness of CGL solutions.

We first provide an equivalent characterization of weak-${H}_{x}^1$
 solutions to the nonlinear Schr\"odinger equation \eqref{eq1.1v18}; a proof of this equivalence can be found, for example, in \cite{Ca03}.

 	\begin{definition}
A complex-valued function $v$ on a time interval $I \subseteq \mathbb{R}$ is a weak-$H_{x }^1$ solution to \eqref{eq1.1v18} if $v\in L^\infty_t H_x^{1}\left(I \times \mathcal{M}^3 \right)$, $\partial_tv \in L^\infty_t H_x^{-1}\left(I \times \mathcal{M}^3 \right)$, and \eqref{eq1.1v18} holds for almost every $t\in I$ in $H^{-1}_x$.
 	\end{definition}

 	\begin{proof}[Proof of Theorem \ref{th1.2v21} $1)$ ]
 		By Proposition \ref{prop:global}, there are global solutions $u_n$ of
 		\begin{align}\label{eq7.1v24}
 			\begin{cases}
 				\partial_t u_n - z_n \Delta_g u_n +  z_n f(u_n) = 0, \\
 				u_n(0) = v_0,
 			\end{cases}
 		\end{align}
where $z_n \in \mathbb{C}$, $\Re z_n > 0$, $|z_n | = 1$, and $z_n\to i$ as $n \to \infty$.
From Proposition \ref{prop:global}, we have $u_n \in L^\infty_t  H^1_x $ and $\pa_tu_n\in L^\infty_t H^{-1}_x$,
 with the uniform bound
 		\begin{align*}
 			\|u_n \|_{L^\infty_t H^1_x} + \left\|\pa_tu_n \right\|_{L^\infty_t H^{-1}_x} \le C \left(\|v_0\|_{H^1_x }\right).
 		\end{align*}
By Rellich's compactness theorem, this
yields compactness of $\{ u_n \}_{n = 0}^\infty $ in $C_t^0 L^2_{x,\mathrm{loc}}  \left([0,T] \times \mathcal{M}^3 \right)$.
To control the tails,
take a Lipschitz cutoff $\chi_R$ equal to zero on $B_R$ and one outside $B_{2R}$ with $|\nabla_g\chi_R|\le  \frac{ C}R$.
The localized mass dissipation identity gives
\[
\sup_{0\le t\le T}\int\chi_R |u_n(t)|^2\,\mathrm d\mu \le\int\chi_R|v_0|^2\,\mathrm d\mu+  \frac{ CT}R,
\]
uniformly in $n$. Hence, there is $v\in L_t^\infty H_x^1$ such that
\begin{align}\label{eq:weak_star}
u_n\to v \text{ in } C_t^0 L_x^2  \left([0,T] \times \mathcal{M}^3  \right)\text{ for every } T<\infty,
\quad u_n(t)\rightharpoonup v(t)\text{ in }H_x^1\text{ for every }t\ge 0.
\end{align}
After passing to a further subsequence, $u_n\to v$ almost everywhere in spacetime.
Multiplying both sides of \eqref{eq7.1v24} by $\varphi \in C_c^\infty \left(\mathcal{M}^3 \right)$ and $\psi \in  C^\infty_c ( \mathbb{R}_+)$, we get
 		\begin{align*}
 			-\int_0^\infty \left( \int_{\mathcal{M}^3 } \bar{z}_n   u_n  \varphi  \,\mathrm{d}{ \mu} \right) \psi_t \, \mathrm{d}t
 + \int_0^\infty \left(\int_{\mathcal{M}^3 }
 			\left\langle \nabla_g u_n  ,  \nabla_g  \varphi \right\rangle_g  \,\mathrm{d}{ \mu}\right) \psi \,  \mathrm{d}t = - \int_0^\infty \left(\int_{\mathcal{M}^3 } f(u_n)     \varphi \,\mathrm{d}{ \mu}\right) \psi \, \mathrm{d}t .
 		\end{align*}
 		Taking $n\to \infty$ and using the weak convergence in \eqref{eq:weak_star}, we obtain
 		\begin{align*}
 			\int_0^\infty \left(\int_{\mathcal{M}^3 } \bar{z}_n   u_n  \varphi \,\mathrm{d}{ \mu}\right) \psi_t \, \mathrm{d}t
 			\xrightarrow{n\to \infty} - \int_0^\infty \left(\int_{\mathcal{M}^3 } i v \varphi \,\mathrm{d}{ \mu} \right) \psi_t \, \mathrm{d}t,
 		\end{align*}
 		and
 		\begin{align*}
 			\int_0^\infty \left(\int_{\mathcal{M}^3 }
 			\langle \nabla_g  u_n  , \nabla_g  \varphi \rangle_g  \,\mathrm{d}{ \mu} \right) \psi \, \mathrm{d}t \xrightarrow{n\to \infty} \int_0^\infty \left(\int_{\mathcal{M}^3 }
 			\left\langle \nabla_g  v ,  \nabla_g \varphi \right\rangle_g  \,\mathrm{d}{ \mu} \right) \psi \, \mathrm{d}t.
 		\end{align*}
 Note that $  f(u_n) \varphi\psi$ is compactly supported and thus is bounded in $L^\infty_t L^{ \frac{p+1}p }_x$. Then by the almost everywhere convergence in \eqref{eq:weak_star}, we have that $  f(u_n) \varphi \psi \to f(v)   \varphi \psi$ for almost every $(t, x)$,
 which  guarantees that $    f(u_n)  \varphi \psi \to   f(v)  \varphi \psi$ in $L^1_{t,x}$ as $n\to \infty$.

Consequently, by taking $n \to \infty $, we obtain
 		\begin{align*}
 			\int_0^\infty\int_{\mathcal{M}^3 } i v \varphi \psi_t \,\mathrm{d}{ \mu} \mathrm{d}t + \int_0^\infty\int_{\mathcal{M}^3 }
 			\left\langle \nabla_g v ,  \nabla_g  \varphi \right\rangle_g  \psi\,\mathrm{d}{ \mu} \mathrm{d}t
 = -  \int_0^\infty\int_{\mathcal{M}^3 }    f(v) \varphi \psi\,\mathrm{d}{ \mu} \mathrm{d}t .
 		\end{align*}
Hence $v$ satisfies
 		\begin{equation}\label{eq7.2v87}
 			iv_t +  \Delta_g v =  f(v)
 		\end{equation}
in the sense of distributions.

By the lower semicontinuity of the $L^2$-norm and the weak convergence in \eqref{eq:weak_star}, we have, for every $t\in\mathbb{R}_+$,
\begin{align}\label{est:normdecrease_1}
 			\|\nabla_g  v(t)\|_{L_x^2} \le \liminf_{n \to\infty}  \left\|\nabla_g  u_{n } (t) \right\|_{L_x^2} .
 		\end{align}
For each fixed $t$, the same $H^1$ bound and strong $L^2$ convergence imply weak convergence in $L^{p+1}$.
By Fatou's lemma, we have
\begin{align}\label{est:normdecrease_2}
\int_{\mathcal M^3}|v(t)|^{p+1}\,\mathrm d\mu
\le\liminf_{n\to\infty}\int_{\mathcal M^3}|u_n(t)|^{p+1}\,\mathrm d\mu.
\end{align}
Combining \eqref{est:normdecrease_1} and \eqref{est:normdecrease_2}, we have
 		\begin{align*}
 			\begin{aligned}
 				E(v(t)) \le \liminf_{n \to\infty} E \left(u_{n }(t) \right) \le E(v_0).
 			\end{aligned}
 		\end{align*}
 		On each $[0,T]$, the bounds $v\in L_t^2H_x^1$ and $\partial_tv\in L_t^2H_x^{-1}$ give
 		\begin{align*}
 \|v(t)\|_{L_x^2}^2 = \|v(s)\|_{L_x^2}^2 + 2\Re\int_s^t \left< \pa_\tau v (\tau), v(\tau)\right>_{H_x^{-1}, H_x^1} \,\mathrm{d}\tau, \quad 0\le s \le t \le T,
 		\end{align*}
which implies $v \in C_t^0 L_x^2 \left([0, T] \times \mathcal{M}^3\right)$.

Multiplying both sides of \eqref{eq7.2v87} by $i\bar{v}$ and integrating over $\mathcal{M}^3$, we obtain the mass conservation.
Thus, we have obtained a global weak solution of \eqref{eq1.1v18} on $ \mathbb{R}_+ $.
Apply the same construction forward in time to the datum $\overline{v}_0 $.
If its solution is $v_+$, define $v(t)=\overline{v_+(-t)}$ for $t<0$.
This solves the NLS equation backward in time, matches $v_0$ at zero, and has the same mass and energy bounds.
 	\end{proof}

\subsubsection{Weak-strong uniqueness}
In this subsection, we show that a weak solution coincides with a strong
solution from the same initial datum, under the regularity assumptions
in Theorem~\ref{th1.2v21}.

\begin{proof}[Proof of Theorem \ref{th1.2v21} 2)]
We first consider $1<p\le2$.
Write $w=v-\tilde v$. The power nonlinearity satisfies
\[
\left|\Im\left(( f(\eta+w)-f(\eta))\bar w\right)\right| \le C_p|\eta|^{p-1}|w|^2.
\]
Indeed, if $|w|\le  \frac{ |\eta|}2$, this follows from the local derivative bound for $f$.
If $|w|> \frac{ |\eta|}2$, subtract the real quantity $f(w)\bar w$.
The remaining terms have absolute value at most $C_p \left( |\eta||w|^p+|\eta|^p|w| \right)$,
which is bounded by the displayed right hand side because $p\le2$ and $|\eta|<2|w|$.
By direct computation, we have
\[
\frac{\mathrm d}{\mathrm dt} \left( \|w(t)\|_{L_x^2}^2  \right)
\le C_p  \left\|\tilde v(t) \right\|_{L_x^\infty}^{p-1}\|w(t)\|_{L_x^2}^2.
\]
Since $\tilde v\in C_t^0 H_x^2\hookrightarrow C_t^0 L_x^\infty$,
by Gr\"onwall's inequality and $w(0)=0$, uniqueness holds on every bounded forward interval.
Negative times follow by conjugate time reversal.

For the remaining powers we prove the following stability result,
which directly implies weak-strong uniqueness.
\begin{lemma}\label{thm:stability}
Let $2\le p\le5$.
Suppose that $\tilde v\in C_t^0H_x^3 \left(\mathbb R\times\mathcal M^3 \right)$ is a strong solution to \eqref{eq1.1v18}.
Let $v$ be a global weak solution with initial datum $v_0$, satisfying mass conservation
\begin{align}\label{eq2.22v246}
M(v(t)) = M(v_0)
\end{align}
and the energy inequality
\begin{align}\label{eq2.23v246}
E(v(t))\le E(v_0),\quad t\in\mathbb R.
\end{align}
Fix $T>0$ and $R>0$, and assume $\|v_0\|_{H^1}\le R$.
Then, with $w=v-\tilde v$, there is
$C=C(\tilde v,T,R)>0$ such that
\[
\|w(t)\|_{H_x^1}\le C\|w(0)\|_{H_x^1},\qquad |t|\le T.
\]
\end{lemma}
\begin{proof}
Without loss of generality, we consider forward times $t\ge0$.
The difference $w$ satisfies
\[
\begin{cases}
i \partial_t w +\Delta_gw=-f \left(\tilde v \right)+f \left(w+\tilde v \right),\\
w(0)=v_0-\tilde v_0.
\end{cases}
\]
We can expand $E(v)$ as
\begin{align}\label{eq:decompose}
E(v)=E \left(\tilde v \right)+I(t)+II(t),
\end{align}
where
\[
I(t)=\Re\int_{\mathcal M^3}
\left(\left\langle\nabla_g\tilde v,\overline{\nabla_gw}\right\rangle_g
+f \left(\tilde v \right)\bar w\right)\,\mathrm d\mu,
\]
and
\[
II(t)=\int_{\mathcal M^3}\left(\frac12 |\nabla_gw|^2
+F \left(w+\tilde v \right)-F \left(\tilde v \right)-\Re \left(f \left(\tilde v \right)\bar w \right)\right)\,\mathrm d\mu
\text{ with }  F( u )=\frac{| u |^{p+1}}{p+1}.
\]
By \eqref{eq2.22v246}, \eqref{eq2.23v246},
and Sobolev embedding, we have a uniform $H^1$ bound for $v$ and $w$ on $[0,T]$,
depending only on $R$ and $\tilde v$. All constants below may depend on these bounds.

Notice first that
\begin{align*}
\frac{\mathrm d}{\mathrm dt} \left( \|w (t) \|_{L_x^2}^2 \right)
=2\Re\int_{\mathcal M^3}
\overline{ \left( f \left(\tilde v+w \right)-f \left(\tilde v \right) \right) }\,iw\,\mathrm d\mu
 \le C_1\int_{\mathcal M^3} \left(|w|^2+F(w) \right)\,\mathrm d\mu
\le C_1'\|w\|_{L_x^2}^2+C_2\|\nabla_gw\|_{L_x^2}^2.
\end{align*}
By Gr\"onwall's inequality, we obtain
\begin{align}\label{est:w_l2norm}
\|w(t)\|_{L_x^2}^2\le e^{C_1't}\|w(0)\|_{L_x^2}^2 +C_2\int_0^t e^{C_1'(t- \tau )} \|\nabla_gw( \tau )\|_{L_x^2}^2\,\mathrm d \tau .
\end{align}
To justify the estimate for $I$, we write
\[
I(t)=\Re\int_{\mathcal M^3}
\left(-\Delta_g\tilde v+f \left(\tilde v \right)\right)\overline{w}\,\mathrm d\mu.
\]
By a simple calculation, we get
\[
\begin{aligned}
I'(t)=-\Re\int_{\mathcal M^3}\tilde v_t\,
\Bigl(
&\overline{f \left(\tilde v+w \right)-f \left(\tilde v \right)}
-\left|\tilde v\right|^{p-1}\overline{w} -(p-1)\left|\tilde v\right|^{p-3}\overline{\tilde v}\,
\Re\left(\overline{\tilde v}\,w\right)
\Bigr)\,\mathrm d\mu.
\end{aligned}
\]
For $2\le p\le5$, the power nonlinearity satisfies
\[
\begin{aligned}
\Bigl|f \left(\tilde v+w \right)-f \left(\tilde v \right)
&-\left|\tilde v\right|^{p-1}w
-(p-1)\left|\tilde v\right|^{p-3}\tilde v\,
\Re\left(\overline{\tilde v}\,w\right)\Bigr| \le C_p\left( \left|\tilde v \right|^{p-2}|w|^2+|w|^p\right).
\end{aligned}
\]
Since $\tilde v_t\in C^0_t H_x^1 $ and $\tilde v\in C_t^0 L_x^\infty$, the H\"older and Sobolev inequalities give
\[
\left|I'(t) \right|\lesssim \|w\|_{ L_x^{ \frac{ 12}5}}^2+\|w\|_{  L_x^{ \frac{ 6p}5}}^p  \lesssim  \|w\|_{H_x^1}^2.
\]
Consequently, using \eqref{est:w_l2norm}, we obtain
\begin{align}\label{est:i_0-i_t}
\begin{aligned}
\left| I(0)-I(t) \right|
\lesssim \int_0^t\|w( \tau )\|_{H_x^1}^2\,\mathrm d \tau
\lesssim  \|w(0)\|_{L_x^2}^2+ \int_0^t\|\nabla_gw( \tau )\|_{L_x^2}^2\,\mathrm d \tau .
\end{aligned}
\end{align}
Since $E \left(\tilde v(t) \right)=E \left(\tilde v_0 \right)$ and $E(v(t))\le E(v_0)$,
the expansion \eqref{eq:decompose} implies
\begin{align}\label{est:v_t}
0\le E(v_0)-E(v(t))=I(0)-I(t)+II(0)-II(t).
\end{align}
The uniform convexity of $F$ gives
\begin{align}\label{est:ii_t}
II(t)\ge\frac12\|\nabla_gw(t)\|_{L_x^2}^2 +c_p\|w(t)\|_{ L_x^{ p+1 } }^{p+1}\ge0,
\end{align}
where $c_p$ is some fixed positive constant.
Meanwhile, Taylor's formula, the boundedness of $\tilde v_0$, and Sobolev embedding give
\begin{align}\label{est:ii_0}
\begin{aligned}
II(0) \lesssim \|\nabla_gw(0)\|_{L_x^2}^2+\|w(0)\|_{L_x^2}^2 +\|w(0)\|_{ L_x^{p+1 } }^{p+1} \lesssim_{\tilde{v}_0 }
  \|w(0)\|_{H_x^1}^2.
\end{aligned}
\end{align}
Combining \eqref{est:w_l2norm}, \eqref{est:i_0-i_t}, \eqref{est:v_t}, \eqref{est:ii_t}, and \eqref{est:ii_0}, we conclude
\[
\begin{aligned}
\|\nabla_g w(t)\|_{L_x^2}^2+\|w(t)\|_{L_x^2}^2
& \lesssim I(0)-I(t)+II(0)+\|w(t)\|_{L_x^2}^2 \\
& \lesssim  \|w(0)\|_{H^1_x }^2 + \int_0^t\left(\|\nabla_gw( \tau )\|_{L_x^2}^2+\|w( \tau )\|_{L_x^2}^2\right)\,\mathrm d \tau .
\end{aligned}
\]

By Gr\"onwall's inequality, we get the claimed estimate.
Applying the same argument to $\overline{v(-t)}$ and $\overline{\tilde v(-t)}$ gives the estimate at negative times.
\end{proof}

Taking $v_0=\tilde v_0$ in
Lemma \ref{thm:stability} proves weak-strong
uniqueness for $2<p\le5$ with an $H^3$ comparison solution.
Together with the first part, this completes the proof.
\end{proof}

\section{Decay estimates and higher-order energy analysis}\label{se3}

In this section, we derive quantitative decay estimates for the fluctuation component of the solution and establish the higher-order regularity required for the proof of Theorem \ref{th2v79}.
The central ingredient is a modified $H_x^2$-energy functional that captures the dissipative structure of the complex Ginzburg-Landau equation on $\mathcal M^3$. This functional allows us to overcome the lack of coercivity of the standard $H_x^2$-norm and provides a unified framework for deriving both polynomial and exponential decay estimates.

\subsection{The decay of the mass and energy functionals}\label{decay_ME}

From Proposition \ref{pr2.12v70}, we have
\begin{align}\label{eq3.1v118}
\frac{\mathrm d}{\mathrm dt}M\left(u\left(t\right)\right)=-2\Re z E_1\left(u\left(t\right)\right),
\end{align}
where
\begin{align*}
E_1\left(u\left(t\right)\right):=\|\nabla_gu\left(t\right)\|_{L_x^2}^2+\|u\left(t\right)\|_{L_x^{p+1}}^{p+1} ,
\end{align*}
and
\begin{align}\label{eq3.2v118}
\frac{\mathrm d}{\mathrm dt}E\left(u\left(t\right)\right)=-\frac{\Re z}{|z|^2}\|\partial_tu\left(t\right)\|_{L_x^2}^2.
\end{align}
Since the manifold is compact, the Cauchy-Schwarz inequality gives rise to
\begin{align*}
\|u\left(t\right)\|_{L_x^2}\lesssim \|u\left(t\right)\|_{L_x^{p+1}}.
\end{align*}
We then deduce from the above inequality and \eqref{eq3.1v118} that
\[
\frac{\mathrm d}{\mathrm dt} M\left(u\left(t\right)\right)\le -CM\left(u\left(t\right)\right)^{\frac{p+1}{2}},
\]
where $C$ is a large positive constant.

Hence,
\begin{align}\label{inq_fprime}
\frac{\mathrm d}{\mathrm dt} \left( M\left(u\left(t\right)\right)^{-\frac{p-1}{2}} \right)
=-\frac{p-1}{2}M\left(u\left(t\right)\right)^{-\frac{p+1}{2}}\frac{\mathrm d}{\mathrm dt}M\left(u\left(t\right)\right)\ge c>0,
\end{align}
where $c$ is a positive constant.

Integrating \eqref{inq_fprime} with respect to time, we obtain
\[
M\left(u\left(t\right)\right)^{-\frac{p-1}{2}}\ge M\left(u\left(0\right)\right)^{-\frac{p-1}{2}}+ct,
\]
and therefore we get the preliminary decay estimate
\begin{equation}\label{eq:initial-mass-decay}
M\left(u\left(t\right)\right)\lesssim_{\|u_0\|_{L^2}}\left(1+t\right)^{-\frac{2}{p-1}}.
\end{equation}
Together with \eqref{eq3.1v118}, this implies that
\[
M\left(u\left(t\right)\right)=2\Re z\int_t^{+\infty}E_1\left(u\left(\tau\right)\right)\,\mathrm d\tau.
\]
Due to the equivalence between $E_1$ and $E$, we deduce from \eqref{eq:initial-mass-decay} that
\[
\int_t^{+\infty}E\left(u\left(\tau\right)\right)\,\mathrm d\tau\le CM\left(u\left(t\right)\right)\lesssim_{\|u_0\|_{L^2}}\left(1+t\right)^{-\frac{2}{p-1}}.
\]
By the monotonicity of $E$, which follows from \eqref{eq3.2v118}, we deduce that for any dyadic integer $N$,
\[
N^{-\frac{2}{p-1}}\gtrsim\int_N^{2N}E\left(u\left(t\right)\right)\,\mathrm dt\ge E\left(u\left(2N\right)\right)\cdot N,
\]
which implies that for $t\ge1$,
\[
E\left(u\left(t\right)\right)\lesssim_{\|u_0\|_{H^1}}t^{-1-\frac{2}{p-1}}.
\]
This implies
\[
\|u\left(t\right)\|_{L_x^{p+1}}\lesssim_{\|u_0\|_{H^1}}t^{-\frac{1}{p-1}}.
\]
In summary, we have proved:
\begin{proposition}[Energy decay]\label{Energydecay}
For any $t\ge1$, the solution $u\left(t,x\right)$ to the equation \eqref{eq1.CGL} satisfies the following decay rates:
\begin{align}\label{eq3.4v247}
\|u\left(t\right)\|_{H_x^1}&\lesssim_{\|u_0\|_{H^1}}t^{-\frac{1}{p-1}},\quad
\|u\left(t\right)\|_{L_x^{p+1}}\lesssim_{\|u_0\|_{H^1}}t^{-\frac{1}{p-1}},
\end{align}
and
\begin{align*}
\|\nabla_gu\left(t\right)\|_{L_x^2} \lesssim_{\|u_0\|_{H^1}}t^{- \frac12-\frac{1}{p-1}}.
\end{align*}
In particular, we get
\[
E\left(u\left(t\right)\right)\lesssim_{\|u_0\|_{H^1}}t^{-\frac{p+1}{p-1}}.
\]
\end{proposition}
This, together with \eqref{eq3.2v118}, gives
\begin{corollary}\label{partial_tuapriori}
For any $t\ge1$, there holds
\[
\int_t^{+\infty}\|\partial_su\left(s\right)\|_{L_x^2}^2\,\mathrm ds\lesssim_{\|u_0\|_{H^1}}t^{-1-\frac{2}{p-1}},
\]
and
\[
\int_t^{+\infty}\|\nabla_gu\left(\tau\right)\|_{L_x^2}^2\,\mathrm d\tau\lesssim_{\|u_0\|_{L^2}}t^{-\frac{2}{p-1}}.
\]
\end{corollary}

\subsection{\texorpdfstring{$H_x^2$-decay for $1<p \le 3 $}{H\^2-decay}}\label{H2decay_1}
In Proposition~\ref{Energydecay}, we established the decay estimates for the $L_x^2$ norm of $u$ and its first-order derivative.
While
Theorem \ref{th2.6v107} guarantees the $H_x^2$-regularity of the solution for any positive time with $p\in]1,5]$, it only provides qualitative information without any quantitative decay. In this subsection, we focus exclusively on the delicate case $1<p\le3$ and complement the smoothing effect by establishing a quantitative decay estimate for the second-order derivative. More precisely, we prove the following result.

\begin{theorem}
Let $1<p\le3$, and let $u\left(t\right)$ be a solution of the equation \eqref{eq1.CGL}. Then for any $t\ge1$,
\[
\|\Delta_gu\left(t\right)\|_{L_x^2 \left(\mathcal M^3 \right)} \lesssim_{\|u_0\|_{H^1}}t^{- \frac32-\frac{1}{p-1}}.
\]
\end{theorem}

\begin{proof}
Differentiating $\|\Delta_gu (t) \|_{L_x^2\left(\mathcal M^3\right)}^2$ with respect to time,
using \eqref{eq1.CGL}, and integrating by parts, we get
\[
\frac12\frac{\mathrm d}{\mathrm dt} \left( \|\Delta_gu (t) \|_{L_x^2\left(\mathcal M^3\right)}^2 \right)
+\Re z\|\nabla_g\Delta_gu (t) \|_{L_x^2\left(\mathcal M^3\right)}^2
=\Re\left(  \Delta_gu (t) ,z\Delta_gf\left(u (t) \right)\right)_{L_x^2\left(\mathcal M^3\right)}.
\]
By the Cauchy-Schwarz and Young's inequalities,
\[
\left|\Re\left( \Delta_gu (t) ,z\Delta_gf\left(u (t) \right)\right)_{L^2_x \left( \mathcal{M}^3 \right) } \right|
\le \frac{\Re z}{2}\|\nabla_g\Delta_gu (t) \|_{L_x^2\left(\mathcal M^3\right)}^2
+C\|\nabla_gf\left(u (t) \right)\|_{L_x^2\left(\mathcal M^3\right)}^2,
\]
hence
\[
\frac12\frac{\mathrm d}{\mathrm dt} \left( \|\Delta_gu (t) \|_{L_x^2\left(\mathcal M^3\right)}^2 \right)
+\frac{\Re z}{2}\|\nabla_g\Delta_gu (t) \|_{L_x^2 \left(\mathcal M^3 \right)}^2
\le C\|\nabla_gf\left(u (t) \right)\|_{L_x^2 \left(\mathcal M^3 \right)}^2.
\]
The Gagliardo-Nirenberg interpolation inequality and then Young's inequality lead to
\[
\|\nabla_gu (t) \|_{L_x^{2p}\left(\mathcal M^3\right)}^2
\le \delta\|\nabla_g \Delta_gu (t) \|_{L_x^2\left(\mathcal M^3\right)}^2+C\left(\delta\right)\|\nabla_gu (t) \|_{L_x^2\left(\mathcal M^3\right)}^2,
\]
where $\delta > 0$ is a small constant.
This gives rise to
\[
\|\nabla_gf\left(u (t) \right)\|_{L_x^2\left(\mathcal M^3\right)}^2
\le C ( \delta ) \|u (t) \|_{L_x^{2p}\left(\mathcal M^3\right)}^{2\left(p-1\right)}\|\nabla_gu (t) \|_{L_x^2\left(\mathcal M^3\right)}^2
+C\delta\|u (t) \|_{L_x^{2p}\left(\mathcal M^3\right)}^{2\left(p-1\right)}\|\nabla_g\Delta_gu (t) \|_{L_x^2\left(\mathcal M^3\right)}^2.
\]
The Sobolev embedding $H_x^1\hookrightarrow L_x^{2p}$ and the uniform $H_x^1$ bound from Proposition~\ref{pr2.12v70} imply
\[
\sup_{t\ge0}\|u\left(t\right)\|_{L_x^{2p}\left(\mathcal M^3\right)}<\infty.
\]
We may therefore choose $\delta>0$ sufficiently small so that
\[
C\delta\sup_{t\ge0}\|u\left(t\right)\|_{L_x^{2p}\left(\mathcal M^3\right)}^{2\left(p-1\right)}\le \frac{\Re z}{4}.
\]
It follows that
\begin{equation}\label{energy_1}
\frac12\frac{\mathrm d}{\mathrm dt} \left( \|\Delta_gu (t) \|_{L_x^2\left(\mathcal M^3\right)}^2 \right)
+\frac{\Re z}{4}\|\nabla_g\Delta_gu (t) \|_{L_x^2\left(\mathcal M^3\right)}^2
\le C ( \delta ) \|u (t) \|_{L_x^{2p}\left(\mathcal M^3\right)}^{2\left(p-1\right)}\|\nabla_gu (t) \|_{L_x^2\left(\mathcal M^3\right)}^2.
\end{equation}
Since $\Delta_gu$ has zero mean, by Proposition \ref{pr2.3v122}, we have
\begin{equation}\label{poin}
\|\nabla_g\Delta_gu (t) \|_{L_x^2\left(\mathcal M^3\right)}^2\ge\lambda_1\|\Delta_gu (t) \|_{L_x^2\left(\mathcal M^3\right)}^2.
\end{equation}
Substituting \eqref{poin} into \eqref{energy_1} yields
\begin{equation}\label{energy_2}
\frac12\frac{\mathrm d}{\mathrm dt} \left( \|\Delta_gu (t) \|_{L_x^2\left(\mathcal M^3\right)}^2 \right)
+\frac{\Re z\,\lambda_1}{4}\|\Delta_gu (t) \|_{L_x^2\left(\mathcal M^3\right)}^2
\le C ( \delta ) \|u (t) \|_{L_x^{2p}\left(\mathcal M^3\right)}^{2\left(p-1\right)}\|\nabla_gu (t) \|_{L_x^2\left(\mathcal M^3\right)}^2.
\end{equation}
Using the Sobolev embedding and Proposition \ref{Energydecay}, we obtain
\[
\|u (t) \|_{L_x^{2p}\left(\mathcal M^3\right)}^{2\left(p-1\right)}\|\nabla_gu (t) \|_{L_x^2\left(\mathcal M^3\right)}^2
\lesssim \|u (t) \|_{H_x^1\left(\mathcal M^3\right)}^{2\left(p-1\right)}\|\nabla_gu (t) \|_{L_x^2\left(\mathcal M^3\right)}^2
\lesssim_{\|u_0\|_{H^1}}t^{-3-\frac{2}{p-1}}.
\]
Substituting this into \eqref{energy_2} gives
\begin{equation*}
\frac{\mathrm d}{\mathrm dt} \left( \|\Delta_gu (t) \|_{L_x^2\left(\mathcal M^3\right)}^2 \right)
+\frac{\Re z\,\lambda_1}{2}\|\Delta_gu (t) \|_{L_x^2\left(\mathcal M^3\right)}^2
\le C\left(\|u_0\|_{H^1}\right)t^{-3-\frac{2}{p-1}}.
\end{equation*}
Multiplying by the integrating factor $e^{\frac{\Re z\,\lambda_1}{2}t}$ and integrating from $1$ to $t$, we get
\[
\|\Delta_gu\left(t\right)\|_{L_x^2\left(\mathcal M^3\right)}^2
\le \|\Delta_gu\left(1\right)\|_{L_x^2\left(\mathcal M^3\right)}^2e^{-\frac{\Re z\,\lambda_1}{2}\left(t-1\right)}
+C\left(\|u_0\|_{H^1}\right)\int_1^t e^{-\frac{\Re z\,\lambda_1}{2}\left(t-s\right)}s^{-3-\frac{2}{p-1}}\,\mathrm ds.
\]
Since $3+\frac{2}{p-1}>1$, the standard estimate for the convolution with an exponential kernel yields
\[
\int_1^t e^{-\frac{\Re z\,\lambda_1}{2}\left(t-s\right)}s^{-3-\frac{2}{p-1}}\,\mathrm ds\lesssim t^{-3-\frac{2}{p-1}}.
\]
Therefore,
\[
\|\Delta_gu\left(t\right)\|_{L_x^2\left(\mathcal M^3\right)}^2\lesssim_{\|u_0\|_{H^1}}t^{-3-\frac{2}{p-1}}.
\]
This completes the proof.
\end{proof}

\subsection{\texorpdfstring{$H_x^2$-decay for $3<p\le5$}{H\^2-decay}}\label{H2decay}

In the case $3<p\le5$, the main difficulty in deriving higher-order decay estimates is that the standard $H_x^2$-energy does not satisfy a closed dissipative inequality. The nonlinear interaction terms generate error contributions that cannot be controlled directly by the natural energy functional. Our key observation is that these terms can be compensated by introducing a suitable nonlinear correction. This leads to the modified energy functional defined below, which captures the dissipative structure of the CGL equation at the $H_x^2$-level.

\begin{definition}\label{def_me}
We define the modified energy functional $\mathcal E_2\left(u\left(t\right)\right)$ to be
\[
\mathcal E_2\left(u\left(t\right)\right):=\|z^{-1}\partial_tu\left(t\right)\|_{L^2\left(\mathcal M^3\right)}^2
-\frac{p-1}{p}\int_{\mathcal M^3}|u\left(t\right)|^{2p}\,\mathrm d\mu
-\left(p-1\right)\int_{\mathcal M^3}|u\left(t\right)|^{p-1}|\nabla_g\left(|u\left(t\right)|\right)|^2\,\mathrm d\mu.
\]
\end{definition}
Differentiating $\mathcal E_2\left(u\left(t\right)\right)$ in time, we obtain
\begin{equation*}
\begin{split}
 & \frac{\mathrm d}{\mathrm dt}\mathcal E_2\left(u\left(t\right)\right)
 +2\frac{\Re z}{|z|^2}\|\partial_t\nabla_gu\left(t\right)\|_{L_x^2\left(\mathcal M^3\right)}^2
+2\frac{\Re z}{|z|^2}\big\||u\left(t\right)|^{\frac{p-1}{2}}\partial_tu\left(t\right)\big\|_{L_x^2\left(\mathcal M^3\right)}^2\\
&=\left(p-1\right)\left(p-3\right)\int_{\mathcal M^3}|u\left(t\right)|^{p-2}\partial_t\left(|u\left(t\right)|\right)|\nabla_g\left(|u\left(t\right)|\right)|^2\,\mathrm d\mu\\
&\quad+2\left(p-1\right)\int_{\mathcal M^3}|u\left(t\right)|^{p-2}\partial_t\left(|u\left(t\right)|\right)|\nabla_gu\left(t\right)|^2\,\mathrm d\mu.
\end{split}
\end{equation*}
The rest of this subsection is devoted to proving that the modified energy functional satisfies an almost-monotonicity property. Once such a property is established, quantitative decay estimates follow from an abstract argument.

\begin{lemma}[Almost monotonicity]\label{lem:abstract}
Let $f$, $a$, and $b$ be continuous functions on $\mathbb R_+$, with $a\left(t\right)\ge0$ for all $t\ge0$. Let $\left(\tau_n\right)\subset\left[2^n,2^{n+1}\right)$ be an increasing sequence such that
\begin{align}\label{eq3.9v249}
|f\left(\tau_n\right)|\le C_0\tau_n^{-\beta},\quad \forall \,  n\ge1,
\end{align}
where $\beta > 0$.
Assume that
for any $t_2\ge t_1\ge0$,
\begin{align}\label{eq3.10v249}
f\left(t_2\right)-f\left(t_1\right)=-\int_{t_1}^{t_2}a\left(t\right)\,\mathrm dt+\int_{t_1}^{t_2}b\left(t\right)\,\mathrm dt,
\end{align}
and there exists $\gamma>0$ such that
\begin{align}\label{eq3.11v249}
\left|\int_{t_1}^{t_2}b\left(t\right)\,\mathrm dt\right|
\le\frac12\int_{t_1}^{t_2}a\left(t\right)\,\mathrm dt+C_0\left(1+t_1 \right)^{-\gamma},
\end{align}
where the constant $C_0$ is independent of $t_1$, $t_2$ and $\tau_n$. Then for any $t\ge1$,
\[
\int_t^\infty a\left(\tau\right)\,\mathrm d\tau\le C_1t^{-\min\{\beta,\gamma\}}\quad\text{and}\quad |f\left(t\right)|\le C_1t^{-\min\{\beta,\gamma\}}.
\]
\end{lemma}
\begin{proof}
For any $k>n$, by \eqref{eq3.10v249}, we see
\[
f\left(\tau_n\right)-f\left(\tau_k\right)=\int_{\tau_n}^{\tau_k}a\left(t\right)\,\mathrm dt-\int_{\tau_n}^{\tau_k}b\left(t\right)\,\mathrm dt.
\]
We have
\[
\int_{\tau_n}^{\tau_k}a\left(t\right)\,\mathrm dt\le |f\left(\tau_n\right)|+|f\left(\tau_k\right)|+\left|\int_{\tau_n}^{\tau_k}b\left(t\right)\,\mathrm dt\right|.
\]
By the assumption on $\left|\int_{t_1}^{t_2}b\left(t\right)\,\mathrm dt\right|$ and the bound for $|f\left(\tau_j\right)|$, we deduce that
\begin{align}\label{tauktaun}
\int_{\tau_n}^{\tau_k}a\left(t\right)\,\mathrm dt\le4C_0\tau_n^{-\beta}+C_0'\tau_n^{-\gamma}.
\end{align}
From this and the fact that $t\mapsto\int_{\tau_n}^ta\left(\tau\right)\,\mathrm d\tau$ is monotone, we infer that $\int_1^\infty a\left(t\right)\,\mathrm dt$ converges. Thanks to the choice of $\tau_n$ and \eqref{tauktaun}, it follows that
\[
\int_t^\infty a\left(\tau\right)\,\mathrm d\tau\le Ct^{-\min\{\beta,\gamma\}}.
\]
To estimate $f\left(t\right)$, assume that $t\in[2^n,2^{n+1}]$. Consequently, by \eqref{eq3.9v249} and \eqref{eq3.11v249},
\[
|f\left(t\right)|\le|f\left(t\right)-f\left(\tau_n\right)|+|f\left(\tau_n\right)|
\le\left|\int_{\tau_n}^ta\left(s\right)\,\mathrm ds\right|+\left|\int_{\tau_n}^tb\left(s\right)\,\mathrm ds\right|+C_0\tau_n^{-\beta}
\le C\tau_n^{-\min\{\beta,\gamma\}}\le Ct^{-\min\{\beta,\gamma\}}.
\]
\end{proof}

Set
\begin{align*}
f\left(t\right)&:=\mathcal E_2\left(u\left(t\right)\right),\qquad a\left(t\right):=2\frac{\Re z}{|z|^2}\|\partial_t\nabla_gu\left(t\right)\|_{L_x^2\left(\mathcal M^3\right)}^2,
\end{align*}
and
\begin{align*}
b\left(t\right)&:=\left(p-1\right)\left(p-3\right)\int_{\mathcal M^3}|u\left(t\right)|^{p-2}\partial_t\left(|u\left(t\right)|\right)|\nabla_g\left(|u\left(t\right)|\right)|^2\,\mathrm d\mu\\
&\quad+2\left(p-1\right)\int_{\mathcal M^3}|u\left(t\right)|^{p-2}\partial_t\left(|u\left(t\right)|\right)|\nabla_gu\left(t\right)|^2\,\mathrm d\mu.
\end{align*}
The next proposition provides the quantitative decay estimate for the modified energy functional. Combined with the equivalence between $\mathcal E_2\left(u\left(t\right)\right)$ and the $H_x^2$-norm, it yields the higher-order decay required for the asymptotic analysis carried out in Section \ref{se4}. The core objective of this part is to verify the assumptions for the functions $f\left(t\right)$, $a\left(t\right)$, and $b\left(t\right)$ stated in Lemma \ref{lem:abstract}. The proof of the $H_x^2$-norm decay is presented below.

\begin{proposition}\label{prop_3D}
If $p\in]3,5[$, then for any $t\ge1$, the solution to equation \eqref{eq1.CGL} satisfies the following decay estimates
\begin{align}
|\mathcal E_2\left(u\left(t\right)\right)|&\lesssim_{\|u_0\|_{H^1}}t^{-2-\frac{2}{p-1}},\label{E_23}\\
\|\partial_tu\left(t\right)\|_{L_x^2} &\lesssim_{\|u_0\|_{H^1}}t^{-1 -\frac{1}{p-1}},\label{pt_u3}\\
\|\Delta_gu\left(t\right)\|_{L_x^2} &\lesssim_{\|u_0\|_{H^1}}t^{-1 -\frac{1}{p-1}},\label{Deltau_3}\\
\int_t^\infty\|\partial_\tau\nabla_gu\left(\tau\right)\|_{L_x^2}^2\,\mathrm d\tau&\lesssim_{\|u_0\|_{H^1}}t^{-2-\frac{2}{p-1}}.\label{int_u}
\end{align}
For $p=5$, there exists a time $t_0$ large enough such that for $t\ge t_0$, the above estimates \eqref{E_23}, \eqref{pt_u3}, \eqref{Deltau_3} and \eqref{int_u} still hold.
\end{proposition}

To complete the proof of this proposition, we divide the analysis into three cases according to the nonlinear exponent regimes: $p\in]3,4]$, $p\in]4,5[$, and $p=5$. For this purpose, we first derive the required estimate of the remainders in Subsection \ref{estimate_remainder}. The detailed proofs for these three cases are then presented in Subsections \ref{3D_p34}, \ref{3D_p45} and \ref{3D_p5}, respectively.

\subsubsection{Estimate of the remainders}\label{estimate_remainder}
We establish the estimate of the pointwise bound for $f\left(t\right)$ along a quasi-dyadic sequence. This subsection is devoted to the estimation of $\|u\left(t\right)\|_{L_x^{2p}}^{2p}$ for $p\in[3,5[$. To this end, we start by employing interpolation to derive that
\[
\|u\left(t\right)\|_{L_x^{2p}}^{2p}\le\|u\left(t\right)\|_{L_x^\infty}^{p-1}\|u\left(t\right)\|_{L_x^{p+1}}^{p+1}\lesssim\|u\left(t\right)\|_{L_x^\infty}^{p-1}E\left(u\left(t\right)\right).
\]
Using the Gagliardo-Nirenberg inequality, we have
\begin{equation}\label{e_u_infty}
\|u (t) \|_{L_x^\infty}
\lesssim\left(\|u (t) \|_{L_x^2}^{\frac12}+\|\Delta_gu (t) \|_{L_x^2}^{\frac12}\right)\|u (t) \|_{H_x^1}^{\frac12}
\lesssim\|u (t) \|_{H_x^1}+\left(\|\partial_tu (t) \|_{L_x^2}^{\frac12}+\|u (t) \|_{L_x^{2p}}^{\frac p2}\right)\|u (t) \|_{H_x^1}^{\frac12}.
\end{equation}
Hence,
\begin{align}\label{u_2p_2p}
\|u (t) \|_{L_x^{2p}}^{2p}\lesssim
\|u (t) \|_{H_x^1}^{p-1}E\left(u (t) \right)+\|u (t) \|_{H_x^1}^{\frac{p-1}{2}}E\left(u (t) \right)
\left(\|\partial_tu (t) \|_{L_x^2}^{\frac{p-1}{2}}+\|u (t) \|_{L_x^{2p}}^{\frac{\left(p-1\right)p}{2}}\right).
\end{align}
By Proposition \ref{Energydecay}, we obtain
 \[
\|u\left(t\right)\|_{L_x^{2p}}^{2p}\lesssim_{\|u_0\|_{H^1}}t^{-2-\frac{2}{p-1}}
+t^{-\frac32-\frac{2}{p-1}}\left(\|\partial_tu\left(t\right)\|_{L_x^2}^{\frac{p-1}{2}}+\|u\left(t\right)\|_{L_x^{2p}}^{\frac{\left(p-1\right)p}{2}}\right).
\]
For $3\le p<5$, we obtain that for $t$ large enough,
\begin{align}\label{perturbation:E2}
\|u\left(t\right)\|_{L_x^{2p}}^{2p}\lesssim_{\|u_0\|_{H^1}}t^{-2-\frac{2}{p-1}}
+t^{-\frac32-\frac{2}{p-1}}\|\partial_tu\left(t\right)\|_{L_x^2}^{\frac{p-1}{2}}.
\end{align}
This, together with Young's inequality, implies that for any $\delta>0$,
\[
\|u\left(t\right)\|_{L_x^{2p}}^{2p}\lesssim_{\|u_0\|_{H^1}}t^{-2-\frac{2}{p-1}}
+t^{-\frac4{5-p}\left(\frac32+\frac2{p-1}\right)}+\delta\|\partial_tu\left(t\right)\|_{L_x^2}^2
\le C\left(\delta,\|u_0\|_{H^1}\right)t^{-2-\frac{2}{p-1}}+\delta\|\partial_tu\left(t\right)\|_{L_x^2}^2.
\]
To summarize the preceding analysis, by combining estimates \eqref{e_u_infty} and \eqref{perturbation:E2} with Proposition \ref{Energydecay}, we deduce the following decay estimate for the solution $u\left(t,x\right)$ of the equation \eqref{eq1.CGL} whenever $p\in[3,5[$.
\begin{proposition}[Lebesgue norm decay]\label{lem:pointwise}
If $3\le p<5$, we have
\begin{align*}
& \|u\left(t\right)\|_{L_x^{2p}}\lesssim_{\|u_0\|_{H^1}}t^{-\frac1p\left(1+\frac1{p-1}\right)}
+t^{-\frac1p\left(\frac34+\frac1{p-1}\right)}\|\partial_tu\left(t\right)\|_{L_x^2}^{\frac{p-1}{4p}},
\intertext{ and }
& \|u\left(t\right)\|_{L_x^\infty}\lesssim_{\|u_0\|_{H^1}}t^{-\frac1{p-1}}
+t^{-\frac1{2\left(p-1\right)}}\|\partial_tu\left(t\right)\|_{L_x^2}^{\frac12}
+t^{-\frac1{p-1}-\frac38}\|\partial_tu\left(t\right)\|_{L_x^2}^{\frac{p-1}{8}}.
\end{align*}

\end{proposition}

\subsubsection{Proof of Proposition \ref{prop_3D} with $p\in]3,4]$}\label{3D_p34}
We start the proof with the estimate for $b\left(t\right)$ in the energy derivatives.
Invoking the Cauchy-Schwarz inequality, we have
\begin{align}\label{e_b}
|b\left(t\right)|\lesssim\|\partial_tu\left(t\right)\|_{L_x^6} \|\nabla_gu\left(t\right)\|_{L_x^2}\|\nabla_gu\left(t\right)\|_{L_x^\infty}\|u\left(t\right)\|_{L_x^{3\left(p-2\right)}}^{p-2}.
\end{align}
By the interpolation and Young inequalities, for any $\delta>0$,
\begin{align}\label{parabolic}
\|\nabla_gu (t) \|_{L_x^\infty}\le  \left\|\nabla_g^3u (t)  \right\|_{L_x^2}^{\frac34}\|\nabla_gu (t) \|_{L_x^2}^{\frac14}
\le\delta\|\nabla_g\Delta_gu (t) \|_{L_x^2}+C\left(\delta\right)\|\nabla_gu (t) \|_{L_x^2}.
\end{align}
Replacing $\Delta_gu$ by $z^{-1}\partial_tu+|u|^{p-1}u$ yields
\[
\|\nabla_gu (t) \|_{L_x^\infty}\le C\delta\|\partial_t\nabla_gu (t) \|_{L_x^2}
+C\delta\|\nabla_gu (t) \|_{L_x^\infty}\|u (t) \|_{L_x^{2\left(p-1\right)}}^{p-1}+C\left(\delta\right)\|\nabla_gu (t) \|_{L_x^2},
\]
for some constant $C>0$ independent of $\delta$. By the Sobolev inequality, we obtain that
\[
\|\nabla_gu (t) \|_{L_x^\infty}\left(1-C\delta\|u (t) \|_{H_x^1}^{p-1}\right)
\le\delta\|\partial_t\nabla_gu (t) \|_{L_x^2}+C\left(\delta\right)\|\nabla_gu (t) \|_{L_x^2}.
\]
Taking $\delta$ sufficiently small and using Proposition \ref{Energydecay}, we obtain
\begin{align}\label{n_u_infty}
\|\nabla_gu (t) \|_{L_x^\infty}\le C\delta\|\partial_t\nabla_gu (t) \|_{L_x^2}+C\left(\delta\right)\|\nabla_gu (t) \|_{L_x^2}.
\end{align}
Using the Sobolev inequality, we can also get
\begin{align}\label{pt_u_6}
\|\partial_tu (t) \|_{L_x^6}\lesssim\|\partial_tu (t) \|_{L_x^2}+\|\partial_t\nabla_gu (t) \|_{L_x^2}.
\end{align}
Combining \eqref{e_b}, \eqref{n_u_infty}, and \eqref{pt_u_6}, together with the boundedness of $E\left(u\left(t\right)\right)$ and the Sobolev inequality, yields
\begin{align*}
|b\left(t\right)|\lesssim_{\|u_0\|_{H^1}}&\|u\left(t\right)\|_{H_x^1}^{p-2}\|\nabla_gu\left(t\right)\|_{L_x^2}
\left(\|\partial_tu\left(t\right)\|_{L_x^2}+\|\partial_t\nabla_gu\left(t\right)\|_{L_x^2}\right)
\left(\delta\|\partial_t\nabla_gu\left(t\right)\|_{L_x^2}+\|\nabla_gu\left(t\right)\|_{L_x^2}\right)\\
\lesssim_{\|u_0\|_{H^1}}&\delta\|\partial_t\nabla_gu\left(t\right)\|_{L_x^2}^2
\left(1+\|u\left(t\right)\|_{H_x^1}^{p-2}\|\nabla_gu\left(t\right)\|_{L_x^2}\right)
+\|u\left(t\right)\|_{H_x^1}^{2p-4}\|\nabla_gu\left(t\right)\|_{L_x^2}^4\\
&+\|u\left(t\right)\|_{H_x^1}^{2p-4}\|\nabla_gu\left(t\right)\|_{L_x^2}^2\|\partial_tu\left(t\right)\|_{L_x^2}^2
+\|u\left(t\right)\|_{H_x^1}^{p-1}\|\partial_tu\left(t\right)\|_{L_x^2}^2
+\|u\left(t\right)\|_{H_x^1}^{p-3}\|\nabla_gu\left(t\right)\|_{L_x^2}^4.
\end{align*}
Therefore, for any $1\le t_1\le t_2<\infty$, we have
\begin{align*}
\int_{t_1}^{t_2}|b\left(t\right)|\,\mathrm dt\lesssim_{\|u_0\|_{H^1}}&\delta\left(1+\|u\left(t_1\right)\|_{H_x^1}^{p-2}\|\nabla_gu\left(t_1\right)\|_{L_x^2}\right)
\int_{t_1}^{t_2}\|\partial_t\nabla_gu\left(t\right)\|_{L_x^2}^2\,\mathrm dt\\
&+\left(\|u\left(t_1\right)\|_{H_x^1}^{2p-4}+\|u\left(t_1\right)\|_{H_x^1}^{p-3}\right)\|\nabla_gu\left(t_1\right)\|_{L_x^2}^2
\int_{t_1}^{t_2}\|\nabla_gu\left(t\right)\|_{L_x^2}^2\,\mathrm dt\\
&+\left(\|u\left(t_1\right)\|_{H_x^1}^{2p-4}\|\nabla_gu\left(t_1\right)\|_{L_x^2}^2+\|u\left(t_1\right)\|_{H_x^1}^{p-1}\right)
\int_{t_1}^{t_2}\|\partial_tu\left(t\right)\|_{L_x^2}^2\,\mathrm dt,
\end{align*}
which, together with Proposition \ref{Energydecay} and Corollary \ref{partial_tuapriori}, yields the following estimate for sufficiently small  $\delta$:
\[
\int_{t_1}^{t_2}|b\left(t\right)|\,\mathrm dt\le\frac{\Re z}{|z|^2}\int_{t_1}^{t_2}\|\partial_t\nabla_gu\left(t\right)\|_{L_x^2}^2\,\mathrm dt
+C\left(\delta,\|u_0\|_{H^1}\right)t_1^{-\frac{2p}{p-1}}.
\]
For the term $ \left\||u (t) |^{p-1}|\nabla_g\left(|u (t) |\right)|^2 \right\|_{L_x^1}$, we use the Cauchy-Schwarz inequality,
the Sobolev inequality, Young's inequality and $\Delta_gu=z^{-1}\partial_tu+|u|^{p-1}u$ to deduce that for any sufficiently small $\delta$,
\begin{equation}\label{nonlinear_e1}
\begin{split}
\left\|  | u (t) |^{p-1}|\nabla_g\left(|u (t) |\right)|^2 \right\|_{L_x^1}
&\lesssim\|\nabla_gu (t) \|_{L_x^{p+1}}^2\|u (t) \|_{L_x^{p+1}}^{p-1}
\\
&\lesssim\|u (t) \|_{H_x^1}^{p-1}\left( \|\nabla_gu (t) \|_{L_x^2}^2+\delta\|\Delta_gu (t) \|_{L_x^2}^2\right)
\\
&
\lesssim\|u (t) \|_{H_x^1}^{p-1}\left(\|\nabla_gu (t) \|_{L_x^2}^2+\delta\|\partial_tu (t) \|_{L_x^2}^2+\delta\|u (t) \|_{L_x^{2p}}^{2p}\right).
\end{split}
\end{equation}
Now, to verify the condition satisfied by $f\left(t\right)$, we observe that, by Corollary \ref{partial_tuapriori}, for any dyadic interval $[2^n,2^{n+1}]$,
\[
2^{-n\left(1+\frac2{p-1}\right)}\ge\int_{2^n}^{2^{n+1}}\|\partial_tu\left(t\right)\|_{L_x^2}^2\,\mathrm dt
\ge2^n\min_{t\in[2^n,2^{n+1}]}\|\partial_tu\left(t\right)\|_{L_x^2}^2.
\]
Let $\tau_n=\arg\min_{t\in[2^n,2^{n+1}]}\|\partial_tu\left(t\right)\|_{L_x^2}^2$. Then we obtain that
\begin{equation}\label{pt_u}
\|\partial_tu\left(\tau_n\right)\|_{L_x^2}^2\le2^{-n\left(2+\frac2{p-1}\right)}.
\end{equation}
Combining \eqref{nonlinear_e1} and \eqref{pt_u} with Proposition \ref{Energydecay} and Proposition \ref{lem:pointwise}, we have for $f\left(t\right):=\mathcal E_2\left(u\left(t\right)\right)$,
\[
|f\left(\tau_n\right)|\le C\tau_n^{-\left(2+\frac2{p-1}\right)}.
\]
In summary, $f\left(t\right)$, $a\left(t\right)$, and $b\left(t\right)$ verify the hypothesis of Lemma \ref{lem:abstract} with $\beta=\gamma=\frac{2p}{p-1}$. This implies
\[
|\mathcal E_2\left(u\left(t\right)\right)|\le C\left(\|u_0\|_{H^1}\right)t^{-2-\frac2{p-1}}.
\]
We then combine \eqref{nonlinear_e1} and \eqref{perturbation:E2} to derive that
\begin{align*}
\|\partial_tu\left(t\right)\|_{L_x^2}^2
&\lesssim\mathcal E_2\left(u\left(t\right)\right)+ \left\|u\left(t\right) \right\|_{L_x^{2p}}^{2p}
+ \left\||u (t) |^{p-1}|\nabla_g\left( |u\left(t\right)|\right)|^2 \right\|_{L_x^1}\\
&\lesssim_{\|u_0\|_{H^1}}\mathcal E_2\left(u\left(t\right)\right)+\delta\|\partial_tu\left(t\right)\|_{L_x^2}^2+t^{-\frac{2p}{p-1}}
+\|u\left(t\right)\|_{H_x^1}^{p-1}\left(\|\nabla_gu\left(t\right)\|_{L_x^2}^2 +\delta\|\partial_tu\left(t\right)\|_{L_x^2}^2+t^{-\frac{2p}{p-1}}\right),
\end{align*}
which, due to the smallness of $\delta$ and Proposition \ref{Energydecay}, leads to
\[
\|\partial_tu\left(t\right)\|_{L_x^2}^2\lesssim_{\|u_0\|_{H^1}}\mathcal E_2\left(u\left(t\right)\right)+t^{-\frac{2p}{p-1}}
+\|u\left(t\right)\|_{H_x^1}^{p-1}\left(\|\nabla_gu\left(t\right)\|_{L_x^2}^2+t^{-\frac{2p}{p-1}}\right).
\]
By Proposition \ref{Energydecay}, this implies that
\[
\|\partial_tu\left(t\right)\|_{L_x^2}^2\lesssim_{\|u_0\|_{H^1}}t^{-2-\frac2{p-1}}.
\]
Thus, by applying \eqref{perturbation:E2} again with $t\ge1$, we obtain
\[
\|\Delta_gu\left(t\right)\|_{L_x^2}^2\le 2|z|^{-2}\|\partial_tu\left(t\right)\|_{L_x^2}^2+2\|u\left(t\right)\|_{L_x^{2p}}^{2p}
\lesssim_{\|u_0\|_{H^1}}t^{-2-\frac2{p-1}},
\]
which completes the proof of Proposition \ref{prop_3D} with $p\in]3,4]$.

\subsubsection{Proof of Proposition \ref{prop_3D} with $p\in]4,5[$}\label{3D_p45}
We start by estimating $b\left(t\right)$ in the energy derivatives. For this purpose, we give the following lemma.

\begin{lemma}\label{cor:Chebyshev'}
Let $u\in C_t^0H_x^1\left(\mathbb R_+\times\mathcal M^3\right)$ be a solution of \eqref{eq1.CGL}. For any $t\ge1$ and $\lambda>0$, denote
\[
\Lambda_{\lambda,t}:=\{\tau\ge t:\|\partial_\tau u\left(\tau\right)\|_{L_x^2}\ge\lambda\}.
\]
Then we have
\[
|\Lambda_{\lambda,t}|\lesssim_{u_0}\lambda^{-2}t^{-1-\frac2{p-1}}.
\]
\end{lemma}
The critical value of $\lambda$ is such that $\lambda^{-2}t^{-1-\frac2{p-1}}=1$, i.e., $\lambda=t^{-\frac12-\frac1{p-1}}$;
 hence, the level set $\Lambda_{\lambda,t}$ has measure of size $O\left(1\right)$. Therefore, for any interval $I\subset\mathbb R_+$ of size $O\left(1\right)$,
we can always find a point $\tau_0\in I$ such that $\|\partial_tu\left(\tau_0\right)\|_{L_x^2}\lesssim\tau_0^{-\frac12-\frac1{p-1}}$.
Combining this with Proposition \ref{lem:pointwise}, we deduce that $\|\Delta_gu\left(\tau_0\right)\|_{L_x^2}\lesssim\tau_0^{-\frac12-\frac1{p-1}}$. Together with Proposition \ref{Energydecay}, we get $\|u\left(\tau_0\right)\|_{H_x^2}\lesssim\tau_0^{-\frac1{p-1}}$. Therefore, by the local well-posedness theory, at least for sufficiently large $\tau_0\in I$, the solution is continuous with respect to the $H_x^2$-topology.
 Hence, on $I$, we always have $\|u\left(t\right)\|_{H_x^2}\lesssim_{\|u_0\|_{H^1}}t^{-\frac1{p-1}}$. This argument holds on any $I$, so we conclude directly that \begin{align}\label{eq3.27v258}
\|u\left(t\right)\|_{H_x^2}\lesssim_{\|u_0\|_{H^1}}t^{-\frac1{p-1}}.
\end{align}
Now, for the term $b\left(t\right)$, we use the Cauchy-Schwarz inequality to deduce that
\begin{align}\label{b_3D}
|b\left(t\right)|
\lesssim \left\||u (t) |^{p-2} \right\|_{L_x^3}\|\partial_tu (t) \|_{L_x^6}\|\nabla_gu (t) \|_{L_x^2}\|\nabla_gu (t) \|_{L_x^\infty}.
\end{align}
Obviously, by the Sobolev embedding,
\begin{align}\label{non_L3}
\left\||u (t) |^{p-2} \right\|_{L_x^3}\le\|u (t) \|_{L_x^\infty}^{\frac{2p-7}{3}}\|u (t) \|_{L_x^{p+1}}^{\frac{p+1}{3}}
\lesssim\|u (t) \|_{H_x^2}^{\frac{2p-7}{3}}E\left(u (t) \right)^{\frac13}.
\end{align}
By \eqref{parabolic}, we replace $\Delta_gu$ by $z^{-1}\partial_tu+|u|^{p-1}u$ to get
\[
\begin{split}
\|\nabla_gu (t) \|_{L_x^\infty}&\lesssim\delta \|\partial_t\nabla_gu (t) \|_{L_x^2}
+\delta \left\||u (t) |^{p-1}\nabla_gu (t)  \right\|_{L_x^2}+\|\nabla_gu (t) \|_{L_x^2}\\
&\lesssim\delta\|\partial_t\nabla_gu (t) \|_{L_x^2}
+\delta\|u (t) \|_{L_x^\infty}\|\nabla_gu (t) \|_{L_x^\infty}\|u (t) \|_{L_x^{2\left(p-2\right)}}^{p-2}+\|\nabla_gu (t) \|_{L_x^2} .
\end{split}
\]
It follows from the Sobolev embedding and \eqref{eq3.27v258} that
\[
\begin{split}
\|\nabla_gu (t) \|_{L_x^\infty}
&\le C\delta\|\partial_t\nabla_gu (t) \|_{L_x^2}
+C\delta\|u (t) \|_{H_x^2}\|\nabla_gu (t) \|_{L_x^\infty}\|u (t) \|_{H_x^1}^{p-2}+C\left(\delta\right)\|\nabla_gu (t) \|_{L_x^2}\\
&\lesssim_{\|u_0\|_{H^1}}\delta\|\partial_t\nabla_gu (t) \|_{L_x^2}+\delta t^{-1}\|\nabla_gu (t) \|_{L_x^\infty}+\|\nabla_gu (t) \|_{L_x^2}.
\end{split}
\]
For any $t\ge1$ and $\delta$ small enough, there holds
\begin{equation}\label{non_Linfty}
\|\nabla_gu (t) \|_{L_x^\infty}\lesssim_{\|u_0\|_{H^1}}\delta\|\partial_t\nabla_gu (t) \|_{L_x^2}+\|\nabla_gu (t) \|_{L_x^2}.
\end{equation}
Gathering \eqref{b_3D}, \eqref{non_L3}, and \eqref{non_Linfty}, we obtain
\begin{align*}
|b\left(t\right)| & \lesssim_{\|u_0\|_{H^1}}
\|u\left(t\right)\|_{H_x^2}^{\frac{2p-7}{3}}E\left(u\left(t\right)\right)^{\frac13}
\left(\|\partial_tu\left(t\right)\|_{L_x^2}+\|\partial_t\nabla_gu\left(t\right)\|_{L_x^2}\right)\|\nabla_gu\left(t\right)\|_{L_x^2}
\left(\delta\|\partial_t\nabla_gu\left(t\right)\|_{L_x^2}+\|\nabla_gu\left(t\right)\|_{L_x^2}\right)\\
& \sim_{\|u_0\|_{H^1}}\delta\|u\left(t\right)\|_{H_x^2}^{\frac{2p-7}{3}}E\left(u\left(t\right)\right)^{\frac13} \|\nabla_gu\left(t\right)\|_{L_x^2}\|\partial_t\nabla_gu\left(t\right)\|_{L_x^2}^2\\
& \qquad \qquad +\delta\|u\left(t\right)\|_{H_x^2}^{\frac{2p-7}{3}}E\left(u\left(t\right)\right)^{\frac13}\|\nabla_gu\left(t\right)\|_{L_x^2} \|\partial_t\nabla_gu\left(t\right)\|_{L_x^2}\|\partial_tu\left(t\right)\|_{L_x^2}\\
& \qquad \qquad +\|u\left(t\right)\|_{H_x^2}^{\frac{2p-7}{3}}E\left(u\left(t\right)\right)^{\frac13}\|\nabla_gu\left(t\right)\|_{L_x^2}^2 \|\partial_t\nabla_gu\left(t\right)\|_{L_x^2}\\
&  \qquad \qquad +\|u\left(t\right)\|_{H_x^2}^{\frac{2p-7}{3}}E\left(u\left(t\right)\right)^{\frac13}\|\nabla_gu\left(t\right) \|_{L_x^2}^2\|\partial_tu\left(t\right)\|_{L_x^2}\\
& :=  I_1+I_2+I_3+I_4.
\end{align*}
By Young's inequality, we have
\begin{align*}
I_1+I_2\lesssim_{\|u_0\|_{H^1}}& \delta\|u\left(t\right)\|_{H_x^2}^{\frac{2p-7}{3}}E\left(u\left(t\right)\right)^{\frac13}
\|\nabla_gu\left(t\right)\|_{L_x^2}\|\partial_t\nabla_gu\left(t\right)\|_{L_x^2}^2
+\delta\|\partial_t\nabla_g u (t) \|_{L_x^2}^2\\
&+\|u (t) \|_{H_x^2}^{\frac{2\left(2p-7\right)}{3}}E\left(u (t) \right)^{\frac23}\|\nabla_gu (t) \|_{L_x^2}^2\|\partial_tu (t) \|_{L_x^2}^2,
\intertext{ and }
I_3\lesssim_{\|u_0\|_{H^1}}
& \delta\|\partial_t\nabla_gu (t) \|_{L_x^2}^2
+\|u (t) \|_{H_x^2}^{\frac{2\left(2p-7\right)}{3}} E\left(u (t) \right)^{\frac23}\|\nabla_gu (t) \|_{L_x^2}^4.
\end{align*}
For the term $I_4$, simple calculations yield
\begin{align*}
I_4
&\sim\|\nabla_gu (t) \|_{L_x^2}^{\frac{p-1}{p+1}}\|\partial_tu (t) \|_{L_x^2}
\|u (t) \|_{H_x^2}^{\frac{2p-7}{3}} E\left(u (t) \right)^{\frac13}\|\nabla_gu (t) \|_{L_x^2}^{2-\frac{p-1}{p+1}}\\
&\lesssim\|\nabla_gu (t) \|_{L_x^2}^{\frac{2\left(p-1\right)}{p+1}}\|\partial_tu (t) \|_{L_x^2}^2
+\|u (t) \|_{H_x^2}^{\frac{2\left(2p-7\right)}{3}}E\left(u (t) \right)^{\frac23}\|\nabla_gu (t) \|_{L_x^2}^{\frac4{p+1}}\|\nabla_gu (t) \|_{L_x^2}^2.
\end{align*}
Hence,
\begin{align*}
|b\left(t\right)|\lesssim_{\|u_0\|_{H^1}}&\delta\left(1+\|u\left(t\right)\|_{H_x^2}^{\frac{2p-7}{3}}E\left(u\left(t\right)\right)^{\frac13}\|\nabla_gu\left(t\right)\|_{L_x^2}\right)\|\partial_t\nabla_gu\left(t\right)\|_{L_x^2}^2\\
&+\left(\|\nabla_gu\left(t\right)\|_{L_x^2}^{\frac{2\left(p-1\right)}{p+1}}+\|u\left(t\right)\|_{H_x^2}^{\frac{2\left(2p-7\right)}{3}}E\left(u\left(t\right)\right)^{\frac23}\|\nabla_gu\left(t\right)\|_{L_x^2}^2\right)\|\partial_tu\left(t\right)\|_{L_x^2}^2\\
&+\|u\left(t\right)\|_{H_x^2}^{\frac{2\left(2p-7\right)}{3}}E\left(u\left(t\right)\right)^{\frac23}
\left(\|\nabla_gu\left(t\right)\|_{L_x^2}^2+\|\nabla_gu\left(t\right)\|_{L_x^2}^{\frac4{p+1}}\right)\|\nabla_gu\left(t\right)\|_{L_x^2}^2.
\end{align*}
The integration of $|b\left(t\right)|$ with respect to time gives for any $1\le t_1\le t_2<\infty$,
\begin{align*}
\int_{t_1}^{t_2}|b\left(t\right)|\,\mathrm dt\lesssim_{\|u_0\|_{H^1}}&\delta\left(1+\|u\left(t_1\right)\|_{H_x^2}^{\frac{2p-7}{3}}E\left(u\left(t_1\right)\right)^{\frac13}\|\nabla_gu\left(t_1\right)\|_{L_x^2}\right)
\int_{t_1}^{t_2}\|\partial_t\nabla_gu\left(t\right)\|_{L_x^2}^2\,\mathrm dt\\
&+\left(\|\nabla_gu\left(t_1\right)\|_{L_x^2}^{\frac{2\left(p-1\right)}{p+1}}+\|u\left(t_1\right)\|_{H_x^2}^{\frac{2\left(2p-7\right)}{3}}E\left(u\left(t_1\right)\right)^{\frac23}\|\nabla_gu\left(t_1\right)\|_{L_x^2}^2\right)
\int_{t_1}^{t_2}\|\partial_tu\left(t\right)\|_{L_x^2}^2\,\mathrm dt\\
&+\|u\left(t_1\right)\|_{H_x^2}^{\frac{2\left(2p-7\right)}{3}}E\left(u\left(t_1\right)\right)^{\frac23}
\left(\|\nabla_gu\left(t_1\right)\|_{L_x^2}^2+\|\nabla_gu\left(t_1\right)\|_{L_x^2}^{\frac4{p+1}}\right)
\int_{t_1}^{t_2}\|\nabla_gu\left(t\right)\|_{L_x^2}^2\,\mathrm dt.
\end{align*}
Taking $\delta$ sufficiently small and using Proposition \ref{Energydecay}, Corollary \ref{partial_tuapriori}, and \eqref{eq3.27v258},
we obtain
\[
\int_{t_1}^{t_2}|b\left(t\right)|\,\mathrm dt\le\frac{\Re z}{|z|^2}\int_{t_1}^{t_2}\|\partial_t\nabla_gu\left(t\right)\|_{L_x^2}^2\,\mathrm dt+C\left(\delta,\|u_0\|_{H^1}\right)t_1^{-\frac{2p}{p-1}}.
\]
By \eqref{pt_u}, \eqref{nonlinear_e1}, Proposition \ref{Energydecay} and Corollary \ref{lem:pointwise}, we have for $f\left(t\right):=\mathcal E_2\left(u\left(t\right)\right)$,
\[
|f\left(\tau_n\right)|\lesssim\tau_n^{-\left(2+\frac2{p-1}\right)}.
\]
In summary, $f\left(t\right)$, $a\left(t\right)$, and $b\left(t\right)$ verify the hypothesis of Lemma \ref{lem:abstract} with $\beta=\gamma=\frac{2p}{p-1}$, so that
\begin{align}\label{eq3.28newv167}
|\mathcal E_2\left(u\left(t\right)\right)|\lesssim_{u_0}t^{-2-\frac2{p-1}}.
\end{align}
Since
\begin{equation}\label{pt_u2}
\begin{split}
\|\partial_tu\left(t\right)\|_{L_x^2}^2
&\lesssim\mathcal E_2\left(u\left(t\right)\right)+\|u\left(t\right)\|_{L_x^{2p}}^{2p}+ \left\||u\left(t\right)|^{p-1}
|\nabla_g\left(|u\left(t\right)|\right)|^2 \right\|_{L_x^1}\\
&\lesssim\mathcal E_2\left(u\left(t\right)\right)+\|u\left(t\right)\|_{H_x^2}^{p-1}E\left(u\left(t\right)\right)
+\|u\left(t\right)\|_{H_x^2}^{p-1}\|\nabla_gu\left(t\right)\|_{L_x^2}^2,
\end{split}
\end{equation}
it follows from \eqref{eq3.28newv167} together with Proposition \ref{Energydecay} that
\[
\|\partial_tu\left(t\right)\|_{L_x^2}\lesssim_{\|u_0\|_{H^1}}t^{-1-\frac1{p-1}}.
\]
We can then readily deduce that
\[
\|\Delta_gu\left(t\right)\|_{L_x^2}\lesssim_{\|u_0\|_{H^1}}t^{-1-\frac1{p-1}}.
\]
So we can finish the proof of Proposition \ref{prop_3D} with $p\in]4,5[$.

\subsubsection{Proof of Proposition \ref{prop_3D} with $p=5$}\label{3D_p5}
Firstly, we estimate $\|u\left(t\right)\|_{L_x^{10}} $.
Invoking \eqref{u_2p_2p}, we have
\begin{align*}
\|u\left(t\right)\|_{L_x^{10 }}^{10 }
\lesssim&\|u\left(t\right)\|_{H_x^1}^{4 }E\left(u\left(t\right)\right)
+\|u\left(t\right)\|_{H_x^1}^{2}E\left(u\left(t\right)\right)\left(\|\partial_tu\left(t\right)\|_{L_x^2}^{2 } +\|u\left(t\right)\|_{L_x^{10 }}^{10 }\right)\\
\sim&\|u\left(t\right)\|_{H_x^1}^{4}E\left(u\left(t\right)\right)
+\|u\left(t\right)\|_{H_x^1}^{2 }E\left(u\left(t\right)\right)\|\partial_tu\left(t\right)\|_{L_x^2}^{2 }
+\|u\left(t\right)\|_{H_x^1}^{2 }E\left(u\left(t\right)\right)\|u\left(t\right)\|_{ L_x^{10 } }^{10 }.
\end{align*}
By Proposition \ref{Energydecay}, this gives rise to
\[
\|u\left(t\right)\|_{L_x^{10 }}^{10 }
\le C\left(\|u_0\|_{H^1}\right)t^{-\frac{5}{2}}
+C\left(\|u_0\|_{H^1}\right)t^{-2 }\|\partial_tu\left(t\right)\|_{L_x^2}^{2 }
+C\left(\|u_0\|_{H^1}\right)t^{-2 }\|u\left(t\right)\|_{L_x^{10 }}^{10 }.
\]
Taking $t_0$ large enough, we have for $t \ge t_0$,
\begin{align}\label{nonlinear_10}
\|u\left(t\right)\|_{L_x^{10}}^{10 }
\le Ct^{-\frac{5}{2}}+Ct^{-2 }\|\partial_tu\left(t\right)\|_{L_x^2}^{2}.
\end{align}
Now, fix any interval $I\subset]t_0,\infty[$ of size $O\left(1\right)$. We can always find a point $\tau_0\in I$ such that
\[
\|\partial_tu\left(\tau_0\right)\|_{L_x^2}\lesssim\tau_0^{-\frac34 }.
\]
Combining this with \eqref{nonlinear_10}, we deduce that $\|\Delta_gu\left(\tau_0\right)\|_{L_x^2}\lesssim\tau_0^{-\frac34 }$.
 Together with \eqref{eq3.4v247}, we get $\|u\left(\tau_0\right)\|_{H_x^2}\lesssim\tau_0^{-\frac1{p-1}}$.
 Therefore, by the local well-posedness theory, at least for sufficiently large $\tau_0\in I$, the solution is continuous with respect to the $H_x^2$-topology.
 Hence, on $I$, we always have $\|u\left(t\right)\|_{H_x^2}\lesssim_{\|u_0\|_{H^1}}t^{-\frac1{4}}$.
 This argument works for any $I$, so we conclude directly that $\|u\left(t\right)\|_{H_x^2}\lesssim_{\|u_0\|_{H^1}}t^{-\frac1{4}}$.

A simple calculation yields
\begin{align*}
|b\left(t\right)|&
\lesssim \left\| |u\left(t\right)|^3\partial_tu\left(t\right)|\nabla_g\left(|u\left(t\right)|\right)|^2 \right\|_{L_x^1}
\\
&\lesssim\|u\left(t\right)\|_{H_x^2}E\left(u\left(t\right)\right)^{\frac13}
\|\partial_tu\left(t\right)\|_{L_x^6}\|\nabla_gu\left(t\right)\|_{L_x^2}\|\nabla_gu\left(t\right)\|_{L_x^\infty},
\end{align*}
where we have used \eqref{non_L3}. Invoking \eqref{non_Linfty}, we obtain
\[
|b\left(t\right)|\lesssim\|u\left(t\right)\|_{H_x^2}E\left(u\left(t\right)\right)^{\frac13}
\left(\|\partial_tu\left(t\right)\|_{L_x^2}+\|\partial_t\nabla_gu\left(t\right)\|_{L_x^2}\right)\|\nabla_gu\left(t\right)\|_{L_x^2}
\left(\delta\|\partial_t\nabla_gu\left(t\right)\|_{L_x^2}+\|\nabla_gu\left(t\right)\|_{L_x^2}\right).
\]
Using the same argument as in the proof of the case $p\in]4,5[$, one could easily deduce that for $t_0\le t_1\le t_2<\infty$,
\[
\int_{t_1}^{t_2}|b\left(t\right)|\,\mathrm dt\le\frac{\Re z}{|z|^2}\int_{t_1}^{t_2}\|\partial_t\nabla_gu\left(t\right)\|_{L_x^2}^2\,\mathrm dt
+C\left(\|u_0\|_{H^1},\delta\right)t_1^{-\frac{5}{2}}.
\]
Hence, in this case, we derive that
\[
\mathcal E_2\left(u\left(t\right)\right)\lesssim_{\|u_0\|_{H^1}}t^{- \frac52 }.
\]
Then, by \eqref{pt_u2}, we can conclude that
\[
\|\partial_tu\left(t\right)\|_{L_x^2}\lesssim_{\|u_0\|_{H^1}}t^{- \frac54 } .
\]
This leads to
\[
\|\Delta_gu\left(t\right)\|_{L_x^2}\lesssim_{\|u_0\|_{H^1}}t^{- \frac54 }.
\]
Hence, we have completed the proof of Proposition \ref{prop_3D} with $p=5$ for $t\ge t_0$.

\section{Asymptotic profile}\label{se4}

In this section, we complete the proof of Theorem \ref{th2v79} by deriving the precise asymptotic profile of the global solutions.
The key observation is that the fluctuation component $v$ decays much faster than the spatial average $m$.
Hence, the long time dynamics are governed by an effective finite-dimensional equation satisfied by $m$,
 while the influence of $v$ appears only as a perturbative correction.
 The proof therefore reduces to a detailed analysis of the asymptotic behavior of the averaged mode.

In the sequel, we introduce the following notation. For any function $F$, denote
\begin{align*}
 \Pi_0(F)=\dashint_{\mathcal{M}^3}F \, \dd\mu,\quad \Pi_0^{\perp}:=\mathrm{Id}-\Pi_0.
\end{align*}
Our starting point is the decomposition
\begin{align}\label{eq4.1v250}
	u(t,x)=m(t)+v(t,x),
\end{align}
as in \eqref{eq4.1v127},
where $m(t)$ stands for the spatial integral average of $u$ over $\mathcal{M}^3$.
The fluctuation component $v$ has zero mean
\begin{align}\label{int_v}
\dashint_{\mathcal{M}^3} v \,\mathrm{d}\mu=0,
\end{align}
and therefore satisfies Poincar\'e's inequality in Proposition~\ref{pr2.3v122}, which will play a crucial role in the subsequent analysis.

Substituting the decomposition \eqref{eq4.1v250} into \eqref{eq1.CGL}, we have
\begin{align}\label{eq4.2v250}
\partial_t v=z\Delta_g v-z f(u)+z \dashint_{\mathcal{M}^3} f(u)  \,\mathrm{d}\mu.
\end{align}
Taking the $L^{2}$ inner product of both sides against $\ov{v}$ yields
\begin{align*}
	\int_{\mathcal{M}^3} \ov{v}\cdot\partial_t v \,\mathrm{d}\mu =-z\int_{\mathcal{M}^3} |\nabla_gv|^2 \,\mathrm{d}\mu
	-z\int_{\mathcal{M}^3}
	f(u ) \ov{v}\,\mathrm{d}\mu+z\int_{\mathcal{M}^3}
	\left(\dashint_{\mathcal{M}^3} f(u)  \,\mathrm{d}\mu\right)\ov{v}\,\mathrm{d}\mu.
\end{align*}
Owing to \eqref{int_v}, the last two nonlinear terms simplify as
\begin{align*}
&-z\int_{\mathcal{M}^3}
f(u ) \ov{v}\,\mathrm{d}\mu + z\int_{\mathcal{M}^3}
\left(\dashint_{\mathcal{M}^3} f(u)  \,\mathrm{d}\mu\right)\ov{v}\,\mathrm{d}\mu
= -z\int_{\mathcal{M}^3}
(f(u)-f(m)) \ov{v}\,\mathrm{d}\mu.
\end{align*}
Combining the above identities, we obtain the following energy dissipation identity:
\begin{equation}\label{eq_v}
	\begin{split}
		& \frac{1}{2}\frac{\mathrm{d}}{\mathrm{d}t} \left( \|v (t) \|_{L_{x}^2}^2 \right) +\Re{z}\|\nabla_gv (t) \|_{L_{x}^2}^2
=
	-\Re\left[z\int_{\mathcal{M}^3}
	(f(v+m)-f(m)) \ov{v}\,\mathrm{d}\mu\right].
	\end{split}
\end{equation}

\subsection{\texorpdfstring{$H^{1}$-decay of $v(t)$}{Lx2-decay of v(t)}}

In the following, we establish $H^{1}$-decay of $v(t)$.

\begin{proposition}\label{non_integral}
For initial data $u_0\in H_x^1 \left(\mathcal{M}^3 \right)$, the following decay estimate holds:
\begin{align}\label{eq4.2newv174}
\|v(t)\|_{H_x^1 \left(\mathcal{M}^3 \right)} \leq C e^{-\lambda t}\|u_0\|_{H_x^1 \left(\mathcal{M}^3 \right)},
\end{align}
where $\lambda>0$ is a positive constant depending only on the geometry of $\mathcal{M}^3$ and the parameter $\Re z $,
and $C>0$ depends in addition on $\|u_0 \|_{H^1_x\left(\mathcal{M}^3 \right)}$.
\end{proposition}

\begin{proof}
We begin with the elementary estimate
\begin{align}\label{eq_non}
  \left|\Re\left[z\int_{\mathcal{M}^3}
(f(v+m)-f(m)) \overline{v}\,\mathrm{d}\mu\right]\right|
\lesssim |m|^{p-1}\|v\|_{L_x^2 \left(\mathcal{M}^3 \right)}^2+\|v\|_{L_x^{p+1} \left(\mathcal{M}^3 \right)}^{p+1}.
\end{align}
By the Sobolev inequality and Proposition \ref{pr2.3v122}, we have
\begin{align}\label{eq_v_p}
\|v\|_{L_x^{p+1} \left(\mathcal{M}^3 \right)}^{p+1}
\lesssim \|v\|_{L_x^2 \left(\mathcal{M}^3 \right)}^{p+1}+\|\nabla_g v\|_{L_x^2 \left(\mathcal{M}^3 \right)}^{p+1}
& \lesssim   \left(1+\lambda_1^{- \frac{p+1}2} \right)\|\nabla_g v\|_{L_x^2 \left(\mathcal{M}^3 \right)}^{p+1} \notag \\
      &  \sim  \left(1+\lambda_1^{- \frac{p+1}2} \right)\|\nabla_g u\|_{L_x^2 \left(\mathcal{M}^3 \right)}^{p-1}
    \|\nabla_g v\|_{L_x^2 \left(\mathcal{M}^3 \right)}^{2}.
\end{align}
Gathering \eqref{eq_non} and \eqref{eq_v_p} yields
\[
\begin{aligned}
  \left|\Re\left[z\int_{\mathcal{M}^3}
(f(v+m)-f(m)) \overline{v}\,\mathrm{d}\mu\right]\right|
  \lesssim |m|^{p-1}\|v\|_{L_x^2 \left(\mathcal{M}^3 \right)}^2
  +  \left(1+\lambda_1^{- \frac{p+1}2} \right)\|\nabla_g u\|_{L_x^2 \left(\mathcal{M}^3 \right)}^{p-1}\|\nabla_g v\|_{L_x^2 \left(\mathcal{M}^3 \right)}^{2}.
\end{aligned}
\]
From \eqref{eq4.1v127},
we have $|m(t)|\lesssim \|u(t)\|_{L_x^2 \left(\mathcal{M}^3 \right)}$.
By Proposition \ref{Energydecay}, we obtain
\begin{align}\label{eq4.4newv174}
|m(t)|^{p-1} \lesssim t^{-1} \text{ and } \|\nabla_g u(t)\|_{L_x^2 \left(\mathcal{M}^3 \right)}^{p-1} \lesssim t^{-1}.
\end{align}
Thus, we obtain
\[
\begin{aligned}
  \left|\Re\left[z\int_{\mathcal{M}^3}
(f(v+m)-f(m)) \overline{v}\,\mathrm{d}\mu\right]\right|
  \lesssim t^{-1}\|v\|_{L_x^2 \left(\mathcal{M}^3 \right)}^2
 + \left(1+\lambda_1^{- \frac{p+1}2} \right)t^{-1}\|\nabla_g v\|_{L_x^2 \left(\mathcal{M}^3 \right)}^{2}.
\end{aligned}
\]
Substituting this into \eqref{eq_v} gives
\[
\begin{aligned}
  \frac{\mathrm{d}}{\mathrm{d}t} \left( \|v(t)\|_{L_x^2 \left(\mathcal{M}^3 \right)}^2 \right)
    +2\Re{z}\|\nabla_g v(t)\|_{L_x^2 \left(\mathcal{M}^3 \right)}^2
  \lesssim t^{-1}\|v(t)\|_{L_x^2 \left(\mathcal{M}^3 \right)}^2
  +  \left(1+\lambda_1^{- \frac{ p+1}2} \right)t^{-1}\|\nabla_g v(t)\|_{L_x^2 \left(\mathcal{M}^3 \right)}^{2}.
\end{aligned}
\]
Using Proposition \ref{pr2.3v122} again, we summarize the energy inequality
\begin{align*}
\begin{aligned}
  &  \frac{\mathrm{d}}{\mathrm{d}t} \left( \|v(t)\|_{L_x^2 \left(\mathcal{M}^3 \right)}^2 \right)
  +\Re{z}\lambda_1\|v(t)\|_{L_x^2 \left(\mathcal{M}^3 \right)}^2
  +\Re{z}\|\nabla_g v(t)\|_{L_x^2 \left(\mathcal{M}^3 \right)}^2 \\
  & \leq Ct^{-1}\|v(t)\|_{L_x^2 \left(\mathcal{M}^3 \right)}^2
  +  C\left(1+\lambda_1^{-  \frac{ p+1}2} \right) t^{-1}\|\nabla_g v(t)\|_{L_x^2 \left(\mathcal{M}^3 \right)}^{2}.
\end{aligned}
\end{align*}
Choose $T>0$ such that
\begin{align}\label{condition:largeT}
C T^{-1} \le \frac{1}{2}\Re{z}\lambda_1  \quad \text{  and  }  \quad  C \left(1+\lambda_1^{- \frac{ p+1}2} \right)T^{-1} \le \frac{1}{2}\Re{z}.
\end{align}
Hence, for $t\ge T$, we have
\begin{align*}
\frac{\mathrm{d}}{\mathrm{d}t} \left( \|v(t)\|_{L_x^2 \left(\mathcal{M}^3 \right)}^2 \right)
+\frac{1}{2}\Re z\lambda_1\|v(t)\|_{L_x^2 \left(\mathcal{M}^3 \right)}^2
+\frac{1}{2}\Re z\|\nabla_g v(t)\|_{L_x^2 \left(\mathcal{M}^3 \right)}^2 \le 0.
\end{align*}
By Gr\"onwall's inequality and the mass dissipation identity \eqref{eq_ener}, we have
\[
\|v(t)\|_{L_x^2 \left(\mathcal{M}^3 \right)}
 \leq e^{- \frac{1}{4}\Re z \lambda_1 (t-T)}\|v(T)\|_{L_x^2 \left(\mathcal{M}^3 \right)}
 \leq e^{- \frac{1}{4}\Re z \lambda_1 (t-T)}\|u(T)\|_{L_x^2 \left(\mathcal{M}^3 \right)}
 \leq e^{- \frac{1}{4}\Re z \lambda_1 (t-T)}\|u_0\|_{L_x^2 \left(\mathcal{M}^3 \right)}.
\]
To obtain the $H^1$-decay, we deduce from \eqref{eq4.2v250} the energy dissipation identity
\begin{align*}
\frac{1}{2}\frac{\dd}{\dd t} \left( \|\nabla v(t)\|_{L_x^2 \left(\mathcal{M}^3 \right)}^2 \right)
=& -\Re z \|\Delta_g v(t)\|_{L_x^2 \left(\mathcal{M}^3 \right)}^2
-\Re \left[z\int \nabla\Pi_0^{\perp}\big(f(m+v)-f(m)\big)\cdot \overline{\nabla v}  \, \dd\mu \right].
\end{align*}
Similarly, by H\"older and Sobolev, we have
\begin{align*}
\left|\int_{\mathcal{M}^3} \nabla_{g} \Pi_0^{\perp}\big(f(m+v)-f(m)\big)\cdot \overline{\nabla v}  \, \dd\mu \right|
\lesssim  &  \left\| |\nabla_{g} v|^2 |v+m|^{p-1} \right\|_{L_{x}^1 \left(\mathcal{M}^3 \right)}
\lesssim    \|\nabla_{g} v\|_{L_{x}^6 \left(\mathcal{M}^3 \right)}^2
\|v+m \|_{L_{x}^{p+1} \left(\mathcal{M}^3 \right) }^{p-1}.
\end{align*}
By Proposition \ref{th2.10v116}, we have
$$ \|\nabla_{g} v\|_{L_{x}^6 \left(\mathcal{M}^3  \right)}\lesssim  \|\Delta_{g} v\|_{L_{x}^2 \left(\mathcal{M}^3 \right)},
$$
and by \eqref{eq4.4newv174},
$$ \|v+m\|_{L_{x}^{p+1} \left(\mathcal{M}^3 \right)}^{p-1}\lesssim  t^{-1} .
$$
Thus, we obtain the energy inequality:
\begin{align*}
\frac{\dd}{\dd t} \left( \|\nabla_{g} v(t)\|_{L_x^2 \left(\mathcal{M}^3 \right)}^2 \right)
+ \Re z \|\Delta_{g} v(t)\|_{L_x^2 \left(\mathcal{M}^3 \right)}^2
\lesssim t^{-1} \|\Delta_{g} v(t)\|_{L_x^2 \left(\mathcal{M}^3 \right)}^2, \ \forall\, t \ge 1 .
\end{align*}
The same Gr\"onwall argument yields
\begin{align*}
\|\nabla_{g} v(t)\|_{L_x^2  \left( \mathcal{M}^3 \right)}
\leq e^{- \frac{1}{4}\Re z \lambda_1(t-T)}\|\nabla u(T)\|_{L_x^2 \left(\mathcal{M}^3 \right)},
\end{align*}
for sufficiently large $T>0$ satisfying \eqref{condition:largeT} with a possibly larger implicit constant $C>0$,
depending only on $\|u_0\|_{H_x^1 \left(\mathcal{M}^3 \right)}$.

In summary, the proof of Proposition \ref{non_integral} is now completed.
\end{proof}

By Proposition \ref{non_integral}, we have exponential decay of the fluctuation component.
We now analyze the scalar equation for $m(t)$ to derive the refined asymptotic profile.

\subsection{Asymptotic dynamics of the spatial mean}
Employing the $L_{x}^{2}$ decay of $v(t)$, we next derive an upper bound for the integral $z\int_{\mathcal{M}^3} \left( f(v (t) + m(t) ) -
f(m(t) ) \right) \,\mathrm{d}\mu$ so as to finalize the proof of Theorem \ref{th2v79}.

\begin{lemma}\label{u1_p25}
For initial data $u_0\in H^1$,
 we have
\begin{equation}\label{inf_ex}
	\left| \int_{\mathcal{M}^3}  \left(
f(v (t) + m (t) ) -
f(m  (t) ) \right)\,\mathrm{d}\mu \right| \lesssim e^{-\lambda t}, \ \forall\, t \ge 1 ,
\end{equation}
where $\lambda>0$ is a positive constant depending only on $\mathcal{M}^3$,
and the implicit constant depends also on
 $\|u_0\|_{H^1}$.
\end{lemma}
\begin{proof}
A direct computation yields
\begin{equation*}
	\left|\int_{\mathcal{M}^3}\big(f(v+m)-f(m)\big) \, \mathrm{d}\mu\right|
	\lesssim \int_{\mathcal{M}^3}\left(|v||m|^{p-1}+|v|^p\right)\,\mathrm{d}\mu.
\end{equation*}
This together with \eqref{eq4.2newv174} and \eqref{eq4.4newv174} yields \eqref{inf_ex}.
\end{proof}

Now we are able to prove the dichotomy for fast and slow dynamics. Recall that
$$ m(t)=\dashint_{\mathcal{M}^3} u(t,x)\dd \mu,
$$
where $u(t,x)$ is the global solution to \eqref{eq1.CGL}.
Set
\begin{align}\label{def:rt}
 r(t):=-z\dashint_{\mathcal{M}^3}\big(f(u(t,x))-f(m(t)) \big)\dd\mu.
\end{align}
The key proposition is the following asymptotic result for the slow solution. For the proof, we loosely follow the strategy in \cite{Ben} for semilinear heat equations.
\begin{proposition}\label{prop:dichotomy}
Let $0<\gamma<\frac{\lambda}{p}$, where $\lambda>0$ is the constant in Lemma \ref{u1_p25}.
Assume that
\begin{align}\label{ht:exponentialdecay}
\limsup_{t\rightarrow+\infty}e^{\gamma t}|m(t)|=+\infty.
\end{align}
Then there exist $T_{*}\gg 1$ and a continuous function $\theta(t)$ such that, for all $t\ge T_{*}$,
$$  m(t)=\rho(t)e^{i\theta(t)}, \quad \rho(t)=|m(t)|,
$$
and
\begin{equation}\label{rho_theta}
\begin{cases}
\dfrac{ \mathrm{d} \rho(t)}{ \mathrm{d} t}=-(\Re z)\rho^p(t)+\Re\left(r(t)e^{-i\theta(t)}\right),\\[6pt]
\dfrac{ \mathrm{d} \theta(t)}{  \mathrm{d} t}=-(\Im z)\rho^{p-1}(t)+\Im\left(r(t)e^{-i\theta(t)}\rho^{-1}(t)\right),
\end{cases}
\end{equation}
where
\[
r(t):=-z\dashint_{\mathcal{M}^3}\left(f(v(t)+m(t))-f(m(t))\right)\,  \mathrm{d} \mu,
\]
and if $\Im z\neq 0$,
\[
\lim_{t\rightarrow+\infty}\left(\theta(t)+\frac{\Im z}{(p-1)\Re z}\ln t \right)  \text{ exists}.
\]
Moreover, there exists a bounded function $\nu:[T_*,\infty [ \rightarrow \mathbb{R}_+$ such that
\[\rho(t)=\left(\frac{1}{\nu(t)+\Re z (p-1) (t-T_* ) }\right)^{\frac{1}{p-1}}.
\]
\end{proposition}

\begin{proof}
By \eqref{ht:exponentialdecay}, there exists a sequence $T_j\to\infty$ such that
\[
|m(T_j)|\ge e^{-\gamma T_j}.
\]
We fix $T=T_j$ with $j$ sufficiently large.
From \eqref{eq:initial-mass-decay}  and the Cauchy-Schwarz inequality, we have
\[
|m(T)|\le C T^{-\frac1{p-1}}.
\]
Consequently, setting $N_T:=|m(T)|^{-(p-1)}$ and choosing $\alpha_0 \in \mathbb{R}$ such that $m(T) = |m(T)|e^{i \alpha_0}$, we obtain
\begin{align}\label{estimate_nu}
cT \le N_T \le e^{\gamma(p-1)T}.
\end{align}
We seek continuous modulation functions
\[
\nu:[T,\infty[\to\mathbb R_+,\quad \alpha:[T,\infty[\to\mathbb R
\]
such that
$m(t)=\rho(t)e^{i\theta(t)}$
holds for all $t\ge T$ with
\begin{align*}
& \rho(t)=\left(\nu(t)+  \Re z  (p-1)(t-T)\right)^{-\frac1{p-1}},
\intertext{ and }
& \theta(t)=- \Im z  \int_T^t \frac{ \mathrm{d} \tau}{\nu(\tau)+  \Re z (p-1)(\tau-T)}+\alpha(t).
\end{align*}
The idea behind choosing $\nu(t)$ and $\alpha(t)$ is to modulate the targeted asymptotic profile
\begin{align*}
\rho_c(t):=(\nu+ \Re z (p-1)(t-T) )^{-\frac{1}{p-1}},\quad \theta(t)=- \Im z \int_T^t\frac{\dd\tau}{\nu+ \Re z (p-1)(\tau-T) },
\end{align*}
governed by the model ODE
$$ \frac{ \mathrm{d} m_c (t) }{ \mathrm{d} t} =-zf(m_c (t) ).
$$

Substituting the above expressions into the polar system \eqref{rho_theta}, we obtain the equivalent modulation equations
\begin{align}\label{eq4.15v250}
\begin{cases}
\frac{ \mathrm{d} {\nu (t) }}{ \mathrm{d} t}  =-(p-1)\rho^{-p}\Re\left( r(t)e^{-i\theta(t)}\right),\\[4pt]
\frac{ \mathrm{d} {\alpha (t) }}{ \mathrm{d} t}  =\Im\left(r(t)e^{-i\theta(t)}\rho(t)^{-1}\right),\\[4pt]
\nu(T)=N_T,\\[4pt]
\alpha(T)=\alpha_0.
\end{cases}
\end{align}
Define the map $(\Phi,\Psi)$ by
\begin{align*}
\Phi(\nu,\alpha)(t):= & N_T-(p-1)\int_T^t \varrho( \tau ;\nu)^{-p}\Re\left(r( \tau ) e^{-i\vartheta( \tau ;\nu,\alpha)}\right)\, \mathrm{d} \tau ,
\intertext{ and }
\Psi(\nu,\alpha)(t):= & \alpha_0+\int_T^t \Im\left(r( \tau )e^{-i\vartheta( \tau ;\nu,\alpha)}\varrho( \tau ;\nu)^{-1}\right)\, \mathrm{d}  \tau ,
\end{align*}
where
\begin{align*}
\varrho(t;\nu):= & \left(\nu(t)+ \Re z (p-1)(t-T)\right)^{-\frac1{p-1}},
\intertext{ and }
\vartheta(t;\nu,\alpha):= & - \Im z \int_T^t \frac{ \mathrm{d} \tau}{\nu(\tau)+ \Re z (p-1)(\tau-T)}+\alpha(t).
\end{align*}
Here \(r(t)\), defined by \eqref{def:rt}, is the forcing term already determined by the given global solution \(u\).
Take $T$ large enough and define the closed set in the Banach space $C_t^0  \left([T,\infty[ ;\mathbb R^2 \right)$, equipped with the sum of the two uniform norms,
\[
\mathcal B:=\left\{(\nu,\alpha)\in C^0  \left([T,\infty [ ;\mathbb R^2 \right) \,\big|\,
\nu(T)=N_T, \  \frac12N_T \le \nu(t)\le 2N_T ,\   \alpha(T)=\alpha_0,\ |\alpha(t)-\alpha_0|\le \frac{\pi}{10} \right\}.
\]
For any $(\nu,\alpha)\in\mathcal B$, we have the uniform pointwise bound
\[
\varrho( \tau ;\nu)^{-p}
\le C\bigl(N_T + \tau -T\bigr)^{\frac{p}{p-1}}.
\]
Using the exponential decay $|r(s)|\le Ce^{-\lambda s}$ thanks to Lemma \ref{u1_p25} and the upper bound for $N_T$ from \eqref{estimate_nu},
we get
\[
\begin{aligned}
|\Phi(\nu,\alpha)(t)-N_T|
&\le C\int_T^t \bigl(N_T+ \tau -T\bigr)^{\frac{p}{p-1}} e^{-\lambda  \tau }\, \mathrm{d}  \tau  \\
&\le C e^{-\lambda T}\int_0^\infty (N_T+\tau)^{\frac{p}{p-1}} e^{-\lambda \tau}\, \mathrm{d} \tau
\le C e^{-\lambda T}\left(N_T^{\frac{p}{p-1}}+1\right)  \le C e^{(\gamma p-\lambda)T}.
\end{aligned}
\]
Since $\gamma p<\lambda$, the right hand side decays exponentially to zero as $T\to\infty$. Thus, for $T$ sufficiently large,
\[
|\Phi(\nu,\alpha)(t)-N_T|\le \frac12N_T,
\]
so $\Phi(\nu,\alpha)(t)\in
\left[\frac12N_T,2N_T \right]$.
Similarly, using $\varrho( \tau ;\nu)^{-1}\le C(N_T+ \tau -T)^{ \frac1{p-1} }$, we have
\[
\begin{aligned}
|\Psi(\nu,\alpha)(t)-\alpha_0|
&\le C\int_T^t \bigl(N_T+ \tau -T\bigr)^{\frac1{p-1}} e^{-\lambda  \tau }\, \mathrm{d}  \tau
 \le C e^{-\lambda T}\left(N_T^{\frac1{p-1}}+1\right)
 \le C e^{(\gamma-\lambda)T}.
\end{aligned}
\]
Since the right hand side is exponentially small in $T$, the bound $|\Psi(\nu,\alpha)(t)-\alpha_0|\le \frac{\pi}{10}$ holds for all sufficiently large $T$.
\vspace{0.3cm}

To establish the contraction property, for  $(\nu_1,\alpha_1),(\nu_2,\alpha_2)\in \mathcal B$, we have
\begin{align*}
\|\vartheta(\cdot;\nu_1,\alpha_1)-\vartheta(\cdot;\nu_2,\alpha_2)\|_{L_t^\infty([T,\infty[ )}
\le \|\alpha_1-\alpha_2\|_{L^\infty } +\frac{C}{N_T}\|\nu_1-\nu_2\|_{L_t^\infty([T,\infty [ )},
\end{align*}
and
\begin{align*}
|\varrho( t ;\nu_1)-\varrho( t ;\nu_2) |
\leq C(N_T+ \Re z (p-1)(t-T))^{-\frac{p}{p-1}}\|\nu_1-\nu_2\|_{L^{\infty}([T,\infty[) }, \ \forall\, t \ge T.
\end{align*}
These give
\begin{align*}
& \left\|\Phi(\nu_1,\alpha_1)-\Phi(\nu_2,\alpha_2) \right\|_{L_t^\infty([T,\infty [ )}
+ \left\|\Psi(\nu_1,\alpha_1)-\Psi(\nu_2,\alpha_2) \right\|_{L_t^\infty([T,\infty[ )}  \notag\\
& \lesssim   e^{-\lambda T}\int_0^\infty
\left((N_T+\tau)^{\frac{p}{p-1}}+(N_T+\tau)^{\frac1{p-1}}+1\right)e^{-\lambda\tau}\,\mathrm d\tau
 \left( \|\nu_1-\nu_2\|_{L^\infty([T,\infty [ )}+\|\alpha_1-\alpha_2\|_{L_t^\infty([T,\infty [ )}  \right)\, \notag\\
& \lesssim  e^{-\lambda T}  \left(N_T^{\frac{p}{p-1}}+1 \right)\left( \|\nu_1-\nu_2\|_{L_t^\infty([T,\infty [ )}+\|\alpha_1-\alpha_2\|_{L_t^\infty([T,\infty [ )}  \right)
\notag \\
&   \lesssim   e^{-(\lambda-\gamma p)T}\left( \|\nu_1-\nu_2\|_{L_t^\infty([T,\infty [ )}+\|\alpha_1-\alpha_2\|_{L_t^\infty([T,\infty [ )}  \right).
\end{align*}
Since $\lambda-\gamma p>0$, taking $T$ sufficiently large makes this coefficient less than $ \frac12$.
Thus, $(\Phi,\Psi)$ is a strict contraction on $\mathcal B$.
By the Banach fixed point theorem, there exists a unique pair $(\nu,\alpha)\in\mathcal B$ solving the modulation system \eqref{eq4.15v250} for all $t\ge T$. Since $\nu(t)\ge \frac12N_T>0$ uniformly in time, we have
\[
\nu(t)+ \Re z (p-1)(t-T)\ge \frac12N_T>0
\]
for all finite $t\ge T$. The reconstructed function $m = \rho e^{i\theta}$
satisfies $\frac{ \mathrm{d}  m}{ \mathrm{d} t} =-z f(m )
+r(t)$ with the same initial value as
the actual spatial mean. Since \(p>1\), the map \(w\mapsto f(w) \) is locally Lipschitz on \(\mathbb C\).
Hence, the forced scalar ODE has a unique solution for the prescribed initial value,
 and the reconstructed solution coincides with the original mean \(m\). This establishes the polar
decomposition on $[T,\infty[$.

To end the proof, we analyze the asymptotic behavior of the phase as $ t \rightarrow +\infty$. From the boundedness of $\nu(t)$, we have
\[
\frac{1}{\nu( t )+ \Re z (p-1)( t -T)}  =\frac{1}{ \Re z (p-1) t }+O \left(t^{-2} \right).
\]
This implies that
\begin{align}\label{theta}
\begin{aligned}
 & \theta(t)+\frac{ \Im z }{ \Re z (p-1)}\ln  t \\
 = & - \Im z \int_T^t\left(\frac{1}{\nu(\tau)+ \Re z (p-1)(\tau-T)}
-\frac{1}{ \Re z (p-1)\tau}\right) \, \mathrm{d} \tau +\frac{ \Im z }{ \Re z (p-1)}\ln  T+\alpha(t).
\end{aligned}
\end{align}
The integrand in  \eqref{theta} is integrable as it is $O \left(\tau^{-2} \right)$. Moreover,
\[
 \left| \frac{ \mathrm{d} {\alpha}(t)}{ \mathrm{d} t}  \right|\lesssim   e^{-\lambda t}\bigl(N_T+t-T\bigr)^{\frac1{p-1}},
\]
which is integrable on $[T,\infty[$. Hence, $\alpha(t)$ converges as $t\to\infty$. Consequently,
$\theta(t)+\frac{ \Im z }{ \Re z (p-1)}\ln t$ admits a finite limit $\theta_\infty\in\mathbb R$.
The boundedness of $\nu(t)$ also yields the asymptotic expansion
\[
\rho(t)=\left(\nu(t)+ \Re z (p-1)(t-T)\right)^{-\frac1{p-1}}
=\left( \Re z (p-1)t\right)^{-\frac1{p-1}} \left(1+O \left(t^{-1} \right)\right).
\]
This completes the proof of Proposition \ref{prop:dichotomy}.
\end{proof}

\subsection{Existence of fast and slow solutions}

The existence of a slow solution can be verified directly by taking the initial datum to be a constant function.
So we are going to construct a global solution $u(t)$ with the decomposition
$$ u(t)=m(t)+v(t),\quad m(t)=\dashint_{\mathcal{M}^3}u(t,x) \, \mathrm{d} \mu,
$$
with the prescribed exponential decay rate
$$ m(t)=O \left(e^{-\gamma t} \right),\text{ as }  t\rightarrow+\infty.
$$
Let $v_0\in H_{x}^1 \left(\mathcal{M}^3 \right)$ satisfy $$ \dashint_{\mathcal{M}^3}v_0 \, \mathrm{d} \mu=0.
$$
We then consider the fixed-point problem for $(m,v)$:
\begin{align}\label{eq:fixedpointmv}
\begin{cases}
& m(t)=z\int_t^{+\infty}
f( m(\tau))  \, \dd \tau
+z\int_t^{\infty} \dashint_{\mathcal{M}^3} \left(f(m(\tau)+v(\tau))-f(m(\tau)) \right) \, \mathrm d\mu\,\dd\tau,\\
& \\
& v(t)=e^{t z\Delta_g }v_0-z\int_0^te^{z(t-\tau)\Delta_g }\left(f(m(\tau)+v(\tau))-\dashint_{\mathcal{M}^3}f \left(m(\tau)+v(\tau) \right) \right)\, \dd \tau.
\end{cases}
\end{align}
\begin{proposition}\label{pr4.4v210}
For any $0<\gamma<\lambda< \Re z \lambda_1$\footnote{Recall that $\lambda_1>0$ is the first non-zero
eigenvalue of $-\Delta_g$. }, there exists $\epsilon_0>0$ such that
every zero-mean $v_0\in H^1 \left(\mathcal M^3 \right)$ with $\|v_0\|_{H^1}\le\epsilon_0$
admits a solution $(m,v)\in C_t^0  \left(\mathbb R_+;\mathbb C\times H^1_x  \right)$ of \eqref{eq:fixedpointmv} satisfying
\[
|m(t)|\lesssim  e^{-\gamma t}\|v_0\|_{H^1_x }^{p},\qquad  \|v(t)\|_{H^1_x }\lesssim  e^{-\lambda t}\|v_0\|_{H_x^1}.
\]
In particular, $u=m+v$ solves \eqref{eq1.CGL} and $\|u(t)\|_{H^1}\lesssim  e^{-\gamma t}\|v_0\|_{H^1_x }$.
\end{proposition}

\begin{remark}\label{prop_fast_decay}
Under the preceding assumptions, the construction determines a mean $m(0)$ satisfying
\begin{align}\label{eq4.19v250}
|m(0)|\lesssim \|v_0\|_{H^1_x }^{p}.
\end{align}
The initial datum $u_0=m(0)+v_0$ therefore generates the above exponentially decaying solution.
The mean is part of the construction;
it cannot be prescribed independently of $v_0$.
However, the smallness of an arbitrary initial value alone does not imply exponential decay, as the following example shows.
Let $\epsilon > 0  $. We consider the solution $u(t)$ of \eqref{eq1.CGL}
 with $ u (0)= \epsilon $.
The solution $u (t)$ satisfies
\[
| u (t)| = \left( \Re z (p-1)t + | \epsilon |^{-(p-1)} \right)^{- \frac1{p-1}}
\sim \left( \Re z (p-1)t \right)^{- \frac1{p-1} }, \text{ as } t\to\infty.
\]
Exponential decay therefore fails whenever $ \epsilon \neq 0$.

\end{remark}

\begin{proof}[Proof of Proposition \ref{pr4.4v210}]
Choose $\gamma_0$ such that
\[
\max\{\gamma,(2-p)_+\lambda\}<\gamma_0<\lambda.
\]
We apply the Banach fixed point theorem directly to the system \eqref{eq:fixedpointmv}.
For pairs \((m,v)\) with \(\Pi_0v=0\), set
\[
\begin{aligned}
\|(m,v)\|_*
:=\;&\sup_{t\ge0}e^{\gamma_0t}|m(t)|
+\sup_{t\ge0}e^{\lambda t}\|v(t)\|_{H^1_x} +\left(\int_0^\infty e^{2\lambda t}\|v(t)\|_{H^2_x}^2\,\mathrm dt\right)^{ \frac12},
\end{aligned}
\]
and let $B_R = \left\{ (m, v) \in C_t^0  \left( \mathbb{R}_+; \mathbb{C} \times H_x^1 \right)  \,\big|\,  \|(m,v)\|_*\le R  \right\}$.
Denote by \(\mathcal T\) the map defined by the right hand side of \eqref{eq:fixedpointmv}.

Assume \((m,v)\in B_R\). By the Sobolev inequality and the compactness of \(\mathcal M^3\), we obtain
\[
\left|\Pi_0\bigl(f(m+v)-f(m)\bigr)\right|
\lesssim
|m|^{p-1}\|v\|_{L^p}+\|v\|_{L^p}^p
\lesssim R^p
\left(
e^{-((p-1)\gamma_0+\lambda)t}+e^{-p\lambda t}
\right).
\]
This together with
\[
\int_t^\infty |m(\tau)|^p\,\mathrm d\tau
\lesssim R^p e^{-p\gamma_0t}
\]
yields that the first component of \(\mathcal T(m,v)\) has weighted norm
bounded by \(CR^p\).
\vspace{0.3cm}

The second component of $\mathcal T(m,v)$ satisfies
\begin{align}\label{eq4.19v271}
\partial_t v = z \Delta_g v -z\Pi_0^\perp\bigl(f(m+v)-f(m)\bigr),
\end{align}
with initial datum $v_0$.
Taking the real $L^2$ inner product of \eqref{eq4.19v271} with $-e^{2\lambda t}\Delta_g$ applied to the second component,
 and using Proposition \ref{th2.10v116} and Young's inequality, we obtain
\begin{align}\label{maximalregularity}
\|\mathcal T(m,v)\|_*
\lesssim \|v_0\|_{H_x^1}+R^p
+\left\|e^{\lambda t}\Pi_0^\perp\bigl(f(m+v)-f(m)\bigr)\right\|_{L_{t,x}^2}.
\end{align}
Note that
\[
\|f(m+v)-f(m)\|_{L_x^2}
\lesssim |m|^{p-1}\|v\|_{L_x^2}+\|v\|_{L_x^{2p}}^p.
\]
By the Sobolev, interpolation, and H\"older inequalities, we obtain
\[
\left\|e^{\lambda t}
\Pi_0^\perp\bigl(f(m+v)-f(m)\bigr)\right\|_{L_{t,x}^2}
\lesssim R^p.
\]
Consequently, \eqref{maximalregularity} yields
\[
\|\mathcal T(m,v)\|_*  \le C_0\|v_0\|_{H_x^1}+C_1R^p.
\]

The same pointwise estimates for the power nonlinearity give, for any $(m_1,v_1),(m_2,v_2) \in B_R$,
\[
\|\mathcal T(m_1,v_1)-\mathcal T(m_2,v_2)\|_*
\le
C_2R^{p-1}
\|(m_1-m_2,v_1-v_2)\|_* .
\]
The only additional point when \(1<p<2\) is the
variation of the mean in the fluctuation forcing. In that case,
\[
\begin{aligned}
&e^{\lambda t} \left\| \bigl(f(m_1+v)-f(m_1)\bigr) -\bigl(f(m_2+v)-f(m_2)\bigr) \right\|_{L_x^2}
\leq   CR^{p-1} e^{-(\gamma_0-(2-p)\lambda)t} \|(m_1-m_2,0)\|_*,
\end{aligned}
\]
which belongs to \(L_t^2 ([ 0,\infty [ ) \) precisely because
\(\gamma_0>(2-p)\lambda\).

Choose
$R:=2C_0\|v_0\|_{H^1_x}$. If $\|v_0\|_{H^1_x }\leq \epsilon_0$ is sufficiently small, then
$$ C_0\|v_0\|_{H^1}+C_1R^p\leq R,\quad C_2R^{p-1}<\frac{1}{2}.
$$
Hence \(\mathcal T\) maps \(B_R\) into itself, is a contraction, and has a unique fixed point $(m,v)$ satisfying
\[
|m(t)|\lesssim R^p e^{-\gamma_0t} \text{ and } \|v(t)\|_{H_x^1}\lesssim R e^{-\lambda t}.
\]
Since \(\gamma_0>\gamma\), these are the asserted estimates. Moreover, differentiating \eqref{eq:fixedpointmv} shows that
\(u=m+v\) solves \eqref{eq1.CGL}, and
\[
\|u(t)\|_{H_x^1} \lesssim \|v_0\|_{H_x^1}e^{-\gamma t}.
\]
Furthermore, \eqref{eq4.19v250}
holds as stated in Remark~\ref{prop_fast_decay}.
This completes the proof of Proposition \ref{pr4.4v210}.
\end{proof}

\section{Asymptotic decay of the solutions to CGL on some noncompact three-dimensional Riemannian manifolds}\label{se5}

The analysis carried out in the preceding sections crucially hinges on the compactness of the underlying Riemannian manifold, most notably via Poincar\'e's inequality and the spectral gap property of the Laplace-Beltrami operator.
For noncompact Riemannian manifolds, such tools are generally no longer accessible,
rendering the asymptotic behavior of solutions considerably more intricate.
In particular, convergence towards spatially homogeneous steady states can no longer be anticipated in general.
Accordingly, we adopt a more modest objective in this section: under appropriate geometric assumptions (Assumption \ref{assump:A}) on the underlying Riemannian manifold,
we establish the decay to zero in time for global solutions to CGL.
We also show that there is no uniform decay rate for CGL on $\mathcal{M}^3$ when the bottom of the spectrum of $- \Delta_g$ is $0$,
whereas exponential decay holds when it is positive.
This highlights the difference between the asymptotic dynamics of solutions to the CGL equation on noncompact and compact Riemannian manifolds.

\subsection{Decay of global solutions to CGL on some noncompact three-dimensional Riemannian manifolds}

In this subsection, we prove that the global solution of the equation \eqref{eq1.CGL} on a class of noncompact three-dimensional Riemannian manifolds that satisfy Assumption \ref{assump:A} below converges to zero in $H_{x}^1$ as time tends to infinity.

Before presenting our main results in this subsection, we state the structural assumption imposed on the underlying Riemannian manifolds. The geometric conditions introduced below are formulated to ensure adequate control over the linear heat semigroup and rule out undesirable volume-growth behaviors at infinity.

\begin{assumption}\label{assump:A}
The three-dimensional Riemannian manifold $ \left(\mathcal{M}^3, g \right)$
 satisfies the following conditions:
\begin{enumerate}
\item There exists a constant $K>0$ such that the Ricci curvature satisfies $\mathcal{R}_{ij}\ge -K$ on every geodesic ball $B(p,\rho)\subset \mathcal{M}^3$, for all $p\in\mathcal{M}^3$ and $\rho>0$.
\item The sectional curvature is uniformly bounded: $|\mathrm{Rm}_g|\le K'$ for some constant $K'>0$.
\item There exists a constant $c>0$ such that the volume of every geodesic ball satisfies $\mathrm{Vol}_g B(p,r)\ge c r^3$, for any $p\in\mathcal{M}^3$ and $r>0$.
\item The homogeneous Sobolev inequality $\|\phi\|_{L^6 \left(\mathcal M^3 \right)}\le C\|\nabla_g\phi\|_{L^2 \left(\mathcal M^3 \right)}$ holds for all $\phi\in C_c^\infty \left(\mathcal M^3 \right)$, and the Sobolev inequalities used below hold.
\end{enumerate}
\end{assumption}

\begin{remark}
Within the class of noncompact three-dimensional Riemannian manifolds, Assumption~\ref{assump:A} covers a broad and geometrically relevant family of examples, including the Euclidean space $\mathbb R^3$, hyperbolic space $\mathbb H^3$, asymptotically Euclidean spaces, as well as Cartan-Hadamard manifolds with uniformly bounded sectional curvature.
Conditions (1)--(4) ensure the Sobolev inequalities and Gaussian upper bounds for the heat kernel used below.

\end{remark}

The integral kernel associated with the semigroup $ e^{ z  \Delta_g} $ is $p_{z } (x, y) = p(z , x,y )$.
We next recall the upper bound estimate for the heat kernel $p_t$ on $\mathcal{M}^{3 }$ from \cite{LY}, where
\begin{align*}
e^{t \Delta_g} \phi(x)=\int_{\mathcal{M}^3} p_t(x,y ) \phi(y ) \,\mathrm{d} \mu (y ).
\end{align*}
Let $d_g(x,y)$ stand for the geodesic distance between two points $x,y\in \mathcal{M}^3 $.

\begin{theorem}\label{thm:real_heat_kernel}
Let $\left(\mathcal{M}^3 ,g \right)$ satisfy Assumption \ref{assump:A}. Then its heat kernel satisfies the following bound
\begin{align*}
p_t(x,y) \lesssim t^{- \frac{3 }2} e^{- \frac{d_g^2(x,y)}{ct}}, \ \forall\, t > 0,
\end{align*}
where $c$ is a positive constant.
\end{theorem}
We next present a kernel estimate for the linear Ginzburg-Landau propagator.
\begin{theorem}[\cite{Coulhon-Sikora-2008}]\label{thm:complex_heat_kernel}
Suppose that
\begin{align*}
p_t(x,x) \le K t^{-\frac{3 }2}, \ \forall\, t  > 0  , x \in  \mathcal{M}^3 ,
\end{align*}
for some $K> 0$. Then for any $x , y \in \mathcal{M}^3 $,
\begin{align}\label{eq1.4v3}
\left| p_{ t z  } (x,y)\right|
\lesssim K  t^{-\frac{3 }2} ( \Re z)^{- \frac{3 }2} \left(1 + \Re \left(  \frac{d_g^2(x,y)}{4 z t } \right)  \right)^\frac{3 }2
 e^{- \Re \frac{d_g^2(x,y)}{4  z t} }
\end{align}
for any
$t > 0$.
\end{theorem}

By Theorem~\ref{thm:complex_heat_kernel}, we obtain the spacetime mapping properties of the linear Ginzburg-Landau propagator needed below.
To facilitate the subsequent decay analysis for the CGL equation, we present the following $L_{x}^r-L_{x}^p$ dispersive estimate.
\begin{lemma}\label{le3.1v16}
For $1\le r\le 2 \le p\le\infty$ and $\phi\in L^r \left(\mathcal M^3 \right)$,
\[
\left\|e^{tz\Delta_g}\phi \right\|_{L^p}
\lesssim t^{-\frac32 \left(\frac1r-\frac1p \right)}\|\phi\|_{L^r}, \qquad \forall\, t>0.
\]
\end{lemma}
To set up the spacetime bounds for the linear Ginzburg-Landau propagator, we introduce the standard notion of a Strichartz admissible pair.

\begin{definition}\label{de3.1v16}
The exponent pair $(q, p)$ is called a Strichartz admissible pair if
\begin{align*}
\frac1q = \frac{3 }2  \left( \frac12- \frac1p  \right),
\end{align*}
where $ 2 \le q, p  \le \infty$.
\end{definition}

With the Strichartz admissible pair defined in Definition~\ref{de3.1v16}, we now state the full set of Strichartz estimates on the noncompact three-dimensional Riemannian manifolds considered here.
\begin{proposition}[Strichartz estimates]\label{pr2.1v6}

\begin{enumerate}
\item  If $\phi \in L_{x}^2 \left(  \mathcal{M}^3  \right)$, then for any Strichartz admissible pair $(q, p )$, we have

\begin{align}\label{eq2.25v6}
 \left\| e^{t  z\Delta_g } \phi \right\|_{L_t^q L_x^p \left( \mathbb{R}_+  \times \mathcal{M}^3  \right) }
 \lesssim  \|\phi \|_{L_{x}^2\left(\mathcal{M}^3  \right)}.
\end{align}
\item Let $t_0\ge0$. If $F \in L_t^{q_1'} L_x^{p_1'} \left( [t_0,\infty [  \times \mathcal{M}^3  \right)$,
then for any pairs $(q,p)$ and $(q_1, p_1)$ satisfying $1 \le p_1' \le 2 \le p \le \infty$ and $1 < q_1' < q < \infty$, \begin{align}\label{eq1.8v61}
\left( \frac1{q_1'} - \frac1q  \right) + \frac{3  }2  \left( \frac1{p_1'} - \frac1p  \right) = 1,
\end{align}
we have
\begin{align}\label{eq2.26v6}
\left\| \int_{t_0 }^t e^{(t- \tau ) z \Delta_g } F( \tau ) \,\mathrm{d} \tau  \right\|_{L_t^q L_x^p \left( [t_0,\infty [   \times \mathcal{M}^3  \right)}
 \lesssim \|F\|_{ L_t^{q_1'} L_x^{p_1'}  \left( [t_0,\infty [   \times \mathcal{M}^3  \right)}.
\end{align}

\end{enumerate}

\end{proposition}

\begin{remark}
The condition \eqref{eq1.8v61} is weaker than the Strichartz admissibility of $(q,p)$ and $(q_1,p_1)$.
\end{remark}

\begin{proof}[Proof of Proposition \ref{pr2.1v6}]
(1) Let $(q,p)$ be a Strichartz admissible pair. By Lemma \ref{le3.1v16}, the semigroup $e^{tz \Delta_g}$ satisfies the uniform $L_{x}^2$ bound
\begin{align}\label{eq3.1v61}
\left\| e^{tz\Delta_g} \phi \right\|_{L_x^2 \left(\mathcal{M}^3  \right)}
\lesssim \|\phi\|_{L^2 \left(\mathcal{M}^3  \right)},\quad \forall\,t\in \mathbb{R}_+ ,
\end{align}
and also the decay estimate
\begin{align}\label{eq3.2v61}
\left\| e^{ (t - \tau )  z \Delta_g} \phi \right\|_{L^\infty \left(\mathcal{M}^3  \right)}
\lesssim |t- \tau |^{-\frac{3 }{2}} \|\phi\|_{L^1 \left(\mathcal{M}^3  \right)},\quad \forall\,t \ge  \tau \ge 0.
\end{align}
Combining \eqref{eq3.1v61} and \eqref{eq3.2v61}, we apply Theorem 1.2 from \cite{KT} with the semigroup $U(t)=e^{tz\Delta_g}$ for $t \ge 0$,
$U(t) U( \tau )^* = e^{( t- \tau ) z \Delta_g}$ for $t > \tau \ge 0$,
the Hilbert space $H=L_x^2 \left(\mathcal{M}^3 \right)$ and the manifold $X=\mathcal{M}^3 $, which yields the bound \eqref{eq2.25v6}.

(2) Next, suppose the exponent pairs $(q,p)$ and $(q_1,p_1)$ satisfy $1 \le p_1' \le 2 \le p \le \infty$, $1 < q_1' < q < \infty$ and \eqref{eq1.8v61}.
 By Lemma \ref{le3.1v16}, for $t>\tau\ge t_0$,
 we obtain
\begin{align*}
\left\| e^{(t- \tau )z\Delta_g}F( \tau ,\cdot) \right\|_{L_x^p \left(\mathcal{M}^3  \right)}
\lesssim |t- \tau |^{-\frac{3 }{2}\left(\frac{1}{p_1'} - \frac{1}{p}\right)} \|F( \tau ,\cdot)\|_{L_x^{p_1'} \left(\mathcal{M}^3  \right)}.
\end{align*}
We then invoke the Hardy-Littlewood-Sobolev fractional integral inequality to control the convolution in time:
\begin{align*}
\left\|\int_{t_0}^t e^{(t- \tau )z\Delta_g}F( \tau ,\cdot)\,\mathrm{d} \tau  \right\|_{L_t^q L_x^p \left( [t_0,\infty [  \times\mathcal{M}^3  \right)}
& \lesssim \left(\int_{t_0}^\infty   \left( \int_{t_0}^t \left\| e^{(t- \tau )z\Delta_g}F( \tau ,\cdot) \right\|_{L_x^p \left(\mathcal{M}^3  \right)} \mathrm{d} \tau  \right)^q \mathrm{d}t\right)^{ \frac1q} \\
& \lesssim \left(\int_{t_0}^\infty    \left( \int_{t_0}^t |t- \tau |^{-\frac{3 }{2}\left(\frac{1}{p_1'} - \frac{1}{p}\right)}
\|F( \tau ,\cdot)\|_{L_x^{p_1'} \left(\mathcal{M}^3  \right)} \mathrm{d} \tau \right)^q \mathrm{d}t\right)^{ \frac1q} \\
& \lesssim \|F\|_{L_t^{q_1'} L_x^{p_1'} \left( [t_0,\infty  [  \times\mathcal{M}^3  \right)}.
\end{align*}
This proves \eqref{eq2.26v6}.

\end{proof}

Following the argument in \cite{MZ,Miao2}, we record the local well-posedness theory needed below,
including local existence and the blowup criterion. \begin{theorem}[Local Well-posedness] \label{thm:wellposedness}
Assume that $\mathcal M^3$ obeys Assumption~\ref{assump:A}, $ \frac73\le p\le5$, and $u_0\in H_x^1 \left(\mathcal M^3 \right)$.
Then the following conclusions hold:

\begin{enumerate}
\item[(1)] \textbf{Local existence.}
There exists $\epsilon_0>0$ such that if
\begin{align*}
\left\|e^{tz\Delta_g}u_0 \right\|_{L_{t,x}^{\frac{5(p-1)}2}  \left(\mathbb R_+\times\mathcal M^3 \right)}\leq\epsilon_0,
\end{align*}
then
\begin{align*}
\|u\|_{L_{t,x}^{\frac{5(p-1)}2} \left(\mathbb R_+\times\mathcal M^3 \right)}
\leq 2\epsilon_0.
\end{align*}
This holds in particular when $\|u_0\|_{\dot H^{s_c}}$ is sufficiently small, where $ s_c=\frac32-\frac2{p-1}$.

\item[(2)] \textbf{Maximal-lifespan solution.}
There exists a unique maximal-lifespan solution $u\in C_t^0 H_x^1\cap L^2_{t,\mathrm{loc}}H_x^2$ to the equation \eqref{eq1.CGL} on
$[0, T_{\max}(u_0)[\times \mathcal{M}^3$ such that
\[
\|u\|_{L_{t,x}^{\frac{5(p-1)}2} \left([0,t[\times\mathcal M^3 \right)}<\infty,
\qquad\forall\,t<T_{\max}(u_0).
\]
If $T_{\max}(u_0) < \infty$, then
\[
\|u\|_{L_{t,x}^{\frac{5(p-1)}{2}} \left([0, T_{\max}(u_0)[ \times \mathcal{M}^3 \right)} = \infty.
\]

\end{enumerate}
\end{theorem}

The Riemannian manifold $\mathcal{M}^3$ under Assumption~\ref{assump:A} satisfies \eqref{eq1.5v177}.
Thus, by Proposition \ref{prop:global} and Remark \ref{re4v177},
the solution $u$ to the equation \eqref{eq1.CGL} on a noncompact Riemannian manifold $\mathcal{M}^3$ under Assumption~\ref{assump:A} is global.
 Thus, we have
\begin{proposition}
Assume that $\mathcal M^3$ obeys Assumption~\ref{assump:A}, $ \frac73\le p\le5$, and $u_0\in H_x^1 \left(\mathcal M^3 \right)$.
The solution $u$ to equation \eqref{eq1.CGL} on the Riemannian manifold $\mathcal{M}^3$ under Assumption~\ref{assump:A} is global.
\end{proposition}

We now turn to the proof of Theorem \ref{th3.1v3}. We first establish the decay estimate of $\| \nabla_g u(t)\|_{L^2 \left(\mathcal{M}^3 \right)}$.

\begin{proposition}\label{limit_u_1}
Under the assumptions of Theorem~\ref{th3.1v3}, we have
\begin{align}\label{lim_grad_noncompact}
 \lim_{t\to\infty}\|\nabla_gu(t)\|_{L_{x}^2 \left(\mathcal M^3 \right)}=0.
\end{align}
\end{proposition}

\begin{proof}
We start by observing that
 $E(u(t))$ is nonnegative and nonincreasing by \eqref{eq_mass}. On the other hand, \eqref{eq_ener} gives
\[
\int_0^\infty \left( \|\nabla_g u(t)\|_{L_x^2  \left(\mathcal M^3 \right)}^2
+ \|u(t)\|_{L_x^{p+1} \left(\mathcal M^3 \right)}^{p+1} \right) \mathrm dt < \infty.
\]
Since $E(u(t))$ is comparable to the integrand above, it follows that
\[
E(u(t)) \to 0 \quad \text{as } t\to\infty.
\]
This implies \eqref{lim_grad_noncompact}.
\end{proof}

Before proving the asymptotic $L_x^2$-decay of the solution to equation \eqref{eq1.CGL} on the Riemannian manifold $\mathcal{M}^3$ under Assumption~\ref{assump:A},
 we establish the global spacetime bound \eqref{eq6.2v51}.

\begin{proof}[Proof of \eqref{eq6.2v51}]
We establish \eqref{eq6.2v51} by separately treating the cases \(p= \frac73\)
and \(p> \frac73\).

\noindent
\textbf{Case 1: \(p= \frac73\).}
 \eqref{eq_ener} directly yields
\eqref{eq6.2v51}.

\noindent\textbf{Case 2: \(p> \frac73\).}
By Theorem~\ref{thm:wellposedness}, we have
\begin{align}\label{local}
\|u\|_{L_{t,x}^{\frac{5(p-1)}{2}}([0,t]\times\mathcal M^3)}<\infty,
\quad \forall\, t<\infty.
\end{align}
By
 \eqref{eq_ener} and \eqref{eq_mass}, we obtain
\[
\sup_{t\ge0}\|u(t)\|_{H_x^1 \left(\mathcal M^3 \right)}<\infty.
\]
Here the local spacetime bound is used only on finite intervals for the already constructed global solution.
From Proposition~\ref{limit_u_1},
for any \(\eta>0\), there exists \(T_\eta>0\) such that \(\|\nabla_g u(t)\|_{L_x^2 \left( \mathcal{M}^3 \right) }\le\eta\) for all \(t\ge T_\eta\).
Choose \(t_0\ge T_\eta\). Then \(\|\nabla_g u(t_0)\|_{L_x^2\left( \mathcal{M}^3 \right) }\le\eta\).
Applying the Gagliardo-Nirenberg inequality, we get
\[
\|u(t_0)\|_{\dot H_x^{s_c} \left( \mathcal{M}^3 \right) }
\lesssim \|u(t_0)\|_{L_x^2 \left( \mathcal{M}^3 \right) }^{1-s_c}\|\nabla_g u(t_0)\|_{L_x^2 \left( \mathcal{M}^3 \right) }^{s_c}
\le \|u_0\|_{L_x^2 \left( \mathcal{M}^3 \right) }^{1-s_c}\eta^{s_c},
\]
where the last inequality follows from \eqref{eq_ener}.
Upon taking \(\eta\) sufficiently small, \(\|u(t_0)\|_{\dot H_x^{s_c} \left( \mathcal{M}^3 \right) }\) becomes arbitrarily small.
Invoking Theorem~\ref{thm:wellposedness}, we obtain
\begin{align*}
\|u\|_{L_{t,x}^{\frac{5(p-1)}{2}} \left([t_0,\infty[\times\mathcal M^3 \right)}<\infty.
\end{align*}
By \eqref{local}, we also have \(\|u\|_{L_{t,x}^{\frac{5(p-1)}{2}} \left([0,t_0[\times\mathcal M^3 \right) }<\infty\).
Combining the bounds over \([0,t_0[\) and \([t_0,\infty[\), we obtain
\[
\|u\|_{L_{t,x}^{\frac{5(p-1)}{2}} \left(\mathbb R_+\times\mathcal M^3 \right)}<\infty,
\]
which establishes \eqref{eq6.2v51}.
\end{proof}

\begin{proposition}\label{limit_u_2}
Under the assumptions of Theorem~\ref{th3.1v3}, we have
\begin{align}\label{lim_L2_noncompact}
\lim_{t\to\infty}\|u(t)\|_{L_{x}^2 \left(\mathcal{M}^3 \right)} = 0.
\end{align}
\end{proposition}

\begin{proof}
We first show \(u\in L^{ \frac{ 10}3}_{t,x} \left(\mathbb{R}_+\times\mathcal{M}^3 \right)\).
For any finite time interval \(I=[a,b]\), Duhamel's formula combined with \eqref{eq2.25v6}, \eqref{eq2.26v6} and H\"older's inequality yields
\[
\begin{aligned}
\|u\|_{L^{ \frac{ 10}3}_{t,x} \left(I \times \mathcal{M}^3 \right)}
 \leq C\|u_0\|_{L_x^2 \left(\mathcal{M}^3 \right)}
+ C\|u\|_{L^{\frac{5(p-1)}2}_{t,x} \left(I \times \mathcal{M}^3  \right)}^{p-1} \|u\|_{L^{ \frac{ 10}3}_{t,x} \left(I \times \mathcal{M}^3  \right)}.
\end{aligned}
\]
From \eqref{eq6.2v51}, we have \(\|u\|_{L^{\frac{5(p-1)}2}_{t,x} \left(\mathbb{R}_+\times\mathcal{M}^3 \right)}<\infty\).
We may therefore partition \(\mathbb{R}_+\) into a finite family of consecutive subintervals \(I_j\), the last of which is unbounded, such that
$
C\|u\|_{L^{\frac{5(p-1)}2}_{t,x} \left(I_j \times \mathcal{M}^3  \right)}^{p-1} \leq \frac12  \text{ for all }j.
$
The local $H^1_x$ bound ensures that the $L_{t,x}^{ \frac{10}3}$ norm is finite on bounded intervals,
 so absorption is legitimate. On the final unbounded interval,
 first apply the estimate to its finite truncations and then use monotone convergence.
 Absorbing on each $I_j$ gives $\|u\|_{L_{t,x}^{ \frac{10}3} \left(I_j\times\mathcal M^3 \right)}\le 2C\|u_0\|_{L^2_x  \left( \mathcal{M}^3 \right)}$.
 Since there are only finitely many intervals, summing their $\frac{10}3$-powers gives the global spacetime regularity
\(u\in L^{ \frac{10}3}_{t,x} \left(\mathbb{R}_+\times\mathcal{M}^3 \right)\).
As a consequence, the nonlinear term \(f(u)\) belongs to \(L^{ \frac{10}7}_{t,x} \left(\mathbb{R}_+\times\mathcal{M}^3 \right)\). In particular,
\begin{align}\label{eq:nonlinear_tail}
\|f(u)\|_{L^{ \frac{ 10}7}_{t,x} \left([T,\infty[ \times\mathcal{M}^3 \right)} \to 0
\qquad \text{as } T\to\infty.
\end{align}
For any fixed \(T>0\), Duhamel's formula gives
\begin{align}\label{eq:duhamel_noncompact}
u(t) = e^{(t-T)z\Delta_g}u(T) - z\int_T^t e^{(t- \tau )z\Delta_g} f(u(\tau))\,\mathrm{d}\tau .
\end{align}
Applying \eqref{eq2.25v6}, we obtain
\begin{align}\label{eq:duhamel_tail_L2}
\sup_{t\ge T}\left\| \int_T^t e^{(t- \tau )z\Delta_g} f(u(\tau))\, \mathrm{d}\tau  \right\|_{L_x^2 \left(\mathcal{M}^3 \right)}
\lesssim \|f(u)\|_{L^{ \frac{ 10}7}_{t,x} \left([T,\infty[ \times\mathcal{M}^3 \right)}.
\end{align}
We next analyze the linear complex Ginzburg--Landau propagator.
By the spectral theorem,
\[
\left\|e^{(t-T)z\Delta_g}u(T) \right\|_{L_x^2 \left(\mathcal{M}^3 \right)}^2
= \int_{0 }^\infty e^{-2\Re z\,(t-T)\xi}\,\mathrm{d}\mu_{u(T)}(\xi),
\]
where \(\mathrm{d}\mu_{u(T)}\) is the spectral measure associated with \(u(T)\).
 The dominated convergence theorem and \(\ker(-\Delta_g)\cap L_x^2(\mathcal M^3)=\{0\}\) gives
\begin{align}\label{eq:linear_semigroup_decay}
\lim_{t\to\infty} \left\|e^{(t-T)z\Delta_g}u(T) \right\|_{L^2_x \left(\mathcal{M}^3 \right)}
= \left\|P_{\ker(-\Delta_g)}u(T)\right\|_{L^2_x \left(\mathcal{M}^3 \right)} = 0 .
\end{align}
Combining \eqref{eq:duhamel_noncompact}, \eqref{eq:duhamel_tail_L2}, and \eqref{eq:linear_semigroup_decay}, we obtain
\begin{align}\label{eq:l2_limsup}
\limsup_{t\to\infty}\|u(t)\|_{L^2_x  \left(\mathcal{M}^3 \right)}
\lesssim \|f(u)\|_{L^{ \frac{ 10}7}_{t,x} \left([T,\infty[ \times\mathcal{M}^3 \right)}.
\end{align}
Letting \(T\to\infty\) in \eqref{eq:l2_limsup} and using \eqref{eq:nonlinear_tail} yields \eqref{lim_L2_noncompact}. This completes the proof.
\end{proof}

\begin{proof}[End of the proof of Theorem \ref{th3.1v3}]
The desired estimate \eqref{eq6.3v51} follows from \eqref{lim_grad_noncompact} and \eqref{lim_L2_noncompact},
 proved in Propositions~\ref{limit_u_1} and~\ref{limit_u_2}, respectively.
This completes the proof of Theorem \ref{th3.1v3}.
\end{proof}

\subsection{No uniform decay rate of the solution to CGL on noncompact \texorpdfstring{$\mathcal{M}^3 $ when $\lambda_{\min} = 0$
}{\mathcal{M}\^3 }}

The dynamics of the solutions to the CGL equation are dramatically different on noncompact Riemannian manifolds because solutions spreading at large scales may dissipate very slowly. In sharp contrast with the compact case,
one cannot expect a universal decay rate for solutions to the CGL equation on noncompact Riemannian manifolds when
$\lambda_{\min } = 0$.

Before giving the proof, we first establish the following lemma.

\begin{lemma}\label{prop:non-uniform}
Let $\mathcal{M}^3$ satisfy Assumption~\ref{assump:A} and $\lambda_{\min} = 0$.
Let $0<r\leq 1$ and $1<p\leq 5$.
There exist a sequence of initial data $u_n(0)\in H_x^1 \left(\mathcal{M}^3 \right)$, corresponding global solutions $u_n$ of the equation \eqref{eq1.CGL}, and a sequence $t_n\to+\infty$ such that
\begin{align}
\lim_{n\to\infty} \|u_n(0)\|_{H^1_x  \left(\mathcal{M}^3 \right)} &= r,
 \label{bd:initial}
\intertext{ and }
\liminf_{n\to\infty} \|u_n(t_n)\|_{H^1_x  \left(\mathcal{M}^3 \right)} &\ge \frac{r e^{-\Re z}}{10}.
 \label{bd:final}
\end{align}
In particular, there is no uniform decay rate in the unit ball of $H_x^1\left(\mathcal{M}^3 \right)$.
\end{lemma}

\begin{proof}
By the variational characterization of the spectrum (see \cite[Theorem 10.8]{Gri}),
there exists a sequence $\phi_n \in C_c^\infty \left( \mathcal{M}^3 \right)$ such that
\begin{align*}
\|\phi_n \|_{L^2} = 1,  \text{ and } h_n := \|\nabla_g \phi_n \|_{L^2} \to 0, \text{ as } n \to \infty.
\end{align*}
Let
\begin{align*}
u_n(0,x) = \frac{r }{ \sqrt{ 1 + h_n^2}} \phi_n(x).
\end{align*}
It is clear that
\begin{align}\label{eq5.17v275}
\ \| u_n(0) \|_{L^2} \to r,
\  \|\nabla_g u_n(0) \|_{L^2 \left(\mathcal{M}^3 \right)}\sim rh_n,
\  \| u_n(0) \|_{L^q  \left(\mathcal{M}^3 \right)}\sim rh_n^{3  \left( \frac12 - \frac1q \right)}, \ 2 \le q \le 6.
\end{align}
Then \eqref{bd:initial} holds.
We now turn to \eqref{bd:final}. By \eqref{eq_mass},
\[
\|\nabla_g  u_n (t) \|_{L^2 \left(\mathcal{M}^3 \right)} \lesssim E(
u_n (0) )^{ \frac12}, \ \forall\, t \ge 0,
\]
and the Sobolev inequality yields
\[
\|u_n (t) \|_{L^6 \left(\mathcal{M}^3 \right)} \lesssim E( u_n(0) )^{ \frac12}, \ \forall\, t \ge 0.
\]
Now we fix $t_n=h_n^{- \beta }$, where $0 < \beta < \min  \left( 2, \frac{3(p-1)}2  \right)$. To verify \eqref{bd:final},
it suffices to prove a lower bound for $\|u_n(t_n)\|_{L^2 \left(\mathcal{M}^3 \right)}$.
By the triangle inequality, we have
\begin{align}\label{bd:untn}
\|u_n(t_n)\|_{L^2  \left(\mathcal{M}^3 \right)}
\geq  \left\|e^{t_n z\Delta_g } u_n(0) \right\|_{L^2 \left(\mathcal{M}^3 \right)}
-|z|\int_0^{t_n} \big\| f( u_n(\tau) )  \big\|_{L^2 \left(\mathcal{M}^3 \right)}\,  \mathrm{d} \tau.
\end{align}
We estimate the nonlinear term in three cases.

\textbf{Case 1: $3 \le p \le 5$.}
Theorem \ref{th2.6v107} holds on $\mathcal M^3$,
so $u_n(t)\in H^2 \left(\mathcal M^3 \right)$ for $t>0$.
By Lemma~\ref{lem:monotonicity} and \eqref{eq_mass}, we have
\[
\int_0^\infty\|\Delta_g  u_n(t)\|_{L^2  \left(\mathcal{M}^3 \right)}^2\,\mathrm dt
\le |z|^{-2}\int_0^\infty\|\partial_t u_n(t)\|_{L^2  \left(\mathcal{M}^3 \right) }^2\,\mathrm dt
\le\frac{E( u_n(0) )}{\Re z}\lesssim h_n^2.
\]
Since $p\ge3$, we have $E(u_n(0) )\lesssim h_n^2$ for fixed $r$, and consequently
\[
\sup_{t\ge0}\|u_n(t)\|_{L^6  \left(\mathcal{M}^3 \right) }
\lesssim\sup_{t\ge0}\|\nabla_g  u_n(t)\|_{L^2  \left(\mathcal{M}^3 \right)} \lesssim h_n.
\]
By the Gagliardo-Nirenberg and Sobolev inequalities, we have
\[
\|u_n(t)\|_{L^\infty  \left(\mathcal{M}^3 \right)}
\lesssim\|\nabla_g  u_n(t)\|_{L^2  \left(\mathcal{M}^3 \right)}^{ \frac12}\|\Delta_g  u_n(t)\|_{L^2  \left(\mathcal{M}^3 \right)}^{ \frac12}
\lesssim h_n^{ \frac12}\|\Delta_g  u_n(t)\|_{L^2  \left(\mathcal{M}^3 \right)}^{ \frac12}.
\]
Combining the preceding $L^6$ bound with the interpolation and Gagliardo--Nirenberg inequalities, we obtain
\[
\|f(u_n(t))\|_{L^2_x  \left(\mathcal{M}^3 \right)}
\le\|u_n(t)\|_{L^6  \left(\mathcal{M}^3 \right)}^3\|u_n(t)\|_{L^\infty  \left(\mathcal{M}^3 \right)}^{p-3}
\lesssim h_n^{ \frac{ p+3}2}\|\Delta_g  u_n(t)\|_{L^2  \left(\mathcal{M}^3 \right)}^{ \frac{ p-3}2}.
\]
By H\"older's inequality, we obtain
\begin{align}\label{eq5.22v258}
\begin{aligned}
\int_0^{t_n}\|f(u_n(t))\|_{L^2  \left(\mathcal{M}^3 \right) } \,\mathrm dt
&\lesssim h_n^{ \frac{ p+3}2}t_n^{ \frac{ 7-p}4}
\left(\int_0^{t_n}\|\Delta_g  u_n(t)\|_{L^2  \left(\mathcal{M}^3 \right) }^2\,\mathrm dt\right)^{ \frac{ p-3}4}
\lesssim h_n^pt_n^{ \frac{ 7-p}4}   \to 0, \text{ as } n\to \infty .
\end{aligned}
\end{align}

\textbf{Case 2: $ \frac32 < p < 3$.}
By
 the interpolation inequality,
we have
\[
\| f(u) \|_{L^2 \left(\mathcal{M}^3 \right)}
\lesssim \|u \|_{L^2 \left(\mathcal{M}^3 \right)}^{ \frac{ 3-p}2} \|u \|_{L^6 \left(\mathcal{M}^3 \right)}^{ \frac{ 3(p-1)}2}.
\]
By \eqref{eq_mass}, we have
\[
\|u_n(t)\|_{L^2 \left(\mathcal{M}^3 \right)} \le r \text{ and } \|u_n(t)\|_{L^6 \left(\mathcal{M}^3 \right)} \lesssim h_n^{\gamma_1},
\]
where $\gamma_1 = \min\left\{1, \frac{3(p-1)}{4}\right\} > 0$.
Hence,
\[
\| f(u_n(t)) \|_{L^2  \left(\mathcal{M}^3 \right)} \lesssim h_n^{\frac{3\gamma_1(p-1)}{2}}.
\]
Shrinking $\beta$ if necessary so that $0 < \beta < \frac{3\gamma_1(p-1)}{2}$, we obtain
\begin{align}\label{eq5.23v258}
\int_0^{t_n} \| f(u_n(\tau)) \|_{L^2 \left(\mathcal{M}^3 \right)} \,  \mathrm{d} \tau
\lesssim t_n
 h_n^{\frac{3\gamma_1(p-1)}{2}}
= h_n^{\frac{3\gamma_1(p-1)}{2} - \beta} \to 0 \quad \text{as } n \to \infty.
\end{align}

\textbf{Case 3: $1 < p \le  \frac32$.}
By the interpolation inequality, we get
\[
\| f(u) \|_{L^2 \left(\mathcal{M}^3 \right)}
\lesssim
\|u\|_{L^2 \left(\mathcal{M}^3 \right)}^{3-2p}
\|u\|_{L^3 \left(\mathcal{M}^3 \right)}^{3(p-1)}.
\]
By \eqref{eq_mass} and the Gagliardo--Nirenberg inequality, we obtain
\[
\|u_n(t)\|_{L^2 \left(\mathcal{M}^3 \right)} \le r
\text{ and }
 \|u_n(t)\|_{L^3 \left(\mathcal{M}^3 \right)} \lesssim h_n^{ \frac{ \gamma_1}2},
\]
where $\gamma_1 = \min\left\{1, \frac{3(p-1)}{4}\right\} > 0$.
Hence,
\[
\| f(u_n(t)) \|_{L^2 \left(\mathcal{M}^3 \right)}
\lesssim h_n^{\frac{3\gamma_1(p-1)}2}.
\]
Shrinking $\beta$ if necessary so that
\[
0 < \beta < \frac{3\gamma_1(p-1)}2,
\]
we obtain
\begin{align}\label{eq5.24v258}
\int_0^{t_n} \| f(u_n(\tau)) \|_{L^2 \left(\mathcal{M}^3 \right)} \,  \mathrm{d} \tau
\lesssim t_n h_n^{\frac{3\gamma_1(p-1)}2}
= h_n^{\frac{3\gamma_1(p-1)}2-\beta}
\to 0 \quad \text{as } n \to \infty.
\end{align}
Therefore, for all $1 < p \le 5$, by \eqref{eq5.22v258}, \eqref{eq5.23v258}, and \eqref{eq5.24v258}, we have
\[
\int_0^{t_n} \left\|   f(u_n(\tau))    \right\|_{L^2 \left(\mathcal{M}^3 \right)} \,  \mathrm{d}\tau = o(1) \text{ as } n\to \infty .
\]
Substituting this into \eqref{bd:untn}, we obtain
\begin{align}\label{eq5.22v275}
\|u_n(t_n)\|_{L^2 \left(\mathcal{M}^3 \right)}
\geq \left\|e^{t_n z\Delta_g }
u_n(0) \right\|_{L^2 \left(\mathcal{M}^3 \right)}
- o(1).
\end{align}
It remains to estimate the linear Ginzburg--Landau propagator. Since $\beta<2$, we have
\begin{align*}
\left\|e^{t_n z \Delta_g} u_n(0) - u_n (0) \right\|_{L^2}^2
& = \int_0^\infty  \left| e^{- t_n z\lambda} - 1 \right|^2 \,  \mathrm{d}\mu_{u_n(0)} ( \lambda)
\\
& \lesssim_z t_n \int_0^\infty \lambda \,\mathrm{d} \mu_{u_n(0)} ( \lambda)
= t_n \|\nabla_g u_n(0) \|_{L^2}^2 \lesssim_r t_n h_n^2 = h_n^{2- \beta} \to 0, \text{ as } n \to \infty.
\end{align*}
This together with \eqref{eq5.17v275}
implies $\left\|e^{t_nz\Delta} u_n(0)  \right\|_{L^2 \left(\mathcal{M}^3 \right)}\to r$ as $n \to \infty$.
Combining this with \eqref{bd:untn} and \eqref{eq5.22v275} gives
$\liminf\limits_{ n \to \infty } \|u_n(t_n)\|_{H^1  \left(\mathcal{M}^3 \right)}\ge r$.
This verifies \eqref{bd:final}.
The proof of Lemma \ref{prop:non-uniform} is now complete.
\end{proof}

\begin{proof}[Proof of Proposition \ref{pr1.4v152}]

Assume for contradiction that such a function $\Psi$ exists. Fix $0<r\le1$ as in Lemma~\ref{prop:non-uniform}. Since $\Psi(t)\to0$, there exists $T_*\gg1$ such that
\[
C(r)\Psi(t)\le \frac{r e^{-\Re z}}{100},\qquad t\ge T_*.
\]
By Lemma~\ref{prop:non-uniform}, for all sufficiently large $n$ we have $t_n\ge2T_*$, $\|u_n(0)\|_{H^1 \left(\mathcal M^3 \right)}=r$, and
\[
\|u_n(t_n)\|_{H^1 \left(\mathcal{M}^3 \right)}\ge \frac{r e^{-\Re z}}{20}.
\]
On the other hand, the assumed uniform decay estimate gives
\[
\|u_n(t_n)\|_{H^1 \left(\mathcal M^3 \right)}
\le C(r)\Psi(t_n)
\le \frac{r e^{-\Re z}}{100},
\]
which is a contradiction.
\end{proof}

\subsection{Exponential decay of the solution to CGL on noncompact $\mathcal{M}^3$ when $\lambda_{\min } > 0$
}

By the variational characterization of the bottom of the spectrum (see \cite[Theorem 10.8]{Gri}), we have
\begin{align}\label{eq5.23v275}
\lambda_{\min } \| f\|_{L^2 \left(\mathcal M^3 \right)}^2  \le  \|\nabla_g f \|_{L^2 \left(\mathcal M^3 \right)}^2,
\qquad  f \in H^1 \left(\mathcal M^3 \right).
\end{align}

\begin{proof}[Proof of Proposition \ref{pr1.5v275}]

From the equation \eqref{eq1.CGL}, we have
\[
\|\nabla_g u\|_{L^2}^2 + \|u\|_{L^{p+1}}^{p+1} = -\Re\left(  z^{-1} \partial_tu,u \right)_{L^2} .
\]
Therefore, by the Cauchy--Schwarz inequality and \eqref{eq5.23v275},
\[
\|\nabla_g u\|_{L^2}^2 + \|u\|_{L^{p+1}}^{p+1}
\le \frac1{|z|} \|\partial_tu\|_{L^2}\|u\|_{L^2}
\le \frac1{|z|\sqrt{\lambda_{\min } }} \|\partial_tu\|_{L^2}\|\nabla_g u\|_{L^2}.
\]
By Young's inequality,
\[
\|\partial_tu\|_{L^2}^2 \ge |z|^2\lambda_{\min }  \left( \|\nabla_g u\|_{L^2}^2 +  \|u\|_{L^{p+1}}^{p+1} \right).
\]
This together with \eqref{eq_mass} implies
\[
\frac{ \mathrm{d} }{ \mathrm{d} t}E(u(t)) \le -2 \Re z \lambda_{\min }  E(u(t)).
\]
By Gr\"onwall's inequality
and \eqref{eq5.23v275}, we obtain
\[
\|u(t)\|_{H^1 \left(\mathcal M^3 \right)} \le \left(2 \left(1+\lambda_{\min}^{-1} \right)E(u_0)\right)^{ \frac12}  e^{- \Re z \lambda_{\min } t}.
\]
\end{proof}

\noindent \textbf{Acknowledgments.}
We would like to thank Professors ~Philippe Souplet, Lixin Yan, and Hatem Zaag for helpful discussions.

\end{document}